\documentclass[11pt]{amsart}

\usepackage{amsmath}
\usepackage{amssymb}
\usepackage[normalem]{ulem}
\usepackage[all]{xy}
\usepackage{longtable}
\usepackage{titletoc}
\usepackage{mathrsfs}
\usepackage{mathtools}
\usepackage{amscd}
\usepackage{appendix}
\usepackage{amsthm}
\usepackage{amsfonts}
\usepackage{tikz-cd}
\usepackage{array}
\usepackage{graphicx}
\usepackage{hyperref}
\usepackage{longtable}
\usepackage{xtab}
\usepackage{soul,xcolor}
\usepackage{upgreek}
\usepackage[total={15.875truecm,22.86truecm},centering]{geometry}
\usepackage{stackrel}
\usepackage{eucal}
\usepackage{bbm}
\usepackage[outline]{contour}
\usepackage{scalerel}
\usepackage{color}

\SelectTips{cm}{}

\newtheorem{thm}[equation]{Theorem}
\newtheorem{lem}[equation]{Lemma}
\newtheorem{prop}[equation]{Proposition}
\newtheorem{cor}[equation]{Corollary}
\newtheorem{conj}[equation]{Conjecture}

\theoremstyle{definition}
\newtheorem{defn}[equation]{Definition}
\newtheorem{ex}[equation]{Example}

\theoremstyle{remark}
\newtheorem{rem}[equation]{Remark}
\newtheorem{Subsubsec}[equation]{}

\numberwithin{equation}{subsection}

\DeclareMathAlphabet{\mathpzc}{OT1}{pzc}{m}{n}
\DeclareMathAlphabet{\matheur}{U}{eur}{m}{n}

\definecolor{lava}{RGB}{207,16,32}
\definecolor{purple}{RGB}{148,0,211}

\newcommand{\FF}{\mathbb{F}}
\newcommand{\ZZ}{\mathbb{Z}}
\newcommand{\QQ}{\mathbb{Q}}

\newcommand{\GG}{\mathbb{G}}

\newcommand{\CC}{\mathbb{C}}

\newcommand{\PP}{\mathbb{P}}

\newcommand{\NN}{\mathbb{N}}

\newcommand{\boxx}{\boldsymbol{x}}

\newcommand{\afk}{\mathfrak{a}}

\newcommand{\cfk}{\mathfrak{c}}
\newcommand{\dfk}{\mathfrak{d}}

\newcommand{\mfk}{\mathfrak{m}}
\newcommand{\nfk}{\mathfrak{n}}

\newcommand{\pfk}{\mathfrak{p}}

\newcommand{\Pfk}{\mathfrak{P}}

\newcommand{\Gcal}{\CMcal{G}}

\newcommand{\Mcal}{\CMcal{M}}

\newcommand{\Ocal}{\CMcal{O}}
\newcommand{\Pcal}{\CMcal{P}}
\newcommand{\Rcal}{\CMcal{R}}

\newcommand{\Dcal}{\CMcal{D}}

\newcommand{\Wcal}{\CMcal{W}}

\newcommand{\Ascr}{\mathscr{A}}

\newcommand{\Dscr}{\mathscr{D}}

\newcommand{\Uscr}{\mathscr{U}}
\newcommand{\Vscr}{\mathscr{V}}
\newcommand{\Wscr}{\mathscr{W}}

\newcommand{\eT}{\matheur{T}}

\DeclareMathOperator{\GL}{GL}
\DeclareMathOperator{\Mat}{Mat}

\DeclareMathOperator{\End}{End}

\DeclareMathOperator{\cris}{cris}
\DeclareMathOperator{\dR}{dR}

\DeclareMathOperator{\Gal}{Gal}
\DeclareMathOperator{\Hom}{Hom}

\DeclareMathOperator{\trdeg}{tr.deg}
\DeclareMathOperator{\wt}{wt}

\DeclareMathOperator{\Inf}{Inf}
\DeclareMathOperator{\Spec}{Spec}
\DeclareMathOperator{\Div}{Div}

\DeclareMathOperator{\rank}{rank}
\DeclareMathOperator{\ord}{ord}

\newcommand{\ok}{\bar{k}}

\newcommand{\tr}{\mathrm{tr}}

\makeatletter
\DeclareFontEncoding{LS2}{}{\@noaccents}
\makeatother
\DeclareFontSubstitution{LS2}{stix}{m}{n}

\DeclareSymbolFont{largesymbolsstix}{LS2}{stixex}{m}{n}

\DeclareMathDelimiter{\lbrbrak}{\mathopen}{largesymbolsstix}{"EE}{largesymbolsstix}{"14}
\DeclareMathDelimiter{\rbrbrak}{\mathclose}{largesymbolsstix}{"EF}{largesymbolsstix}{"15}

\newcommand{\lbangle}{\lbrbrak}
\newcommand{\rbangle}{\rbrbrak}

\DeclareSymbolFont{stixletters}{LS2}{stix}{m}{it}

\DeclareMathSymbol{\vartau}{\mathord}{stixletters}{"1C}

\def\rank{\operatorname{rank}}

\newcommand{\id}{\operatorname{id}}

\newcommand{\geo}{\operatorname{geo}}

\newcommand{\divv}{\operatorname{div}}

\newcommand{\Pic}{\operatorname{Pic}}

\newcommand{\re}{\operatorname{Re}}

\newcommand\subfrac[2]{\genfrac{}{}{0pt}{}{#1}{#2}}

\newcommand{\assign}{\mathrel{\vcenter{\baselineskip0.5ex \lineskiplimit0pt
                     \hbox{\scriptsize.}\hbox{\scriptsize.}}}%
                     =}
\newcommand{\rassign}{=%
                     \mathrel{\vcenter{\baselineskip0.5ex \lineskiplimit0pt
                     \hbox{\scriptsize.}\hbox{\scriptsize.}}}%
                     }

\DeclareRobustCommand\longtwoheadrightarrow
{\relbar\joinrel\twoheadrightarrow}
\newcommand\mapsfrom{\mathrel{\reflectbox{\ensuremath{\mapsto}}}}
\newcommand{\rmx}{{\rm x}}

\newcommand{\uone}{\underline{1}}
\newcommand{\ua}{\underline{a}}
\newcommand{\ub}{\underline{b}}

\newcommand{\ualpha}{\underline{\alpha}}

\newcommand{\upi}{\underline{\pi}}
\newcommand{\uinfty}{\underline{\infty}}

\newcommand{\uPP}{\underline{\PP}}

\newcommand{\uA}{\underline{A}}
\newcommand{\uC}{\underline{C}}
\newcommand{\uE}{\underline{E}}
\newcommand{\uF}{\underline{F}}

\newcommand{\uK}{\underline{K}}
\newcommand{\uO}{\underline{O}}

\newcommand{\uU}{\underline{U}}
\newcommand{\uX}{\underline{X}}
\newcommand{\uk}{\underline{k}}
\newcommand{\uv}{\underline{v}}

\newcommand{\uafk}{\underline{\afk}}

\newcommand{\ucfk}{\underline{\cfk}}
\newcommand{\udfk}{\underline{\dfk}}

\newcommand{\umfk}{\underline{\mfk}}
\newcommand{\unfk}{\underline{\nfk}}

\hypersetup{
	colorlinks=true,
	linkcolor=blue,
	citecolor=red,
	linktoc=page,
	linkbordercolor=blue,
	citebordercolor=red,
}

\begin{document}

\title{$v$-adic Chowla--Selberg formula over function fields}
\author[Fu-Tsun Wei]{Fu-Tsun Wei}
\address{Department of Mathematics, National Tsing Hua Univeristy, Taiwan}
\email{ftwei@math.nthu.edu.tw}


\thanks{
The author is supported by NSTC grant (no.~114-2628-M-007-003, 115-2628-M-007-008)
and the National Center for Theoretical Sciences.
}

\date{\today}

\begin{abstract}
Let $v$ be a finite place of a rational function field with finite constant field.
We establish an Ogus-type $v$-adic Chowla–Selberg formula for $t$-motives with ``complex multiplication'' by subfields of Carlitz cyclotomic function fields. 
For this purpose, we first introduce the $v$-adic crystalline action of the Weil group on the de~Rham spaces of $t$-motives via the de~Rham--crystalline comparison isomorphism of Hartl--Kim, and use it to give a ``$v$-isocrystal'' period interpretation of special values of Thakur's $v$-adic geometric gamma function.
Our Ogus-type formula is then based on the isogeny theorem for CM $t$-motives, together with a recipe for constructing ``abelian CM types'' given in the $\infty$-adic setting over function fields. 

\end{abstract}

\maketitle

\section{Introduction}\label{Intro}

In the classical setting, the celebrated Chowla–Selberg formula (see \cite{CS67}) expresses, up to algebraic multiples, the periods of CM elliptic curves in terms of suitable twisted products of special values of Euler’s gamma function.
Gross's geometric proof of this formula \cite{Gr78} (see also \cite{Gr21}) revealed that these gamma values arise naturally from the theory of complex multiplication, and initiated a broad program of understanding the intrinsic relation between special gamma values and periods of CM abelian varieties (see \cite{And82,Co93,Yo03,Yang10,BM16}).

The $p$-adic counterpart began with Morita's $p$-adic gamma function, which interpolates factorials $p$-adically. The theory took a decisive step forward with the work of Gross and Koblitz \cite{GrKo79}, who reinterpreted Katz’s $p$-adic limit formula for Gauss sums as Frobenius-twisted products of special values of Morita's $p$-adic gamma function.
Subsequently, Ogus \cite{Og90} established a $p$-adic Chowla–Selberg formula, describing the crystalline action of the Weil group on the de~Rham realizations of CM elliptic curves in terms of twisted products of $p$-adic special gamma values. This formula provides an explicit description of $p$-adic CM periods in the ``abelian case'' and has influenced subsequent developments in explicit $p$-adic period formulas and the study of arithmetic invariants arising from complex multiplication (see \cite{KaYo08,KaYo09,AnFr25,Co26,AsHa26}).  

The present paper continues the author's study of ``abelian CM periods'' over function fields from a $v$-adic perspective. While \cite{Wei26} established an $\infty$-adic Chowla–Selberg formula for CM $t$-motives and \cite{CWY26} gave a $v$-adic period interpretation of arithmetic special gamma values, the present work develops the geometric side by identifying Thakur's $v$-adic geometric special gamma values as crystalline periods of soliton $t$-motives. This new period interpretation serves as the key ingredient in our proof of an Ogus-type $v$-adic Chowla--Selberg formula, describing the crystalline action of the Weil group on the de~Rham spaces of CM $t$-motives in terms of special values of Thakur's $v$-adic geometric gamma function.

\subsection{\texorpdfstring{$v$}{v}-adic geometric gamma function}

Let $A=\FF_q[\theta]$, the polynomial ring with one variable $\theta$ over a finite field $\FF_q$ with $q$ elements, and $k=\FF_q(\theta)$ be the fraction field of $A$. Fix a monic irreducible polynomial $v$ in $A$.
Let $k_v$ be the completion of $k$ at $v$ and $A_v$ be the closure of $A$ in $k_v$.
For each $a \in A_v$, define $a_{\flat} \assign a$ if $v\nmid a$ and $1$ otherwise.
Thakur \cite[Section 5.9]{Tha91} introduced the following Morita-type $v$-adic geometric gamma function:
\[
\Gamma_v^{\geo}(x) \assign x_{\flat}^{-1} \cdot \prod_{\text{monic } \afk \in A}
\left(
\frac{\afk_{\flat}}{(x+\afk)_{\flat}}
\right) \ \
\in k_v^{\times}, \quad \forall x \in A_v.
\]

The function $\Gamma_v^{\geo}(x)$ satisfies natural analogues of the classical functional equations, including translation, reflection, and multiplication formulas (see Section~\ref{sec: v-gamma}). These identities give rise to monomial relations among the special values $\Gamma_v^{\geo}(x)$ for $x\in k\cap A_v$. Moreover, the Gross--Koblitz--Thakur formula established by Chang \cite{Ch25} relates suitable Frobenius-twisted products of these special values to ``geometric Gauss sums'' (see Theorem~\ref{thm: G-K-T-formula-geo}), thereby showing further arithmetic significance of the $v$-adic geometric gamma values.

Beyond these arithmetic properties, the present paper reveals a new conceptual role for the normalized $v$-adic geometric special gamma values by identifying them as crystalline periods of soliton $t$-motives.

\subsection{Crystalline action and the main theorem}

To formulate this period interpretation and the resulting Chowla--Selberg formula, we first recall the crystalline action associated with the de~Rham--crystalline comparison isomorphism.
Let $\uA=\FF_q[t]$, where $t$ is  another variable over $\FF_q$.
Given a $t$-motive $M$ over a finite extension $F$ of $k_v$ (see Definition~\ref{defn: t-mot}), suppose $M$ has good reduction (see Definition~\ref{defn: model}).
By the work of Hartl--Kim \cite{HK20}, there exists a functorial de~Rham--crystalline comparison isomorphism (see Proposition~\ref{prop: c-d-iso})
\[
\Phi_M:H_{\dR}(M,F)\xlongrightarrow{\sim}H_{\cris}(M,F).
\]
The crystalline realization of $M$ naturally carries a ``$v$-isocrystal'' action (induced from the $\tau^{\deg v}$-action on the reduction of $M$, see \eqref{eqn: t-sigma-cris}). 
Transporting this action through $\Phi_M$, we obtain a semilinear crystalline action $\sigma_{\cris}$, for each $\sigma$ in the Weil group $\Wcal_v$ (recalled in the beginning of Section~\ref{sec: cris-act}), on the de~Rham space $H_{\dR}(M,\bar{k}_v)$, where $\bar{k}_v$ is a chosen algebraic closure of $k_v$ (see \eqref{eqn: sigma-cris}).

Choose an ordered basis $\beta = \{\omega_1,...,\omega_r\}$ for $H_{\dR}(M,\bar{k}_v)$, where $r = \dim_{\bar{k}_v}\big(H_{\dR}(M,\bar{k}_v)\big)$.
For every $\sigma \in \Wcal_v$, we define the \emph{$v$-isocrystal period matrix of $\sigma$ on $M$} (with respect to $\beta$) to be the matrix $\Pcal^M_{\beta}(\sigma) \in \Mat_r(\bar{k}_v)$ satisfying (see \eqref{eqn: v-PM})
\[
\begin{pmatrix}
    \sigma_{\cris}(\omega_1) \\
    \vdots \\
    \sigma_{\cris}(\omega_r)
\end{pmatrix}
= \Pcal^M_{\beta}(\sigma)^{\rm t}
\cdot \begin{pmatrix}
    \omega_1 \\
    \vdots \\
    \omega_r
\end{pmatrix}.
\]
Here $\Pcal^M_{\beta}(\sigma)^{\rm t}$ denotes the transpose of $\Pcal^M_{\beta}(\sigma)$.
These matrices encode the crystalline action on the de~Rham realization of $M$ and thus play the role of $v$-adic period matrices.
In particular, suppose that $M$ is  defined over the algebraic closure $\bar{k}$ of $k$ in $\bar{k}_v$ and that the chosen basis $\beta$ is algebraic, i.e.~$\beta \subset H_{\dR}(M,\bar{k})$.
Then $\Pcal^M_{\beta}(\sigma)$ can be obtained from the ``de~Rham--crystalline comparison period matrix'' of $M$.
The transcendental aspects will be further discussed  in Section~\ref{sec: Trans}.

With this period-theoretic framework in place, we now introduce the notation needed to state our main theorem.
Given a monic $\unfk \in \uA$ with $\uv \assign v(t) \nmid \unfk$, let $\uK_{\unfk}$ (resp.~$K_{\unfk}$) be the $\unfk$-th Carlitz cyclotomic function field over $\uk=\FF_q(t)$ (resp.~$k$).
Let $M$ be a CM $t$-motive with CM type $(\uK,\Xi)$ over the algebraic closure $\bar{k}$ of $k$ (see Definition~\ref{defn: CM t-motive}), where $\uK$ is a subextension of $\uK_{\unfk}$ over $\uk$.
Let $G_{\uK} \assign \Gal(\uK/\uk)$, and denote the set of $\FF_q$-algebra embeddings from $\uK$ into $\bar{k}$ by $J_{\uK} = \{\xi_{\varrho}\mid \varrho \in G_{\uK}\}$ as in \eqref{eqn: xi-rho}.
Write $\Xi = \sum_{\varrho \in G_{\uK}} m_{\varrho} \xi_{\varrho}$ with $m_{\varrho} \in \ZZ_{\geq 0}$.
Our main theorem is stated as follows:

\begin{thm}\label{thm: MT1}
{\rm (See Theorem~\ref{thm: v-CS}.)} Keep the above notation. 
Then we can find a basis $\{\omega_{\varrho}^M\mid \varrho \in G_{\uK}\}$ for $H_{\dR}(M,\bar{k})$ consisting of ``eigen-differentials'' with respect to the $\uK$-action, i.e. 
\[
\ualpha_{\dR}(\omega_{\varrho}^M) = \xi_{\varrho}(\ualpha)\cdot \omega_{\varrho}^M, \quad \forall \ualpha \in \uO_{\uK} \ \text{ (the integral closure of $\uA$ in $\uK$)},
\]
so that for every $\sigma_v \in \Wcal_v$ with $\deg \sigma_v =1$ and $\varrho_0 \in G_{\uK}$, we have
\[
\sigma_{v,\cris}(\omega^{M}_{\varrho_0}) = 
\zeta_{\varrho_0} \cdot v^{\frac{\wt(\Xi)}{[\uK:\uK^+]}} \cdot \prod_{\varrho \in G_{\uK}} \left(
\prod_{1\neq \ucfk \mid \unfk} \ \prod_{\subfrac{\ua \in \uA,\, \deg \ua < \deg \ucfk}{\gcd(\ua,\ucfk)=1}}\Gamma_v^{\geo,*}\Big\lbangle \frac{va}{\cfk}\Big\rbangle^{-n_{\ucfk}(\varrho,\ua)}
\right)^{\frac{m_{\varrho\varrho_0}}{[\uK:\uk]}} \cdot \omega^{M}_{\varrho_{\uv}\varrho_0}.
\]
Here:
\begin{itemize}
    \item $\zeta_{\varrho_0} \in \overline{\FF}_q^{\times}$;
    \item $\wt(\Xi)$ is the weight of $\Xi$, see Definition~\ref{defn: CM type}; 
    \item 
    $\uK^+$ is the maximal totally real subfield of $\uK$;
    \item $\Gamma_v^{\geo,*}\lbangle x \rbangle$, for $x \in k\cap A_v$, is the normalized $v$-adic geometric gamma value introduced in \eqref{eqn: normal-v-gamma};
    \item $\cfk = \ucfk(\theta)$, $a = \ua(\theta) \in A$;
    \item $n_{\ucfk}(\varrho,\ua) \in \QQ$ is defined in \eqref{eqn: n-c-rho-a};
    
    \item $\varrho_{\uv}\varrho_0\assign \varrho_{\uv}\big|_{\uK}\circ \varrho_0 \in G_{\uK}$ (where $\varrho_{\uv} \in \Gal(\uK_{\unfk}/\uk)$ is the Artin symbol given in \eqref{eqn: ASrho}).
\end{itemize}
\end{thm}

In particular, modulo multiplication by an element of \(\overline{\FF}_q^\times\), the $v$-isocrystal periods of an ``abelian'' CM $t$-motive can be expressed explicitly as a product of normalized $v$-adic geometric special gamma values.

As an illustration, consider the case where the above CM $t$-motive $M$ arises from a rank $2$ CM Drinfeld module over $\bar{k}$.
In this case, $q$ is odd (see Remark~\ref{rem: CM-quad}) and there exists a monic squarefree $\udfk \in \uA$ such that $\uK = \uk(\sqrt{-\udfk})\subset \uK_{\udfk}$.
The main theorem then specializes to the following Ogus-type $v$-adic Chowla--Selberg formula:

\begin{thm}\label{thm: MT2}
{\rm (See Theorem~\ref{thm: v-CS-quad}.)}
Keep the above notation. Suppose $\uv \nmid \udfk$. Let $\uE_{\rho}$ be a CM Drinfeld module of rank $2$ over $\bar{k}$ with $\End(\uE_{\rho})\cong \uO_{\uK}$, and let $M_{\rho}$ be the $t$-motive associated with $\uE_{\rho}$ over $\bar{k}$. 
There exists an eigenbasis $\{\omega^{\rho}_{+},\omega^{\rho}_{-}\}$ for $H_{\dR}(M_{\rho},\bar{k})$ (with respect to the $\uK$-action) such that for every $\sigma_v \in \Wcal_v$ with $\deg \sigma_v = 1$, we have
\[
\sigma_{v,\cris}(\omega_{\pm }^{\rho}) = \zeta_{\pm } \cdot v^{\frac{1}{2}} \cdot \left(\prod_{\subfrac{\ua \in \uA,\, \deg \ua < \deg \udfk}{\gcd(\ua,\udfk)=1}}
\Gamma_v^{\geo,*}\Big\lbangle\frac{va}{\dfk}\Big\rbangle^{\mp \frac{\chi_{\uK}(\ua)}{2h_{\uK}}}\right)
\cdot \omega_{\pm \chi_{\uK}(\uv)}^{\rho}.
\]
Here $\zeta_{\pm} \in \overline{\FF}_q^{\times}$, $h_{\uK}= \#\Pic(\uO_{\uK})$, and $\chi_{\uK}:G_{\uK}\rightarrow \{\pm 1\}$ is the quadratic character associated with $\uK/\uk$.
\end{thm}

\begin{rem}
${}$
\begin{itemize}
\item[(1)] 
In the $\infty$-adic setting, take a nonzero cycle $\gamma$ in the Betti module of $M_{\rho}$.
The Chowla--Selberg formula established in \cite[equation (6.3.12)]{Wei26} (see also \cite[Theorem 2.3.1]{Ch26}) says that
\[
\int_{\gamma} \omega^{\rho}_{\pm }
\sim \tilde{\pi}^{\frac{1}{2}}
\cdot 
\left(\prod_{\subfrac{\ua \in \uA,\, \deg \ua < \deg \udfk}{\gcd(\ua,\udfk)=1}}
\Gamma_{\infty}^{\geo}(\frac{a}{\dfk})^{\pm \frac{\chi_{\uK}(\ua)}{2h_{\uK}}}\right).
\]
Here $\alpha \sim \beta$ means that  $\alpha/\beta \in \bar{k}^{\times}$, $\tilde{\pi}$ is the ``Carlitz period''``1iop, and $\Gamma_{\infty}^{\geo}$ is the ``$\infty$-adic geometric gamma function'' (see \cite[Definition 5.3.1]{Tha91}).
Theorem~\ref{thm: v-CS-quad} can be regarded as the $v$-adic analogue of the above $\infty$-adic Chowla--Selberg formula, where the Betti period integral is replaced by the crystalline action.
\item[(2)] We expect that the framework developed in this paper can be extended to the ``two-variable'' $v$-adic gamma values introduced by Goss (see \cite[Section~9.9]{Goss96}). The geometric case considered here already captures the essential period-theoretic picture, while allowing us to avoid the additional notation and technicalities required in the two-variable setting. We therefore restrict ourselves to the geometric case in the present paper and leave the more general formulation to future work.
\end{itemize}
\end{rem}

The key ingredient underlying these formulas is the explicit realization of $v$-adic geometric special gamma values as $v$-isocrystal periods of soliton $t$-motives, which we now describe.

\subsection{Period interpretation of \texorpdfstring{$v$}{v}-adic geometric gamma values}

Let $\unfk\in\uA$ be monic, and put $\nfk=\unfk(\theta)\in A$.
For each $x\in\nfk^{-1}A\setminus A$, Anderson~\cite{A92} introduced the so-called \emph{Coleman function} $g_x$ (recalled in Theorem~\ref{thm: Coleman-fun}).
This function is the essential ingredient in the construction of the \emph{soliton $t$-motive $M_x$ associated with $x$}, introduced by Sinha~\cite{Sinha97} (see Definition~\ref{defn: Soliton M-x}) in his study of the transcendence of $\infty$-adic geometric gamma values.
Subsequently, Brownawell--Papanikolas~\cite{BP02} and Anderson--Brownawell--Papanikolas~\cite{ABP04} developed the dual $t$-motive version of this construction and proved an analogue of the Lang--Rohrlich conjecture, thereby completely determining the algebraic relations among $\infty$-adic geometric gamma values.

Note that $M_x$ is defined over $K_{\unfk}$. Suppose $v \nmid \nfk$.
Let $w$ be the place of $K_{\unfk}$ corresponding to the valuation on $\bar{k}_v$ (uniquely extended from $v$) restricting to $K_{\unfk}$.
Then $M_x$ admits a good model $\Mcal_{x,w}$ at the place $w$ above $v$ (see Remark~\ref{rem: prop-Mx}~(3)).
Identifying $\Gal(\uK_{\unfk}/\uk)$ with $(\uA/\unfk)^{\times}$ through the Artin map, we may write $J_{\uK_{\unfk}} = \{\xi_{\ua}\mid \ua \in (\uA/\unfk)^{\times}\}$ (where $\xi_{\ua}$ is given in \eqref{eqn: xi-a}).
From an explicit description of the comparison isomorphism $\Phi_{M_x,w}$ given in Proposition~\ref{prop: sp-iso-x}, we derive that:

\begin{thm}\label{thm: MT3}
{\rm (See Theorem~\ref{thm: period-vgamma-x}.)}
Keep the above notation.
There exists an eigenbasis $\{\omega^x_{\ua}\mid \ua \in (\uA/\unfk)^{\times}\}$ for $H_{\dR}(M_x,K_{\unfk})$ with respect to the $\uK_{\unfk}$-action
satisfying that for every $\sigma_v \in \Wcal_v$ with $\deg \sigma_v = 1$, we have
\[
\sigma_{v,\cris}(\omega^{x}_{\ua}) = \Gamma_v^{\geo,*}\lbangle vax\rbangle^{-1} \cdot \omega^{x}_{\uv \ua}, \quad \forall \ua \in (\uA/\unfk)^{\times}\quad (\text{here $a = \ua(\theta)$}).
\]
\end{thm}

Thus the normalized $v$-adic geometric special gamma values appear directly as $v$-isocrystal periods of soliton $t$-motives.

In particular, let $\ell$ be the order of $\uv \bmod \unfk \in (\uA/\unfk)^{\times}$.
The Gross--Koblitz--Thakur formula established by Chang~\cite{Ch25} then gives (see Theorem~\ref{thm: G-K-T-formula-geo})
\[
\prod_{i=0}^{\ell-1}\Gamma_v^{\geo,*}\lbangle v^ix\rbangle^{-1} = G_{\ell}^{\geo}(x),
\]
where $G_{\ell}^{\geo}(x)$ is the so-called \emph{geometric Gauss sum} (which lies in $\bar{k} \cap k_v^{\times}$).
As a consequence, we obtain the following period interpretation of geometric Gauss sums (see Remark~\ref{rem: period-Gauss-vmidn}):

\begin{cor}\label{cor: G-period}
Let $M_{\boxx_v}$ be the general soliton $t$-motive associated with the formal sum $\boxx_v\assign \sum_{i=0}^{\ell-1}[v^{i}x]$ (constructed in Section~\ref{sec: M-boxx}).
There exists an eigenbasis $\{\omega_{\ua}^{\boxx_v}\mid \ua \in (\uA/\unfk)^{\times}\}$ for $H_{\dR}(M_{\boxx_v},K_{\unfk})$ with respect to the $\uK_{\unfk}$-action satisfying that for every $\sigma_v \in \Wcal_v$ with $\deg \sigma_v=1$, we have 
\[
\sigma_{v,\cris}(\omega_{\ua}^{\boxx_v}) = G_{\ell}^{\geo}(ax) \cdot \omega_{\uv \ua}^{\boxx_v}, \quad \forall \ua \in (\uA/\unfk)^{\times}\quad (\text{here $a = \ua(\theta)$}).
\]  
\end{cor}

\subsection{Proof strategy for the main theorem}

With the period interpretation of $v$-adic geometric special gamma values established, our main theorem follows by combining the theory of complex multiplication with the isogeny theorem proved in \cite{BCPW22}.
We briefly outline the argument below and refer the reader to Section~\ref{sec: v-CS} for the details.

Let $\unfk \in \uA$ be monic with $\uv \nmid \unfk$ as before.
For each $x \in \nfk^{-1}A\setminus A$, the soliton $t$-motive $M_x$ is actually a CM $t$-motive with CM type $(\uK_{\unfk},\Xi_x)$ over $K_{\unfk}$ (see Example~\ref{ex: CM-t-mot}~(2)).
Given a CM $t$-motive $M$ with CM type $(\uK,\Xi)$ over $\bar{k}$, where $\uK$ is a subextension of $\uK_{\unfk}$ over $\uk$, the \emph{inflation} $\Inf_{\uK}^{\uK_{\unfk}}(M)$ is a CM $t$-motive with CM type $(\uK_{\unfk},\Inf_{\uK}^{\uK_{\unfk}}(\Xi))$ (see Proposition~\ref{prop: CM}~(2)).
The recipe developed in \cite[Section 6.3]{Wei26} provides nonnegative integers $h$, $n$, and $n'_{\ucfk}(\varrho,\ua)$ (given in \eqref{eqn: h'-x-xi-ell}, \eqref{eqn: CM-type-n} and \eqref{eqn: m-c}) so that the following equality of generalized CM types holds:
\begin{equation}\label{eqn: eq-CM-type}
h\Inf_{\uK}^{\uK_{\unfk}}(\Xi) + n \Inf_{\uk}^{\uK_{\unfk}}(\xi_0)
= \sum_{\varrho \in G_{\uK}} \sum_{1\neq \ucfk|\unfk}
\sum_{\subfrac{\ua \in \uA,\, \deg \ua< \deg \ucfk}{\gcd(\ua,\ucfk)=1}}n'_{\ucfk}(\varrho,\ua)\Xi_{a/\cfk}.
\end{equation}
Here $\xi_0:\uk \stackrel{\sim}{\rightarrow} k\subset \bar{k}$ is the $\FF_q$-algebra embedding sending $t$ to $\theta$ (and $(\uk,\xi_0)$ is the CM type of the \emph{Carlitz $t$-motive} $M_C$, see Example~\ref{ex: CM-t-mot}~(1)).
Let $M_{\boxx_{\Xi,n}}$ be the generalized soliton $t$-motive associated with the formal sum 
\[
\boxx_{\Xi,n}\assign \sum_{\varrho \in G_{\uK}}\sum_{1\neq \ucfk|\unfk}
\sum_{\subfrac{\ua \in \uA,\, \deg \ua< \deg \ucfk}{\gcd(\ua,\ucfk)=1}}n'_{\ucfk}(\varrho,\ua)[a/\cfk],
\]
and denote $M^{\otimes_{\uK}h}(n)$ to be the  \emph{$n$-th Tate twist} (see \eqref{eqn: Tate-twist}) of the tensor power $M^{\otimes_{\uK}h}$ (given in \eqref{eqn: M-tensor-power}). 
Then \eqref{eqn: eq-CM-type} says that $M_{\boxx_{\Xi,n}}$ and $\Inf_{\uK}^{\uK_{\unfk}}(M^{\otimes_{\uK}h}(n))$ have the same CM type.
By the isogeny theorem stated in Theorem~\ref{thm: isog-thm}, we obtain a $\uK_{\unfk}$-isogeny $f: M_{\boxx_{\Xi,n}} \rightarrow \Inf_{\uK}^{\uK_{\unfk}}(M^{\otimes_{\uK}h}(n))$ over $\bar{k}$, which induces a $\uK_{\unfk}$-isomorphism 
\[
f_{\dR}:H_{\dR}(M_{\boxx_{\Xi,n}},\bar{k}_v)\stackrel{\sim}{\rightarrow} H_{\dR}\big(\Inf_{\uK}^{\uK_{\unfk}}(M^{\otimes_{\uK}h}(n)),\bar{k}_v\big).
\]

Now, for each $\varrho_0 \in G_{\uK}$, choose $\ua_0 \in (\uA/\unfk)^{\times}$ so that $\varrho_{\ua_0}\big|_{\uK} = \varrho_0$.
Composing with the \emph{trace morphism} $\tr_{\uK_{\unfk}/\uK}: \Inf_{\uK}^{\uK_{\unfk}}(M^{\otimes_{\uK}h}(n)) \rightarrow M^{\otimes_{\uK}h}(n)$ (given in \eqref{eqn: trace-morphism}), we show that there exists an eigen-differential $\omega^M_{\varrho_0} \in H_{\dR}(M,\bar{k})$, unique up to multiplication by roots of unity, such that
\[
(\tr_{\uK_{\unfk}/\uK}\circ f)_{\dR}(\omega_{\ua_0}^{\boxx_{\Xi,n}}) = (\omega_{\varrho_0}^M)^{\otimes_{\uK}h} \otimes \omega_C^{\otimes n} \quad \in H_{\dR}(M^{\otimes_{\uK}h}(n),\bar{k}).
\]
Here $\omega_C \in H_{\dR}(M_C,k)$ is the differential given in Example~\ref{sec: Ex-Car-t}.
Finally, the morphism
$(\tr_{\uK_{\unfk}/\uK}\circ f)_{\dR}: H_{\dR}(M_{\boxx_{\Xi,n}},\bar{k}_v)\rightarrow H_{\dR}\big(M^{\otimes_{\uK}h}(n),\bar{k}_v\big)$
is equivariant with respect to the crystalline action (see \eqref{eqn: equivariance}).
Combining this equivariance with the explicit crystalline-action formula for soliton $t$-motives obtained from Theorem~\ref{thm: MT3}, and then taking the $h$-th root, yields the formula stated in Theorem~\ref{thm: MT1}.

\subsection{Transcendental aspects}\label{sec: Trans}

The preceding formulas naturally lead to questions concerning the algebraic relations among $v$-isocrystal periods and $v$-adic special gamma values.
Let $M$ be a $t$-motive over $\bar{k}\subset\bar{k}_v$.
Then there exists a finite extension $F_0$ of $k$ in $\bar{k}$ such that $M$ is obtained from a $t$-motive $M_0$ over $F_0$ by base change.
Let $F$ be the compositum of $F_0$ and $k_v$ inside $\bar{k}_v$, which is a finite extension of $k_v$.
We may therefore regard $M$ as a $t$-motive over $F$.

Assume that $M$ has good reduction over $F$, and let $\beta$ be an $F_0$-basis of $H_{\dR}(M_0,F_0)$.
In Remark~\ref{rem: c-d-period-matrix}, we introduce the \emph{de~Rham--crystalline comparison period matrix}
$[\Phi_M]_{\beta}$ of $M$ with respect to $\beta$.
Its entries are expected to involve transcendental elements over $k$ (see Example~\ref{sec: Ex-Car-t} and Remark~\ref{rem: RC}).
Furthermore, by Remark~\ref{rem: matrix P-Phi}, the associated $v$-isocrystal period matrix
$\Pcal_{\beta}^M(\sigma)$, for $\sigma\in\Wcal_v$, is obtained by semilinear conjugation of the algebraic crystalline-action matrix with the comparison period matrix $[\Phi_M]_{\beta}$.
This provides a conceptual explanation for viewing the entries of $\Pcal_{\beta}^M(\sigma)$ as $v$-adic periods.
Consequently, the study of algebraic relations among $v$-adic geometric special gamma values naturally fits into the broader framework of the transcendence theory of de~Rham--crystalline comparison periods.

Note that a $v$-adic Lang--Rohrlich conjecture for arithmetic gammas was derived in \cite{CWY26} by showing a $v$-adic analogue of Papanikolas's theorem on the Grothendieck period conjecture (stated in \eqref{eqn: v-G-conj})
for restrictions of scalars of Carlitz $t$-motives, using a trick of `switching $v$ and $\infty$'.
Suppose \eqref{eqn: v-G-conj} also holds for the soliton $t$-motives $M_x$ where $x\in k\setminus A$.
Together with the known description of the motivic Galois groups in the $\infty$-adic setting, it would determine the transcendence degree of the field generated by the entries of the comparison period matrix $[\Phi_{M_x}]_{\beta}$ over $\bar{k}$, which would in turn determine all algebraic relations among the special values of Thakur's $v$-adic geometric gamma function.

Motivated by this perspective, we formulate in Conjecture~\ref{rem: v-LR-conj} a $v$-adic analogue of the Lang--Rohrlich conjecture for geometric gamma values.
A systematic investigation of this conjecture lies beyond the scope of the present paper.

\subsection{Contents}

The paper is organized as follows.
Section~\ref{sec: Pre} establishes the basic notation and conventions used throughout the paper. Then we recall the definitions of Drinfeld modules in Section~\ref{sec: DM} and the basic properties of the Carlitz module and Carlitz cyclotomic fields in Section~\ref{sec: cyclotomic-field}. Section~\ref{sec: t-mot} gives a brief review of $t$-motives, including Gardeyn's notion of good reduction.

Section~\ref{sec: dR-cry} is devoted to the de~Rham--crystalline comparison theory of Hartl--Kim for $t$-motives. We review the de~Rham and crystalline realizations in Sections~\ref{sec: dR} and~\ref{sec: cry}, respectively, and then recall the comparison isomorphism in Section~\ref{sec: comp}. The crystalline action of the Weil group on de~Rham realizations is introduced in Section~\ref{sec: cris-act}.

Section~\ref{sec: Sol} studies soliton $t$-motives. We first recall Anderson's Coleman functions  in Section~\ref{sec: Col}.
Then we give an explicit description of the de~Rham--crystalline comparison isomorphism for soliton $t$-motives in Section~\ref{sec: v-period}, together with a concrete formula for the crystalline Weil action. Section~\ref{sec: v-gamma} discusses fundamental properties of Thakur's $v$-adic geometric gamma function, and Section~\ref{sec: v-gamma-period} establishes its period interpretation by proving Theorem~\ref{thm: MT3}.

Finally, our $v$-adic Chowla--Selberg formula is established in Section~\ref{Sec: v-CS}. After reviewing the definitions and basic properties of CM $t$-motives in Section~\ref{sec: CM}, we prove Theorems~\ref{thm: MT1} and~\ref{thm: MT2} in Section~\ref{sec: v-CS}. We conclude by formulating a $v$-adic analogue of the Lang–Rohrlich conjecture for geometric gamma values at the end.

\section{Preliminaries}
\label{sec: Pre}

\subsection{Basic settings}\label{Sec: Basic-set}
Let $p$ be a prime number and $q$ be a power of $p$. Let $\FF_q$ be a finite field with $q$ elements, and $A=\FF_q[\theta]$ denotes the polynomial ring with one variable $\theta$ over $\FF_q$.
Put $k=\FF_q(\theta)$, the fraction field of $A$.
Let $|\cdot|_{\infty}$ be the absolute value on $k$ associated with the ``degree valuation'', i.e.~
\[
\left|\frac{a}{b}\right|_{\infty}\assign q^{\deg a - \deg b} \quad \text{for every $a,b \in A$ with $b\neq 0$}.
\]
Given $x \in k$, let $\lbangle x \rbangle$ be the fractional part of $x$, i.e.~$\lbangle x \rbangle$ is the unique element in $k$ with $|\lbangle x \rbangle|_{\infty}<1$ and $x-\lbangle x \rbangle \in A$.
Then $\lbangle x+a \rbangle = \lbangle x \rbangle$ for every $x \in k$ and $a \in A$.

Let $A_+$ be the set of monic polynomials in $A$.
Fix an irreducible polynomial $v \in A_+$ once and for all, which corresponds to an absolute value $|\cdot|_v$ on $k$ given by
\[
\left| \frac{a}{b}\right|_v \assign q^{\deg v \cdot (\ord_v(b)-\ord_v(a))}, \quad \forall a,b \in A \text{ with } b \neq 0.
\]
Denote by $k_v$ the completion of $k$ with respect to $|\cdot|_v$ and $\CC_v$ the completion of a chosen algebraic closure $\ok_v$ of $k_v$.
We extend $|\cdot|_v$ to a unique absolute value on $\CC_v$.
Let $\bar{k}$ be the algebraic closure of $k$ in $\CC_v$. 

\subsection{Drinfeld modules}
\label{sec: DM}

Given an $A$-subalgebra $R$ of $\CC_v$,
let $\GG_{a/R}$ be the additive group over $R$.
Denote by $\End_{\FF_q}(\GG_{a/R})$ the $\FF_q$-linear endomorphism ring of $\GG_{a/R}$ over $R$, i.e.
\[
\End_{\FF_q}(\GG_{a/R}) = \bigg\{
\sum_{i=0}^m a_i \rmx^{q^i} \ \bigg|\  a_i \in R
\bigg\}.
\]
Let $\uA\assign \FF_q[t]$ be the polynomial ring with another variable $t$ over $\FF_q$ (i.e.~$t$ is transcendental over $\CC_v$). 
Given a positive integer $r$, a \emph{Drinfeld module of rank $r$ over $R$} is a pair $\uE_{\rho} = (\GG_{a/R}, \rho)$, where $\rho: \uA \rightarrow \End_{\FF_q}(\GG_{a/R})$ is an $\FF_q$-algebra homomorphism satisfying
\[
\rho_t(\rmx) = \theta \rmx +a_1 \rmx^q+\cdots + a_r \rmx^{q^r} \quad \text{ with $a_r \in R^\times$.}
\]

Let $\uE_{\rho}$ and $\uE_{\rho'}$ be two Drinfeld modules over $R$.
A \emph{morphism} $f: \uE_{\rho} \rightarrow \uE_{\rho'}$ is an $\FF_q$-linear endomorphism $f \in \End_{\FF_q}(\GG_{a/R})$ satisfying
\[
f \circ \rho_{\ua} = \rho'_{\ua} \circ f, \quad \forall \ua \in \uA.
\]
Let
$\End(\uE_{\rho/R})
\assign \big\{\text{morphism } f: \uE_{\rho} \rightarrow \uE_{\rho} \text{ over $R$}\big\}$,
called the \emph{endomorphism ring of $\uE_{\rho}$ over $R$}, and put $\End(\uE_{\rho})\assign \End(\uE_{\rho/\CC_v})$.
For every
\[
f(\rmx) = c_0(f)\rmx+ \cdots + c_{\ell}(f) \rmx^{q^{m}} \quad \in \End_{\FF_q}(\GG_{a/R}),
\]
we write $\partial f \assign c_0(f) \in R$ and let $\iota: \uA\stackrel{\sim}{\rightarrow} A$ be the $\FF_q$-algebra isomorphism with $\iota(t) = \theta$.
Then $\partial: \End(\uE_{\rho})\rightarrow \CC_v$ is an $\uA$-algebra embedding satisfying $\partial \rho_{\ua} = \iota(\ua)$ for every $\ua \in \uA$.

It is known that $\End(\uE_{\rho})$ is a commutative $\uA$-algebra which is free as an $\uA$-module with $\rank_{\uA}(\End(\uE_{\rho}))\leq r$ where $r$ is the rank of $\uE_{\rho}$ (see \cite[(4.10) Remark]{DH87}).
We call $\uE_{\rho}$ a \emph{CM Drinfeld module} if 
$\rank_{\uA}(\End(\uE_{\rho})) = r$.

\subsection{Carlitz polynomials and cyclotomic function fields}
\label{sec: cyclotomic-field}

Recall that the \emph{Carlitz module} (over $A$) is the rank one Drinfeld module $\uC \assign (\GG_{a/A},C)$ defined by $C_t(\rmx) = \theta \rmx + \rmx^q$. 
Given a monic polynomial $\unfk \in \uA$, the group of \emph{Carlitz $\unfk$-torsions} 
\[
\uC[\unfk]\assign \big\{\lambda \in \bar{k}\ \big|\ C_{\unfk}(\lambda) = 0\big\}
\]
is equipped with an $\uA$-module structure given by $\ua \cdot \lambda\assign C_{\ua}(\lambda)$ for every $\ua \in \uA$ and $\lambda \in \uC[\unfk]$. It is known that $\uC[\unfk]\cong \uA/\unfk \uA$.
The \emph{$\unfk$-th Carlitz cyclotomic function field over $k$} is $K_{\unfk}\assign k(\uC[\unfk])\subset \bar{k}$, which is a finite abelian extension of $k$ with the Galois group $\Gal(K_{\unfk}/k)$ isomorphic to $(\uA/\unfk \uA)^{\times}$ via the Artin map (see \cite[Theorem 12.8]{Ros02} and \cite[Corollary 2.5]{Hay74}):
\begin{equation}\label{eqn: Artin-map}
(\uA/\unfk \uA)^{\times} \overset{\sim}{\longrightarrow} \Gal(K_{\unfk}/k), \quad 
(\ua \bmod \unfk) \longmapsto \varsigma_{\ua}\assign (\lambda \mapsto C_{\ua}(\lambda)), \quad  \lambda \in \uC[\unfk].
\end{equation}

Let $\lambda_{\unfk}$ be a primitive element in $\uC[\unfk]$, i.e.
\[
\uC[\unfk] = \{C_{\ua}(\lambda_{\unfk})\mid \ua \in \uA\}.
\]
Then the irreducible polynomial of $\lambda_{\unfk}$ over $k$ is called the  \emph{$\unfk$-th Carlitz  cyclotomic polynomial} and denoted by $C_{\unfk}^*(\rmx) \in A[\rmx]$.
Let $O_{\unfk}$ be the integral closure of $A$ in $K_{\unfk}$. Then (see \cite[Proposition 12.9]{Ros02}) 
\[
O_{\unfk} = A[\lambda_{\unfk}] \cong \frac{A[\rmx]}{(C_{\unfk}^*(\rmx))}.
\]

Let $\uA_+$ be the set of monic polynomials in $\uA$.
For convention, we fix a set of compatible generators $\{\lambda_{\unfk}\mid \unfk \in \uA_+\}$ once and for all (i.e.~$C_{\umfk}(\lambda_{\unfk}) = \lambda_{\unfk/\umfk}$ for every $\umfk,\unfk \in \uA_+$ with $\umfk \mid \unfk$).


\subsection{\texorpdfstring{$t$}{t}-motives and their models}\label{sec: t-mot}

Let $\Rcal$ be an $\FF_q$-algebra together with an $\FF_q$-algebra injective endomorphism $(h\mapsto h^{(1)})$ on $\Rcal$. The twisted polynomial ring $\Rcal[\tau]$ is defined by the following multiplication law:
\[
\tau h = h^{(1)} \tau, \quad \forall h \in \Rcal.
\]

Now, let $R$ be an $\FF_q$-subalgebra of $\CC_v$ and define the \emph{$q$-Frobenius twisting map on $R[t]$} by
\[
(\sum_{i=0}^n c_i t^i)^{(1)}\assign \sum_{i=0}^n c_i^{q} t^i, \quad \forall \sum_{i=0}^n c_i t^i \in R[t].
\]
We put $R[t,\tau]\assign R[t][\tau]$, which is the twisted polynomial ring satisfying
\[
t c = c t, \quad t \tau = \tau t, \quad \tau c = c^q \tau, \quad \forall c \in R.
\]
We naturally regard $R[t]$ and $R[\tau]$ as subrings of $R[t,\tau]$.

\begin{defn}\label{defn: t-mot}
Let $F$ be a field extension of $k$.
A \emph{$t$-motive over $F$} is an $F[t,\tau]$-module $M$ which is free of finite rank over $F[t]$ and satisfies
\begin{equation}\label{eqn: *-con}
(t-\theta)^n M \subset F \cdot (\tau M), \quad \text{ for }  n \gg 0.
\end{equation}
\end{defn}

\begin{rem}
${}$
\begin{itemize}
    \item[(1)]
    Our $t$-motives here coincide with Hartl--Juschka's `effective' $t$-motive (over $F$) in \cite{HJ20}.
    \item[(2)] The inclusion \eqref{eqn: *-con} implies in particular that the map $\tau: M \rightarrow M$ is injective, showing that our $t$-motives are `$\tau$-modules over $F[t]$' introduced by Gardeyn in \cite{Gar03}.
    \item[(3)] Suppose $F$ is perfect. Then one has $F\cdot (\tau M) = \tau M$.
    Furthermore, if $M$ is finitely generated over $F[\tau]$, then our definition of $t$-motives agrees with Anderson's in \cite{A86}.
\end{itemize}
\end{rem}

Let $M$ be a $t$-motive over $F$ with $k\subset F\subset \CC_v$. 
For each subextension $L$ of $F$ in $\CC_v$, the \emph{base change of $M$ to $L$} is defined to be $M_{/L}\assign L\otimes_F M$, which is a $t$-motive over $L$ under the natural left $L[t,\tau]$-module structure.

Let $M$ and $M'$ be $t$-motives over $F$.
For each extension $L$ of $F$, a \emph{morphism $f:M\rightarrow M'$ over $L$} is simply an $L[t,\tau]$-module homomorphism from $M_{/L}$ to $M'_{/L}$.
We call $f$ an \emph{isogeny over $L$} if $f$ is injective with finite dimensional cokernel over $L$.
In this case, it is known (by \cite[Proposition 3.1.2]{Tael09} and \cite[Remark 2.3.2 (d)]{HJ20}) that there exist $\ua \in \FF_q[t]$ and an isogeny $g:M'\rightarrow M$ such that 
\[
g\circ f(m) = \ua \cdot m,\quad 
f\circ g (m') = \ua \cdot m', \quad \forall m \in M,\ m' \in M'.
\]
We say that $M$ and $M'$ are \emph{isogenous over $L$} if there exists an isogeny $f:M\rightarrow M'$ over $L$.

Let  
$\Hom(M_{/L},M'_{/L})$ consist of all morphisms from $M$ to $M'$ over $L$, and
put
\[
\End(M_{/L})\assign \Hom(M_{/L},M_{/L}),
\]
called the \emph{endomorphism ring of $M$ over $L$}.
It is known that $\End(M_{/L})$ is a free $\FF_q[t]$-module of rank less than or equal to $\rank_{F[t]}(M)^2$
(see \cite[Corollary 1.7.2]{A86} and \cite[Remark 2.3.7(c)]{HJ20}).
Denote
$\Hom(M,M')\assign \Hom(M_{/\CC_v},M'_{/\CC_v})$ and
$\End(M)\assign \End(M_{/\CC_v})$.
In particular, when $M$ and $M'$ are both `rigid analytically trivial' and defined over $\bar{k}$, it is known that $\Hom(M,M') = \Hom(M_{/\bar{k}},M'_{/\bar{k}})$, whence $\End(M) = \End(M_{/\bar{k}})$ (see \cite[Lemma~2.2.7]{Wei26}).

Given two $t$-motives $M_1$ and $M_2$ over $F$, the \emph{tensor product of $M_1$ and $M_2$} is defined to be 
\[
M_1\otimes M_2 \assign M_1 \otimes_{F[t]} M_2,
\quad \text{with}\quad 
\tau\cdot (m_1\otimes m_2)\assign (\tau m_1 \otimes \tau m_2), \quad \forall m_1 \in M_1,\ m_2 \in M_2.
\]
Let $\text{$t$-}\!\mathscr{M}_{\!F}$ be the category of $t$-motives over $F$. Then it is an $\uA$-linear symmetric monoidal preabelian category.

\begin{Subsubsec}[\emph{$t$-motives associated with Drinfeld modules}]

Let $F$ be a field with $k\subset F \subset \CC_v$.
We may identify
\[
F[\tau] \stackrel{\sim}{\longrightarrow} \End_{\FF_q}(\GG_{a/F}), \quad 
\sum_{i=0}^{n}a_i \tau^i \longmapsto \sum_{i=0}^n a_i \rmx^{q^i}, \quad \forall a_0,...,a_n \in F.
\]
Let $\uE_{\rho}$ be a Drinfeld module over $F$.
The \emph{$t$-motive associated with $\uE_{\rho}$} is $M_{\rho}\assign F[\tau]$
whose left $F[t,\tau]$-module structure is extended from the natural left multiplication of $F[\tau]$ and
\[
t \cdot m \assign m \rho_t, \quad \forall m \in M_{\rho}.
\]
Moreover, we may identify $\End(\uE_{\rho/F})$ with $\End(M_{\rho/F})$ by sending $f \in \End(\uE_{\rho/F})$ to 
\[
\tilde{f}: M_{\rho}\rightarrow M_{\rho}
\quad \text{ with }\quad 
\tilde{f}(m)\assign m f, \quad \forall m \in M_{\rho} = F[\tau].
\]

In particular, let $M_{C}$ be the $t$-motive (over $F$) associated with the Carlitz module $\uC$.
We may identify $M_C$ with $F[t]$ equipped with the $\tau$-action defined by (see \cite[Example 2.3.6]{HJ20}
\[
\tau\cdot f(t) = (t-\theta)\cdot f^{(1)}(t), \quad \forall f(t) \in F[t].
\]
For every $t$-motive $M$ and a non-negative integer $n$, the \emph{$n$-th Tate twist of $M$} is defined by
\begin{equation}\label{eqn: Tate-twist}
M(0) \assign M \quad \text{ and } \quad 
M(n) \assign M \otimes M_{C}^{\otimes n} \quad \text{ for $n>0$.}
\end{equation}
\end{Subsubsec}

\begin{Subsubsec}[\emph{Models and good reductions}]
\label{sec: good-red}

Let $M$ be a $t$-motive over $F$, where $F$ is a finite extension of $k_{v}$ in $\CC_{v}$.
Denote by $O_{F}$ the ring of integers in $F$ and $\pfk_F$ the unique maximal ideal of $O_{F}$.
Let $\FF_F\assign O_{F}/\pfk_F$, the residue field of $F$, which is naturally regarded as a subfield of $O_{F}$ under the Teichm\"uller embedding
 $\eT: \FF_F \hookrightarrow O_F$ 
given as follows:
\begin{equation}\label{eqn: Teichmuller}
\FF_F= \frac{O_F}{\pfk_F}\ni \ c\bmod \pfk_F \ \longmapsto \ \lim_{n\rightarrow\infty } c^{\#(\FF_F)^n} \ \in O_F.
\end{equation}
A \emph{model of $M$} is a left $O_{F}[t,\tau]$-submodule $\Mcal$ of $M$ which is locally-free of finite rank over $O_{F}[t]$ and satisfies $F \Mcal = M$.
By Seshadri's theorem (i.e.~Quillen--Suslin theorem over a principal ideal domain, see \cite{Qui76}, \cite{Sus76} and \cite[Ch.~5, Theorem 2.7]{Lam78}), one has that $\Mcal$ must be free over $O_F[t]$.
Define $\overline{\Mcal} \assign \Mcal / \pfk_F \Mcal$, which is naturally equipped with a left $\FF_F[t,\tau]$-module structure.

\begin{defn}\label{defn: model}
Keep the above notation. The model $\Mcal$ of $M$ is \emph{good} if the map $\tau: \overline{\Mcal} \rightarrow \overline{\Mcal}$ is injective.
We say that \emph{$M$ has good reduction} if there exists a good model of $M$.
\end{defn}

\begin{rem}\label{rem: good-model}
Let $M$ be a $t$-motive over a finite extension $F$ of $k_{v}$ in $\CC_{v}$ as above.
\begin{itemize}
    \item[(1)] It can be seen that a model $\Mcal$ of $M$ is good if and only if there exists a sufficiently large integer $n$ such that
    \[
    (t-\bar{\theta})^n \overline{\Mcal} \subset \tau \overline{\Mcal}.
    \]
    To see this, let $\{m_1,...,m_r\}$ be an $O_F[t]$-base for $\Mcal$. There exists $U \in \Mat_r(O_F[t])$ so that 
    \[
    \begin{pmatrix}
        \tau m_1 \\ 
        \vdots \\
        \tau m_r
    \end{pmatrix}
    = U \cdot \begin{pmatrix}
        m_1 \\ \vdots \\ m_r
    \end{pmatrix}.
    \]
    The condition~\eqref{eqn: *-con} implies that there exists a sufficiently large integer $n$ and $c \in O_F$ such that $\det(U) = c(t-\theta)^n$.
    Then $\tau: \overline{\Mcal}\rightarrow \overline{\Mcal}$ is injective if and only if $c \in O_F^{\times}$, which is equivalent to $(t-\theta)^n\overline{\Mcal} \subset \tau \overline{\Mcal}$ for $n\gg 0$.
    \item[(2)] Suppose $M$ has good reduction.
    Then $M$ has a unique good model up to a unique isomorphism, see \cite[Proposition 2.13~(i)]{Gar03}.
    \item[(3)] Given a morphism $f: M_1\rightarrow M_2$ of $t$-motives over $F$, suppose $M_1$ and $M_2$ both have good reduction (with good models $\Mcal_1$ and $\Mcal_2$, respectively). Then $f$ can be `extended' to an $O_F[t,\tau]$-module homomorphism $f_0:\Mcal_1\rightarrow \Mcal_2$, see \cite[Proposition 2.13~(ii)]{Gar03}.
\end{itemize} 
\end{rem}

Let $L$ be a finite extension of $k$ in $\bar{k}\subset \CC_v$, and let $w$ be the place of $L$ corresponding to $|\cdot|_v$ restricting to $L$.
Let $L_w$ be the closure of $L$ in $\CC_{v}$, regarding as the completion of $L$ at $w$.
We say that a $t$-motive $M$ over $L$ \emph{has good reduction at $w$} if $M_{/L_w}$ has a good model.

\end{Subsubsec}

\section{de~Rham and crystalline realizations of \texorpdfstring{$t$}{t}-motives}\label{sec: dR-cry}

In this section, we recall the de~Rham and crystalline realizations of $t$-motives introduced in \cite{HK20} and their comparison isomorphism.

\subsection{de~Rham module of \texorpdfstring{$t$}{t}-motives}\label{sec: dR}

Let $M$ be a $t$-motive over a field $F$ with $k\subset F \subset \CC_v$.
For each field $L$ with $F\subset L \subset \CC_v$, the \emph{de Rham module of $M$ over $L$} is given by
\[
H_{\dR}(M,L):= \frac{L\cdot \tau M_{/L}}{(t-\theta) \big(L\cdot \tau M_{/L}\big)}.
\]

It is clear that $L\otimes_F H_{\dR}(M,F)\cong H_{\dR}(M,L)$. Moreover, we have the following identification (cf.~\cite[Lemma 3.4.4]{CWY26}):
\begin{lem}\label{lem: dR-id}
Keep the above notation. For each positive integer $\ell$,
the natural inclusion $\tau^{\ell} M\subset \tau M$ induces the following isomorphism
\[
\frac{L\cdot \tau^{\ell} M_{/L}}{(t-\theta)\big(L\cdot \tau^{\ell} M_{/L}\big)} \overset{\sim}{\longrightarrow}
\frac{L\cdot \tau M_{/L}}{(t-\theta) \big(L\cdot \tau M_{/L}\big)} = H_{\dR}(M,L).
\]  
\end{lem}

\begin{proof}
It suffices to prove the case when $L=F$.
The statement clearly holds when $\ell=1$. Suppose $\ell>1$. 
From the definition of $t$-motives one has that
\[
(t-\theta)^n M \subset F\cdot (\tau M) \quad \text{ for sufficiently large integer $n$.}
\]
Multiplying $\tau^{i}$, $1\leq i \leq \ell-1$ on both sides, we get
\[
\prod_{i=1}^{\ell-1}(t-\theta^{q^i})^n \cdot (F\cdot \tau M) \subset F \cdot (\tau^{\ell} M).
\]
Take $a,b \in F[t]$ with 
\[
a(t-\theta)^n + b\prod_{i=1}^{\ell-1}(t-\theta^{q^i})^n = 1.
\]
Then for every $m \in F\cdot \tau M$, 
we have
\[
m = 1\cdot m = (t-\theta)^n \cdot (a m) + \prod_{i=1}^{\ell-1}(t-\theta^{q^i})^n \cdot (bm) \in F\cdot \tau^{\ell} M + (t-\theta) \cdot (F\cdot \tau M).
\]
Hence the inclusion $\tau^{\ell} M \subset \tau M$ induces a surjective homomorphism
\[
F \cdot \tau^{\ell} M \longtwoheadrightarrow \frac{F\cdot \tau M}{(t-\theta)(F\cdot \tau M)} \quad  \text{with kernel $(t-\theta) (F\cdot \tau^{\ell} M)$.}
\]
This shows the desired isomorphism.
\end{proof}

\begin{rem}
Let $\uE_\rho$ be an ``abelian $t$-module'' over $F$.
The de~Rham module of $\uE_\rho$ was first introduced via \lq\lq bi-derivations\rq\rq\ in~\cite{Ge89, Yu90, BP02}. We also refer the readers to \cite[(4.3.3) and (4.3.4)]{NP24} for further details and the compatibility with the de~Rham module $H_{\dR}(M_{\rho},F)$, where $M_{\rho}$ is the $t$-motive associated with $\uE_{\rho}$. 
\end{rem}

Let $f: M_1 \rightarrow M_2$ be a morphism between $t$-motives $M_1$ and $M_2$ over $F$.
As 
\[
f(\tau M_1) = \tau f(M_1) \subset \tau M_2,
\]
we obtain an induced $F$-linear transformation $f_{\dR}: H_{\dR}(M_1,F)\rightarrow H_{\dR}(M_2,F)$.
In particular, $f_{\dR}$ is an isomorphism if $f$ is an isogeny.
Moreover, the natural homomorphism
from $(F\cdot \tau M_1) \otimes_{F[t]} (F\cdot \tau M_2)$ 
to $H_{\dR}(M_1,F)\otimes_F H_{\dR}(M_2,F)$ is surjective and induces a surjective $F$-linear homomorphism from $H_{\dR}(M_1\otimes M_2,F)$ onto $H_{\dR}(M_1,F)\otimes_F H_{\dR}(M_2,F)$. As they have the same dimension, we get a natural isomorphism
\[
H_{\dR}(M_1\otimes M_2,F) \cong H_{\dR}(M_1,F)\otimes_F H_{\dR}(M_2,F).
\]
In conclusion, let $\mathscr{V}_{\!F}$ be the category of finite dimensional vector spaces over $F$. Then the functor
$H_{\dR}(\cdot,F):\text{$t$-}\mathscr{M}_{\!F} \rightarrow \mathscr{V}_{\!F}$ 
given by
\[
M\longmapsto H_{\dR}(M,F), \quad \forall \text{ $t$-motive $M$ \quad and } \quad f\longmapsto f_{\dR}, \quad \forall \text{ morphism $f:M_1\rightarrow M_2$ over $F$},
\]
is a covariant tensor functor.


\subsection{Crystalline module of \texorpdfstring{$t$}{t}-motives}\label{sec: cry}

Recall that $v$ is a fixed monic irreducible polynomial in $A$.
Put $d\assign \deg(v)$. Let $A_v$ be the closure of $A$ in $k_v$, which is equal to the valuation ring $\big\{a_v \in k_v\ \big|\  |a_v|_v\leq 1\big\}$.
We may identify $A_v$ as the power series ring $\FF_{q^d}[\![v]\!]$ in $v$ over $\FF_{q^d}$, where $\FF_{q^d}$ is the finite extension of $\FF_q$ of degree $d$ in $\CC_v$.
Let $\varepsilon_v$ be the Teichm\"uller image of $\theta$ in $\FF_{q^d}$, i.e.~
$\varepsilon_v \assign \lim_{n\rightarrow \infty}\theta^{q^{dn}} \in \FF_{q^d}$.
Then $\FF_{q^d}=\FF_q(\varepsilon_v)$, and
we may also regard $A_v$ as the power series ring $\FF_{q^d}[\![\theta-\varepsilon_v]\!]$.

Let $\iota:\uA \overset{\sim}{\rightarrow} A$ be the $\FF_q$-algebra isomorphism sending $t$ to $\theta$.
Put $\uv\assign \iota^{-1}(v)$, and let $\uA_{\uv}$ be the completion of $\uA$ at $\uv$.
The isomorphism $\iota$ can be extended to an isomorphism $\iota:\uA_{\uv} \overset{\sim}{\rightarrow} A_v$, which provides an identification of $\uA_{\uv}$ as $\FF_{q^d}[\![\uv]\!]$ and also $\FF_{q^d}[\![t-\varepsilon_v]\!]$.
Given a finite extension $F$ of $k_v$ in $\CC_v$, recall in Section~\ref{sec: good-red} that $O_F$ is the ring of integers in $F$, $\pfk_F$ is the unique maximal ideal of $O_F$, and $\FF_F=O_F/\pfk_F$ is the residue field of $F$.
Put 
\[
\uA_{O_{F}}\assign O_{F}\otimes_{\FF_q}\uA \cong O_{F}[t], \quad 
\uA_{\FF_F}:= \FF_F\otimes_{\FF_q}\uA \cong \FF_F[t],
\]
and
\[
\uA_{\uv,O_{F}}\assign O_{F}[\![\uv]\!]
\cong O_{F}[\![t-\varepsilon_v]\!],
\quad 
\uA_{\uv,\FF_F}\assign \FF_F[\![\uv]\!]
\cong \FF_F[\![t-\varepsilon_v]\!]
.
\]
Then the reduction modulo $\pfk_F$ gives a surjection from 
$\uA_{O_F}$ onto $\uA_{\FF_F}$ and from 
$\uA_{\uv,O_F}$ onto $\uA_{\uv,\FF_F}$, and the Teichm\"uller embedding
$\eT:\FF_F\hookrightarrow O_F$ given in \eqref{eqn: Teichmuller}
induces an injection from $\uA_{\FF_F}$ into $\uA_{O_F}$
and from $\uA_{\uv,\FF_F}$ into $\uA_{\uv,O_F}$. In other words, we have the following commutative diagram:
\[
\hspace{2cm} \xymatrix{
\uA_{\FF_F} \ar@{^{(}->}[rr]
\ar@{^{(}->}[rd]^{\eT}
\ar@{=}[dd]
&& \uA_{\uv,\FF_F} \ar@{^{(}->}[rd]^{\eT} \ar@{=}[dd] |!{[d]}\hole && \\
& \uA_{O_F} \ar@{^{(}->}[rr]
\ar@{->>}[ld]^{\bmod \pfk_F} && 
\uA_{\uv,O_F}
\ar@{->>}[ld]^{\bmod \pfk_F} \\
\uA_{\FF_F}
\ar@{^{(}->}[rr] && \uA_{\uv,\FF_F} &
}
\]

Define the $q^d$-Frobenius twisting map on $\uA_{\uv,O_F} = O_F[\![\uv]\!]$ (resp.~$\uA_{\uv, \FF_F} = \FF_F[\![\uv]\!]$) by
\[
(\sum_{i\geq 0} c_i \uv^i)^{(d)}\assign \sum_{i\geq 0} c_i^{q^d}\uv^i, \quad \forall c_0,...,c_i,... \in O_F\quad 
\text{(resp.~$\FF_F$)}.
\]
The twisted polynomial ring $\uA_{\uv,O_F}[\tau^d]$  (resp.~$\uA_{\uv,\FF_F}[\tau^d]$) satisfies
\[
\tau^d \cdot \big(\sum_{i\geq 0} c_i \uv^i\big) = \big(\sum_{i\geq0} c_i^{q^d} \uv^i\big) \cdot \tau^d, \quad \forall\  \sum_{i\geq 0} c_i \uv^i \in O_F[\![\uv]\!] \quad \text{(resp.~$\FF_F[\![\uv]\!]$)}.
\]
Given a $t$-motive $M$ over $F$ which has good reduction,
let $\Mcal$ be a good model of $M$.
The \emph{local shtuka associated with $M$}
is the $\uA_{\uv,O_F}[\tau^d]$ module
\[
\hat{\Mcal}\assign \uA_{\uv,O_F} \otimes_{\uA_{O_F}} \Mcal
\quad \text{ with } \quad 
\tau^d \cdot \bigg(\big(\sum_{n\geq 0} c_n \uv^i\big) \otimes m \bigg)
\assign 
\big(\sum_{n\geq 0} c_n^{q^d} \uv^i\big) \otimes (\tau^d m).
\]
The reduction of $\hat{\Mcal}$ is defined to be
$
\overline{\hat{\Mcal}}\assign \hat{\Mcal}/\pfk_F \hat{\Mcal}$,
which is an $\FF_F[\![\uv]\!][\tau^d]$-module.
Following \cite[Definition 3.5.14]{HK20}, we introduce the crystalline module of $M$ as follows.

\begin{defn}\label{defn: cris}
Keep the above notation. The \emph{crystalline module of $M$} is
\[
H_{\cris}\big(M,\FF_F(\!(\uv)\!)\big)\assign 
\FF_F(\!(\uv)\!)
\otimes_{\FF_F[\![\uv]\!]} 
\tau^d \overline{\hat{\Mcal}},
\] 
\end{defn}

\begin{rem}
${}$
\begin{itemize}
\item[(1)] We may extend the $\tau^d$-action to $\FF_F(\!(\uv)\!)\otimes_{\FF_F[\![\uv]\!]} \overline{\hat{\Mcal}}$ by
\[
\tau^d\cdot \Big((\sum_{i\geq n}^{\infty}c_i \uv^i)\otimes \bar{m}\Big) \assign (\sum_{i\geq n}^{\infty}c_i^{q^d} \uv^i)\otimes \tau^d \bar{m}, \quad \forall \sum_{i\geq n}^{\infty}c_i \uv^i \in \FF_F(\!(\uv)\!) \ \text{ and } \ \bar{m} \in \overline{\hat{\Mcal}}.
\]
Then 
\[
H_{\cris}\big(M,\FF_F(\!(\uv)\!)\big) = \tau^d \cdot \Big(\FF_F(\!(\uv)\!)\otimes_{\FF_F[\![\uv]\!]} \overline{\hat{\Mcal}}\Big),
\]
and the pair $(H_{\cris}\big(M,\FF_F(\!(\uv)\!)\big),\tau^d)$ is called the \emph{$v$-isocrystal of the local shtuka $\hat{\Mcal}$}, see \cite[Example 3.5.7]{HK20}.
\item[(2)]
Since
$\overline{\hat{\Mcal}} \cong \FF_F[\![\uv]\!] \otimes_{\FF_F[t]} \overline{\Mcal}$, we may identify
\[
H_{\cris}\big(M,\FF_F(\!(\uv)\!)\big)\cong \FF_F(\!(\uv)\!) \otimes_{\FF_F[\![\uv]\!]}\tau^d\big(\FF_F[\![\uv]\!] \otimes_{\FF_F[t]} \overline{\Mcal}\big)
\cong \FF_F(\!(\uv)\!)\otimes_{\FF_F[t]}\tau^d \overline{\Mcal}.
\]
\end{itemize}
\end{rem}

Let $f:M_1\rightarrow M_2$ be a morphism between $t$-motives over $F$. Suppose $M_1$ and $M_2$ have good models $\Mcal_1$ and $\Mcal_2$, respectively.
From Remark~\ref{rem: good-model}~(3), we can extend $f$ to an $O_F[t,\tau]$-module homomorphism $f_0:\Mcal_1\rightarrow \Mcal_2$,
which induces an $\FF_F[\![\uv]\!][\tau^d]$-module homomorphism $\bar{f}_0:\overline{\Mcal}_1\rightarrow \overline{\Mcal}_2$.
By tensoring with $\FF_F(\!(\uv)\!)$, we obtain an $\FF_F(\!(\uv)\!)$-linear map
$f_{\cris}:H_{\cris}\big(M_1,\FF_F(\!(\uv)\!)\big)\rightarrow H_{\cris}\big(M_2,\FF_F(\!(\uv)\!)\big)$ such that 
the following diagram commutes:
\[
\xymatrix{
\tau^d \overline{\Mcal}_1\ar[rrr]^{\bar{f}_0} 
\ar[d] & && \tau^d \overline{\Mcal}_2 
\ar[d]\\
\FF_F(\!(\uv)\!)\otimes_{\FF_F[t]} \tau^d \overline{\Mcal}_1 \ar[r]^{\sim} & H_{\cris}\big(M_1,\FF_F(\!(\uv)\!)\big) \ar[r]^{f_{\cris}} &
H_{\cris}\big(M_2,\FF_F(\!(\uv)\!)\big)
& \FF_F(\!(\uv)\!)\otimes_{\FF_F[t]}\tau^d \overline{\Mcal}_2 \ar[l]_{\sim}.
}
\]
In particular, $f_{\cris}$ preserves the $\tau^d$-actions on both sides, and $f_{\cris}$ becomes an isomorphism when $f$ is an isogeny.
Moreover, $\Mcal_1 \otimes_{O_F[t]} \Mcal_2$ is a good model of $M_1\otimes M_2$, and the identity
\[
O_F \cdot \tau^d (\Mcal_1\otimes_{O_F[t]} \Mcal_2) = (O_F \cdot \tau^d \Mcal_1)\otimes_{O_F[t]}(O_F\cdot \tau^d \Mcal_2) \quad \subset \Mcal_1\otimes_{O_F[t]}\Mcal_2
\]
modulo $\pfk_F$ gives
\[
\tau^d (\overline{\Mcal_1\otimes_{O_F[t]}\Mcal_2})\  \cong \ \tau^d \overline{\Mcal}_1\otimes_{\FF_F[t]} \tau^d \overline{\Mcal}_2 \quad \subset \overline{\Mcal}_1\otimes_{\FF_F[t]}\overline{\Mcal}_2.
\]
By tensoring with $\FF_F(\!(\uv)\!)$, we get
\[
H_{\cris}\big(M_1\otimes M_2,\FF_F(\!(\uv)\!)\big) \ \cong \ H_{\cris}\big(M_1,\FF_F(\!(\uv)\!)\big) \otimes_{\FF_F(\!(\uv)\!)} H_{\cris}\big(M_2,\FF_F(\!(\uv)\!)\big).
\]
In conclusion,
let $\text{$t$-}\mathscr{M}_{\!F}^{\rm good}$
be the full-subcategory of $\text{$t$-}\mathscr{M}_{\!F}$ whose objects are $t$-motives over $F$ which have good reduction, and 
let $\mathscr{V}_{\!\FF_F(\!(\uv)\!)}$ be the category of finite dimensional vector spaces over $\FF_F(\!(\uv)\!)$. Then the functor
$H_{\cris}\big(\cdot,\FF_F(\!(\uv)\!)\big):\text{$t$-}\mathscr{M}_{\!F}^{\rm good} \rightarrow \mathscr{V}_{\!\FF_F(\!(\uv)\!)}$ 
given by
\[
M\longmapsto H_{\cris}(M,\FF_F(\!(\uv)\!)\big), \quad \forall \text{ $M \in {\rm ob}(\text{$t$-}\mathscr{M}_{\!F}^{\rm good})$ \quad and }\quad f\longmapsto f_{\cris}, \quad \forall \text{ morphism $f:M_1\rightarrow M_2$},
\]
is a covariant tensor functor.

\begin{rem}
In fact, this functor actually factors through the category of ``$v$-isocrystals'', see \cite[Definition 3.5.11]{HK20}.
\end{rem}

\subsection{Comparison isomorphism}\label{sec: comp}

To establish the de~Rham--crystalline comparison isomorphism, we first recall the following 
``big ring'' (see \cite[equation (3.5.2)]{HK20}
\begin{equation}\label{eqn: F-ring}
O_F[\![\uv, \uv^{-1}\}
\assign
\left\{
\sum_{i=-\infty}^{\infty} c_i \uv^i\ 
\bigg|\ c_i \in O_F
\ \text{ and } \lim_{i\rightarrow -\infty} |c_i|_v |v|_v^{si} = 0 \text{ for all $s>0$}
\right\},
\end{equation}
which contains (see \cite[Lemma 3.5.2]{CWY26}) 
\[
\omega_v(\uv) \assign \prod_{i=0}^{\infty}\left(1-\frac{v^{q^{di}}}{\uv}\right)
\ \ \text{and} \ \  \ 
\omega_{\theta,v}^{(N)}(\uv) \assign
\prod_{i=N}^{\infty}\left(1-\frac{(\theta-\varepsilon_v)^{q^{i}}}{t-\varepsilon_v^{q^i}}\right),\quad \forall N \in \ZZ_{\geq 0}.
\]
Here we identify $t$ as a power series $t(\uv)\in \FF_{q^d}[\![\uv]\!]$ via the embedding 
$\FF_q[t]=\uA\subset \uA_{\uv} = \FF_{q^d}[\![\uv]\!]$ with $t\big|_{\uv=v} = \theta$ (see \cite[(3.2.1)]{CWY26}).
Note that the equalities \cite[(3.5.7) and (3.5.12)]{CWY26} say in particular that for each $N \in \ZZ_{\geq 0}$, there exists $u_N(\uv) \in 1+ v A_v[\![\uv]\!] \subset A_v[\![\uv]\!]^{\times}$ such that 
\[
\omega_v(\uv) = u_N(\uv) \cdot \omega_{\theta,v}^{(N)}(\uv).
\]
Also,
one has that 
\[ 
(1-\frac{v}{\uv}) \cdot
\omega_v^{(d)}(\uv)=  \omega_v(\uv) \quad \text{ and } \quad 
\left(\prod_{i=0}^{d-1}(1-\frac{(\theta-\varepsilon_v)^{q^i}}{t-\varepsilon_v^{q^i}})
\right)
\cdot 
\omega_{\theta,v}^{(d)}(\uv) = \omega_{\theta,v}(\uv).
\]
which implies that for every $ i \in \ZZ_{\geq 0}$,
\begin{equation}\label{eqn: t-theta-inv}
 (1-\frac{(\theta-\varepsilon_v)^{q^i}}{t-\varepsilon_v^{q^i}}) = \frac{t-\theta^{q^i}}{t-\varepsilon_v^{q^i}} \ \text{ is invertible in }\  O_F[\![\uv,\uv^{-1}\}[\omega_v(\uv)^{-1}].
\end{equation}
This enables us to define a $q^d$-Frobenius twisting map on $O_F[\![\uv,\uv^{-1}\}[\omega_v(v)^{-1}]$ by
\[
(\sum_{i=-\infty}^{\infty}c_i \uv^i)^{(d)}\assign \sum_{i=-\infty}^{\infty}c_i^{q^d} \uv^i
\quad \text{ and } \quad 
(\omega_v(\uv)^{-1})^{(d)} \assign (1-\frac{v}{\uv}) \cdot \omega_v(\uv)^{-1}.
\]
Let $O_F[\![\uv,\uv^{-1}\}[\omega_v(v)^{-1}][\tau^d]$ be the associated twisted polynomial ring.
The lemma of Genestier--Lafforgue
\cite[Lemma 7.4]{GL11}
shows
(see also \cite[Lemma 3.5.8]{HK20}):

\begin{lem}\label{lem: key-iso}
There exists a unique $O_F[\![\uv,\uv^{-1}\}[\omega_v(\uv)^{-1}][\tau^d]$-module isomorphism
\[
\delta_{\Mcal}:
 O_F[\![\uv,\uv^{-1}\}[\omega_v(\uv)^{-1}]   
\otimes_{O_F[t]}
\Mcal
\overset{\sim}{\longrightarrow}  O_F[\![\uv,\uv^{-1}\}[\omega_v(\uv)^{-1}] 
\otimes_{\FF_F[t]}
\overline{\Mcal}
\]
which is congruent, modulo $\pfk_F$, to the identity map on 
$\FF_F(\!(\uv)\!)\otimes_{\FF_F[t]} \overline{\Mcal}$.
\end{lem}

Moreover, we embed $O_F[\![\uv,\uv^{-1}\}$ into $F[\![\uv-v]\!]$
through the identification $\uv = v+(\uv-v)$.
One sees that
$\omega_v^{(d)}(\uv)$ and
$\omega_{\theta,v}^{(d)}(\uv)$
are invertible in $F[\![\uv-v]\!]$ (as their  values at $\uv = v$ are 
\[
\omega_v^{(d)}(v)=
\prod_{i=1}^{\infty}\left(1-\frac{v^{q^di}}{v}\right)\quad \text{ and } \quad 
\omega_{\theta,v}^{(d)}(v)
= \prod_{i=d}^{\infty}\left(
1-\frac{(\theta-\varepsilon_v)^{q^i}}{\theta-\varepsilon_v^{q^i}}
\right),
\quad \text{which are both nonzero).}
\]
Therefore the isomorphism $\delta_{\Mcal}$ in Lemma~\ref{lem: key-iso} leads to an $F[\![\uv-v]\!]$-module isomorphism
\begin{equation}\label{eqn: delta-tilde-M}
\tilde{\delta}_{\Mcal}:F[\![\uv-v]\!] \otimes_{O_F[t]} \Big(O_F \cdot (\tau^d \Mcal)\Big)  
\ \overset{\sim}{\longrightarrow}\ 
F[\![\uv-v]\!] \otimes_{\FF_F[t]} \tau^d \overline{\Mcal}
.
\end{equation}
Here $t$ maps into $F[\![\uv-v]\!]$ via the inclusion 
\[
\FF_q[t]=\uA \subset \uA_{\uv} = \FF_{q^d}[\![\uv]\!] \subset \FF_F[\![\uv]\!] =\uA_{\uv,\FF_F} \overset{\eT}{\hookrightarrow} \uA_{\uv,O_F} = O_F[\![\uv]\!] \subset O_F[\![\uv,\uv^{-1}\} \subset F[\![\uv-v]\!].
\]

Taking reduction modulo $(\uv-v)$ on both sides, one sees that the left hand side of \eqref{eqn: delta-tilde-M} becomes
\[
F \underset{\theta\mapsfrom t,\, O_F[t]}{\otimes}  \Big(O_F \cdot (\tau^d \Mcal)\Big)  \ \cong \ \frac{F\cdot (\tau^d M)}{(t-\theta) \big(F\cdot \tau^d M)} = H_{\dR}(M,F); 
\]
while the right hand side of \eqref{eqn: delta-tilde-M} becomes
\begin{align*}
F
\underset{\theta \mapsfrom t,\, \FF_F[t]}{\otimes}
\tau^d \overline{\Mcal}
& \cong  F 
\underset{v \mapsfrom \uv, \, \FF_F(\!(\uv)\!)}{\otimes} 
\FF_F(\!(\uv)\!) \otimes_{\FF_F[t]} \tau^d \overline{\Mcal} 
\\
& = 
F 
\underset{v \mapsfrom \uv, \, \FF_F(\!(\uv)\!)}{\otimes} 
H_{\cris}\big(M,\FF_F(\!(\uv)\!)\big)
\rassign H_{\cris}(M,F).
\end{align*}
Therefore:
\begin{prop}\label{prop: c-d-iso}
The isomorphism $\tilde{\delta}_{\Mcal}$ in \eqref{eqn: delta-tilde-M} induces the de~Rham--crystalline comparison isomorphism
\[
\Phi_M: H_{\dR}(M,F) \overset{\sim}{\longrightarrow} H_{\cris}(M,F).
\]
\end{prop}

\begin{rem}
The uniqueness of $\delta_M$ implies in particular that for $t$-motives $M_1$ and $M_2$ over $F$ which have good reduction and a morphism $f:M_1 \rightarrow M_2$, we have
\[
\Phi_{M_2} \circ f_{\dR} = f_{\cris} \circ \Phi_{M_1}.
\]
Moreover, the following diagram commutes:
\[
\xymatrix{
H_{\dR}(M_1\otimes M_2,F) \ar[rr]^{\Phi_{M_1\otimes M_2}}_{\sim} 
\ar[d]_[@! 90]{\sim}
&& H_{\cris}(M_1\otimes M_2,F) \ar[d]^[@! 90]{\sim} \\
H_{\dR}(M_1,F)\otimes_F H_{\dR}(M_2,F) \ar[rr]^{\Phi_{M_1}\otimes \Phi_{M_2}}_{\sim} &&
H_{\cris}(M_1,F)\otimes_F H_{\cris}(M_2,F).
}
\]
Thus the isomorphism $\Phi_M$ for $M \in {\rm ob}(\text{$t$-}\mathscr{M}_{\!F}^{v})$ gives a tensor isomorphism between the tensor functors $H_{\dR}(\cdot,F)$ and $H_{\cris}(\cdot,F)$.
\end{rem}

\begin{rem}\label{rem: c-d-period-matrix}
Suppose $M$ is defined over $F_0\assign F \cap \bar{k}$, i.e.~$M = M_{0/F} = F \otimes_{F_0} M_0$ where $M_0$ is defined over $F_0$.
Let $\beta_0=\{m_1,...,m_r\}$ be an $O_F[t]$-base  for the good model $\Mcal$ of $M$ with $m_1,...,m_r \in M_0$.
Then 
\[
\beta_{0,\dR}\assign\big\{\omega_{0,i}\assign \tau^d m_i  \bmod (t-\theta) \in H_{\dR}(M_0,F_0)\ \big|\  i=1,...,r\big\}
\]
is an $F_0$-basis for $H_{\dR}(M_0,F_0)$.
On the other hand, for $i=1,...,r$, let $\bar{m}_i \assign m_i \bmod \pfk_F \in \overline{\Mcal}$.
Then
\[
\beta_{0,\cris}\assign 
\Big\{
\bar{\omega}_{0,i}\assign 1\otimes \tau^d \bar{m}_i \in F
\underset{\theta \mapsfrom t,\, \FF_F[t]}{\otimes}
\tau^d \overline{\Mcal} = H_{\cris}(M,F)\ \Big|\  i=1,...,r \Big\}
\]
is an $F$-basis for $H_{\cris}(M,F)$.
The \emph{de~Rham--crystalline comparison period matrix of $M$ with respect to $\beta_0$} is the matrix $[\Phi_M]_{\beta_0} \in \Mat_r(F)$ defined by
\[
\begin{pmatrix}
\Phi_M(\omega_{0,1}) \\
\vdots \\
\Phi_M(\omega_{0,r})
\end{pmatrix}
= 
[\Phi_M]_{\beta_0}^{\rm t} \cdot 
\begin{pmatrix}
    \bar{\omega}_{0,1} \\
    \vdots \\
    \bar{\omega}_{0,r}
\end{pmatrix}.
\]
Here $[\Phi_M]_{\beta_0}^{\rm t}$ denotes the transpose of $[\Phi_M]_{\beta_0}$.
In general, suppose $\beta = \{\omega_1,...,\omega_r\}$ is another $F_0$-basis of $H_{\dR}(M_0,F_0)$.
Write
\[
\begin{pmatrix}
    \omega_1 \\
    \vdots \\
    \omega_r
\end{pmatrix}
= U^{\rm t} \cdot 
\begin{pmatrix}
    \omega_{0,1} \\
    \vdots \\
    \omega_{0,r}
\end{pmatrix}
\quad \text{ where $U \in \GL_r(F_0)$}.
\]
Let $\bar{\beta}\assign\{\bar{\omega}_1,...,\bar{\omega}_r\} \subset H_{\cris}(M,F)$ be defined by
\[
\begin{pmatrix}
    \bar{\omega}_1 \\
    \vdots \\
    \bar{\omega}_r
\end{pmatrix}
\assign U^{\rm t} \cdot 
\begin{pmatrix}
    \bar{\omega}_{0,1} \\
    \vdots \\
    \bar{\omega}_{0,r}
\end{pmatrix}.
\]
Then
\begin{equation}\label{eqn: [Phi-M]}
\begin{pmatrix}
    \Phi_{M}(\omega_1) \\
    \vdots \\
    \Phi_{M}(\omega_r)
\end{pmatrix}
= [\Phi_M]_{\beta}^t \cdot
\begin{pmatrix}
    \bar{\omega}_1 \\
    \vdots \\
    \bar{\omega}_r
\end{pmatrix}
\quad \text{ where } \quad 
[\Phi_M]_{\beta} \assign U^{-1}\cdot [\Phi_M]_{\beta_0} \cdot U.
\end{equation}

Note that in \cite[Theorem~1.2.3 and 1.3.1]{CWY26}, we completely determined the transcendence properties of the entries in $[\Phi_M]_{\beta}$ when $M$ is the \emph{restriction of scalars of a degree-$\ell$ Carlitz $t$-motive} (see Remark~\ref{rem: RC}).
This example suggests the possibility of a $v$-adic analogue of the Grothendieck period conjecture for de~Rham--crystalline comparison periods. More precisely, it is natural to ask whether
\begin{equation}\label{eqn: v-G-conj}
\trdeg_{\bar{k}}
\Big(\bar{k}\big([\Phi_M]_{\beta}\big)\Big)
\stackrel{?}{=}
\dim\Gamma_M,
\end{equation}
where $\Gamma_M$ denotes the motivic Galois group of $M$.
A systematic investigation of this question lies beyond the scope of this paper, and we do not pursue it further here.
\end{rem}

\subsection{Crystalline action}\label{sec: cris-act}

Let $\Wcal_v$ be the \emph{Weil group} in $\Gal(\bar{k}_v/k_v)$, i.e.~elements in $\Wcal_v$ are lifts of powers of the Frobenius automorphism on the residue fields. More precisely, for every $\sigma \in \Wcal_v$, there exists a unique integer $\deg \sigma$ such that 
\[
|\sigma(c) -  c^{q^{d\deg \sigma}}|_v<1, \quad \forall c \in \bar{k}_v \ \text{ with }\  |c|_v\leq 1.
\]
Then the $\tau^d$-action on $H_{\cris}\big(M,\FF_F(\!(\uv)\!)\big)$ extends to a semilinear $\Wcal_v$-action on
\[
H_{\cris}(M,\bar{k}_v)\assign \bar{k}_v \otimes_F H_{\cris}(M,F)\  \cong \ \bar{k}_v \underset{v \mapsfrom \uv,\, \FF_F(\!(\uv)\!)}{\otimes} H_{\cris}\big(M,\FF_F(\!(\uv)\!)\big)
\]
defined as follows: for every $\sigma \in \Wcal_v$ with $\deg \sigma\geq 0$, the corresponding $\sigma$-linear map $\tilde{\sigma}_{\cris}$ is given by
\begin{equation}\label{eqn: t-sigma-cris}
\tilde{\sigma}_{\cris}(c \otimes m) \assign \sigma(c) \otimes (\tau^{d\deg \sigma} \cdot m),
\quad \forall c \in \bar{k}_v,\ m \in H_{\cris}\big(M,\FF_F(\!(\uv)\!)\big).
\end{equation}
Extending the isomorphism $\Phi_M$ from $F$ to $\bar{k}_v$, we may transform this semilinear $\Wcal_v$-action to $H_{\dR}(M,\bar{k}_v)$,
called the \emph{crystalline action}.
In concrete terms,
for every $\sigma \in \Wcal_v$, we set 
\begin{eqnarray}\label{eqn: sigma-cris}
\sigma_{\cris}& \assign & \Phi_M^{-1}\circ \tilde{\sigma}_{\cris} \circ \Phi_M: H_{\dR}(M,\bar{k}_v) \longrightarrow H_{\dR}(M,\bar{k}_v),
\end{eqnarray}
which is $\sigma$-linear 
so that the following diagram commutes:
\[
\xymatrix{
H_{\dR}(M,\bar{k}_v) \ar[r]^{\Phi_M}_{\sim}
\ar[d]_{\sigma_{\cris}} &
H_{\cris}(M,\bar{k}_v) \ar[d]^{\tilde{\sigma}_{\cris}} \\
H_{\dR}(M,\bar{k}_v) \ar[r]^{\Phi_M}_{\sim}
& H_{\cris}(M,\bar{k}_v).
}
\]
From the functoriality of the $\tau^d$-action on $H_{\cris}\big(M,\FF_F(\!(\uv)\!)\big)$ and the isomorphism $\Phi_M$, we have the compatibility of this crystalline action, i.e.~given $M_1,M_2 \in {\rm ob}(\text{$t$-}\mathscr{M}_{\!F}^{v})$ and a morphism $f:M_1\rightarrow M_2$ over $\bar{k}_v$, the following diagrams commute
\begin{equation}\label{eqn: equivariance}
\xymatrix{
H_{\dR}(M_1,\bar{k}_v)\ar[r]^{f_{\dR}} \ar[d]_{\sigma_{\cris}} & H_{\dR}(M_2,\bar{k}_v) \ar[d]^{\sigma_{\cris}} \\
H_{\dR}(M_1,\bar{k}_v) \ar[r]^{f_{\dR}} & M_{\dR}(M_2,\bar{k}_v),
}
\end{equation}
\[
\xymatrix{
H_{\dR}(M_1\otimes M_2,\bar{k}_v) \ar[r]^<<<<{\sim} 
\ar[d]_{\sigma_{\cris}} 
& H_{\dR}(M_1,\bar{k}_v)\otimes_F H_{\dR}(M_2,\bar{k}_v) \ar[d]^{\sigma_{\cris}\otimes \sigma_{\cris}} \\
H_{\dR}(M_1\otimes M_2,\bar{k}_v) \ar[r]^<<<<{\sim}  
& H_{\dR}(M_1,\bar{k}_v)\otimes_F H_{\dR}(M_2,\bar{k}_v).
}
\]

Let $r = \dim_{\bar{k}_v}\big(H_{\dR}(M,\bar{k}_v)\big)$ and $\beta = \{\omega_1,...,\omega_r\}$ be an ordered $\bar{k}_v$-basis for $H_{\dR}(M,\bar{k}_v)$.
For every $\sigma \in \Wcal_v$,
there exists $\Pcal^M_{\beta}(\sigma) \in \Mat_r(\bar{k}_v)$ so that
\begin{equation}\label{eqn: v-PM}
\hspace{2.5cm}
\begin{pmatrix}
\sigma_{\cris}(\omega_1) \\
\vdots \\
\sigma_{\cris}(\omega_r)
\end{pmatrix}
= \Pcal^M_{\beta}(\sigma)^{\rm t} \cdot 
\begin{pmatrix}
    \omega_1 \\
    \vdots \\
    \omega_r
\end{pmatrix} \hspace{1cm} \text{(here $\Pcal^M_{\beta}(\sigma)^{\rm t}$
is the transpose of $\Pcal^M_{\beta}(\sigma)$)}.
\end{equation}
We call $\Pcal^M_{\beta}(\sigma)$ the \emph{$v$-isocrystal period matrix of $\sigma$ on $M$ with respect to $\beta$}.
In particular, we have the following ``$1$-cocycle relation''
\[
\Pcal^M_{\beta}(\sigma \sigma') = 
\Pcal^M_{\beta}(\sigma)\cdot {}^{\sigma} \Pcal^M_{\beta}(\sigma'), \quad \forall \sigma,\sigma' \in \Wcal_v.
\]

\begin{rem}\label{rem: matrix P-Phi}   

Suppose $M$ is defined over $F_0$ as in Remark~\ref{rem: c-d-period-matrix} and $\beta$ is algebraic, which means $\beta \subset H_{\dR}(M,F_0)$.  
Then
\[
\Pcal_{\beta}^M(\sigma) = [\Phi_M]_{\beta}^{-1} \cdot [\tilde{\sigma}_{\cris}]_{\bar{\beta}} \cdot {}^{\sigma}[\Phi_M]_{\beta},
\]
where $[\Phi]_{\beta}$ is the crystalline-de~Rham comparison period matrix of $M$ with respect to $\beta$, ${}^{\sigma}[\Phi_M]_{\beta}$ is obtained by applying $\sigma$ to each entry of $[\Phi_M]_{\beta}$,
and the matrix $[\tilde{\sigma}_{\cris}]_{\bar{\beta}} \in \GL_r(\bar{k})$ is determined by
\[
\begin{pmatrix}
    \tilde{\sigma}_{\cris}(\bar{\omega}_1) \\
    \vdots \\
    \tilde{\sigma}_{\cris}(\bar{\omega}_r)
\end{pmatrix}
= [\tilde{\sigma}_{\cris}]_{\bar{\beta}}^{\rm t} \cdot 
\begin{pmatrix}
    \bar{\omega}_1 \\
    \vdots \\
    \bar{\omega}_r
\end{pmatrix}.
\]
Consequently, $\Pcal_\beta^M(\sigma)$ is obtained by transporting the algebraic crystalline action to the de~Rham realization via the comparison isomorphism. This justifies regarding $\Pcal_\beta^M(\sigma)$ as the natural $v$-adic analogue of a period matrix.
\end{rem}

\begin{rem}
The evaluation $\uv\mapsto v$ gives an isomorphism $\FF_F(\!(\uv)\!)\stackrel{\sim}{\rightarrow} \FF_F(\!(v)\!) =\FF_Fk_v$, the maximal unramified subextension of $F$ over $k_v$.
Keep the notation of Remark~\ref{rem: c-d-period-matrix}.
Then for every $\sigma \in \Wcal_v$ with $\deg \sigma =s >0$, we have
\[
[\tilde{\sigma}_{\cris}]_{\beta_{0,\cris}}
= [\tilde{\sigma}_{v,\cris}]_{\beta_{0,\cris}} \cdot ({}^{\sigma_v}[\tilde{\sigma}_{v,\cris}]_{\beta_{0,\cris}}) \cdots ({}^{\sigma_v^{s-1}}[\tilde{\sigma}_{v,\cris}]_{\beta_{0,\cris}}),
\]
where $\sigma_v \in \Wcal_v$
is any choice of Frobenius automorphism of $\bar{k}_v/k_v$ (i.e.~$\deg\sigma_v = 1$).
In other words, to describe the crystalline action, it is essentially sufficient to determine a single $v$-isocrystal period matrix $\Pcal^M_{\beta}(\sigma_v)$ (with respect to a given $F$-basis $\beta$ for $H_{\dR}(M,F)$).
\end{rem}

\begin{ex}[\emph{Carlitz $t$-motive}]\label{sec: Ex-Car-t}
Recall that $M_C = k[t]$ has a good model $\Mcal_C = A_v[t]$.
From the identifications of both $A_v[\![\uv,\uv^{-1}\}[\omega_{v}(\uv)^{-1}]\otimes_{A_v[t]} \Mcal_C$
and 
$A_v[\![\uv,\uv^{-1}\}[\omega_{v}(\uv)^{-1}]\otimes_{\FF_v[t]} \overline{\Mcal}_C$
with $A_v[\![\uv,\uv^{-1}\}[\omega_{v}(\uv)^{-1}]$,
the isomorphism $\delta_{M_C}$ in Lemma~\ref{lem: key-iso} is given by
\[
\delta_{M_C}(m) = \omega_{\theta,v}(\uv)\cdot m, \quad \forall m \in A_v[\![\uv,\uv^{-1}\}[\omega_{v}(\uv)^{-1}].
\]
Moreover, let $\omega_C \in H_{\dR}(M_C,k_v)$ be the differential corresponding to the element
\[
\prod_{i=1}^{d-1}(\theta-\theta^{q^i})^{-1} \cdot (\tau^d \cdot 1) \in k \cdot \tau^d M_C.
\]
As $\omega_{\theta,v}^{(d)}(v) \in k_v^{\times}$,
the crystalline action of $\sigma_v$ on $\omega_C$ is given by \cite[Theorem~1.2.3]{CWY26}
\[
\sigma_{v,\cris}(\omega_C) = \omega_{\theta,v}^{(d)}(v)^{-1} \cdot \left(\prod_{i=0}^{d-1}(\theta-\varepsilon_v^{q^i})\right) \cdot \omega_{\theta,v}^{(d)}(v) \cdot  \omega_C = v \cdot \omega_C.
\]
In other words, if we let $\beta = \{\omega_C\}$,
then
\[
[\Phi_{M_C}]_{\beta} =  {}^{\sigma}[\Phi_{M_C}]_{\beta} = \omega_{\theta,v}^{(d)}(v), \quad \forall \sigma \in \Wcal_v, \quad 
\text{ and } 
\Pcal_{M_C}(\sigma_v)=[\tilde{\sigma}_{v,\cris}]_{\bar{\beta}} = v.
\]
In this case, it is known that $\omega_{\theta,v}(v)$ is transcendental over $k$ (see \cite[Theorem~1.3.1]{CWY26}), and the crystalline action becomes simply multiplying the corresponding power of $v$.
\end{ex}

\begin{rem}\label{rem: RC}
Given a positive integer $\ell$, let $\uC_{\ell} = (\GG_{a/k},C_{\ell})$ be the degree-$\ell$ Carlitz module (over $k$) defined by $C_{\ell,t}(\rmx) = \theta \rmx + \rmx^{q^{\ell}}$.
Then $\uC_{\ell}$ is a Drinfeld module of rank $\ell$ whose endomorphism ring $\End(\uC_{\ell}) \cong \FF_{q^{\ell}}[t]$.
The corresponding $t$-motive, denoted by 
$R_{\tau}(\uC_{\ell})$, is called the \emph{restriction of scalars of $\uC_{\ell}$}  (see \cite[Definition~3.1.1]{CWY26}). 
It was shown in \cite[Theorem 1.2.1]{CWY26} that the associated $v$-isocrystal period matrix $\Pcal^{R_{\tau}(\uC_{\ell})}_{\beta}(\sigma_v)$ (for a particular ``algebraic'' basis $\beta$) admits an explicit description in terms of $v$-adic arithmetic gamma values, which agrees with the formula in Example~\ref{sec: Ex-Car-t} when $\ell=1$.
Moreover, from the decomposition of $\Pcal^{R_{\tau}(\uC_{\ell})}_{\beta}(\sigma_v)$ in terms of the associated de~Rham--crystalline period matrix as in Remark~\ref{rem: matrix P-Phi}, a $v$-adic ``arithmetic'' Lang--Rohrlich conjecture was proved in \cite[Theorem~1.4.1]{CWY26} by completely deteremining the transencental properties of the de~Rham--crystalline period matrix. 
\end{rem}

In the following sections, we turn to the ``geometric'' case and establish a $v$-isocrystal period interpretation of $v$-adic geometric special gamma values using ``soliton $t$-motives''.

\section{Soliton \texorpdfstring{$t$}{t}-motives}\label{sec: Sol}

Given a monic polynomial $\unfk \in \uA$, recall that $C_{\unfk}^*(\rmx) \in A[\rmx]=\FF_q[\theta,\rmx]$ is the $\unfk$-th Carlitz cyclotomic polynomial introduced in Section~\ref{sec: cyclotomic-field}.
Let $\FF_q[t,z]=\FF_q[t][z]$ be the polynomial ring with one variable $z$ over $\FF_q[t]=\uA$.
Extend $\iota:\uA \stackrel{\sim}{\rightarrow} A$ to $\uA[z]\stackrel{\sim}{\rightarrow} A[\rmx]$ by sending $z$ to $\rmx$.
Let 
\[
C_{\unfk}^*(t,z)\assign\iota^{-1}(C_{\unfk}^*(\rmx)) \in \FF_q[t,z],
\quad \text{ and put} \quad 
\uO_{\unfk} \assign \frac{\FF_q[t,z]}{(C_{\unfk}^*(t,z))} \underset{\sim}{\overset{\iota}{\longrightarrow}} \frac{A[\rmx]}{C_{\unfk}^*(\rmx)} \cong O_{\unfk}.
\]
The last isomorphism induces from the evaluation homomorphism $A[\rmx]\rightarrow O_{\unfk}$ sending $\rmx$ to $\lambda_{\unfk}$, where $\lambda_{\unfk}$ is the fixed primitive Carlitz $\unfk$-torsion in the end of Section~\ref{sec: cyclotomic-field}.

Let $\uU_{\unfk}\assign \Spec(\uO_{\unfk})$, $U_{\unfk} \assign \Spec(O_{\unfk})$, and their smooth projective models are denoted by $\uX_{\unfk}$ and $X_{\unfk}$, respectively.
Then $X_{\unfk}\times \uX_{\unfk}$ is a smooth projective surface over $\FF_q$ with an affine open subset 
\[
U_{\unfk}\times \uU_{\unfk} = \Spec(O_{\unfk}\otimes_{\FF_q}\uO_{\unfk}), \quad \text{ and } \quad 
O_{\unfk}\otimes_{\FF_q}\uO_{\unfk}
\cong \frac{O_{\unfk}[t,z]}{(C_{\unfk}^*(t,z))}.
\]

Given  $\ua \in \uA$ that is coprime to $\unfk$, we put
\begin{equation}\label{eqn: xi-a}
\xi_{\ua}\assign \varsigma_{\ua}\circ \iota: \uO_{\unfk} \overset{\sim}{\longrightarrow} O_{\unfk} \quad \text{ with } \xi_{\ua}(t)=\theta \text{ and } \xi_{\ua}(z)= C_{\ua}(\lambda_{\unfk}).
\end{equation}
We also denote the $O_{\unfk}$-valued point $(\theta,C_{\ua}(\lambda_{\unfk}))$ of $\uU_{\unfk}$ by $\xi_{\ua}$.
Moreover, let $\uK_{\unfk}$ be the fraction field of $\uO_{\unfk}$.
Then $\xi_{\ua}$ can be extended to an isomorphism $\uK_{\unfk}\stackrel{\sim}{\rightarrow} K_{\unfk}$.
For each nonnegative integer $i$, let $\xi_{\ua}^{(i)}$ be the $i$-th Frobenius twist of $\xi_{\ua}$, which is regarded as an $\FF_q$-algebra homomorphism 
\[
\xi_{\ua}^{(i)}: \uK_{\unfk}\rightarrow K_{\unfk}
\quad \text{with} \quad 
\xi_{\ua}^{(i)}(t) = \theta^{q^i} \text{ and } \xi_{\ua}^{(i)}(z) = C_{\ua}(\lambda_{\unfk})^{q^i}.
\]
On the other hand,
$\xi_{\ua}^{(i)}$ induces an $\FF_q$-morphism $\xi_{\ua}^{(i),*}:
X_{\unfk}\rightarrow \uX_{\unfk}$ with
$\xi_{\ua}^{(i),*}(U_{\unfk}) \subset \uU_{\unfk}$.
The graph of $\xi_{\ua}^{(i),*}$ forms a divisor $Z_{\ua}^{(i)}$ on the surface $X_{\unfk}\times \uX_{\unfk}$,
whose restriction to $U_{\unfk}\times \uU_{\unfk}$
corresponds to the kernel ideal of the $\FF_q$-algebra homomorphism 
\[
{\rm id}_{O_{\unfk}}\otimes \xi_{\ua}^{(i)}: O_{\unfk}\otimes_{\FF_q}\uO_{\unfk} \longrightarrow O_{\unfk}, \quad 
c \otimes 1 \mapsto c,\ 
1 \otimes t \mapsto \theta^{q^i},\ 
1\otimes z \mapsto C_{\ua}(\lambda_{\unfk})^{q^i}.
\]
These divisors enable us to describe the divisors associated with \emph{Coleman functions} introduced in the next subsection.

\subsection{Coleman functions}\label{sec: Col}

Put $\nfk\assign \iota(\unfk) \in A$ and $U'_{\unfk}\assign \Spec(O_{\unfk}[\nfk^{-1}])\subset U_{\unfk}$.
Let $\pi:X_{\unfk}\rightarrow \PP^1$ (resp.~$\upi: \uX_{\unfk} \rightarrow \uPP^1$) be the morphism to the projective $\theta$-line (resp.~$t$-line) corresponding to the inclusion $k\subset K_{\unfk}$ (resp.~$\uk\subset \uK_{\unfk}$).
Let $\infty$ (resp.~$\uinfty$) be the infinite place of $k$ (resp.~$\uk$), corresponding to the infinity point on $\PP^1$ (resp.~$\uPP^1$).
Denote by 
\[
\infty_{\unfk}\assign \sum_{\tilde{\infty} \in \pi^{-1}(\infty)} \tilde{\infty} \quad \in \Div(X_{\unfk})
\hspace{1.5cm} \text{(resp.~$\uinfty_{\unfk}\assign \sum_{\tilde{\uinfty} \in \upi^{-1}(\uinfty)} \tilde{\uinfty} \quad \in \Div(\uX_{\unfk})$)}.
\]
Also, we let $V_{\unfk}\assign U_{\unfk}'\cup \pi^{-1}(\infty) \subset X_{\unfk}$.
By \cite[Theorem 2]{A92}, \cite[Theorem 2.2.4 and 2.2.8]{Sinha97} and \cite[6.3.6 and 6.3.7]{ABP04}, we have that:

\begin{thm}\label{thm: Coleman-fun}
Keep the above notation.
For  $\ua \in \uA$ with $\gcd(\ua,\unfk)=1$, put $a = \iota(\ua) \in A$.
Let $x \in \nfk^{-1} A \setminus A$.
Recall in Section~\ref{Sec: Basic-set} we let $\lbangle ax \rbangle \in k$ be the fractional part of $ax$, i.e.~$|\lbangle ax \rbangle|_{\infty}<1$ and $ax - \lbangle ax \rbangle \in A$.
There exists a unique $g_x \in (O_{\unfk}[\nfk^{-1}])\otimes_{\FF_q}\uO_{\unfk}$ satisfying:
\begin{itemize}
\item[(1)] For every $N \in \ZZ_{\geq 0}$,
\[
g_x^{(N+1)}(\xi_{\ua}) = \prod_{\afk \in A_{+,N}}\left(1+\frac{\lbangle ax \rbangle}{\afk}\right).
\]
Here $A_{+,N}$ consists of all monic polynomials of degree $N$ in $A$; $\xi_{\ua}$ is regarded as an $O_{\unfk}$-valued point on $\uU_{\unfk}$;
and for a pure-tensor $g = c\otimes h \in (O_{\unfk}[\nfk^{-1}])\otimes_{\FF_q}\uO_{\unfk}$, the $s$-th Frobenius twist of $g$ is given by
\[
g^{(s)}\assign c^{q^s}\otimes h, \quad \forall s \in \ZZ_{\geq 0}.
\]
\item[(2)] When regarding as a rational function on $X_{\unfk}\times \uX_{\unfk}$, $g_x$ is regular on $V_{\unfk}\times \uU_{\unfk}$ and
\[
\divv(g_x)\big|_{V_{\unfk}\times \uX_{\unfk}}
=-V_{\unfk}\times \uinfty_{\unfk} + \sum_{\subfrac{\ua \in \uA,\deg \ua <\deg \unfk}{\gcd(\ua,\unfk)=1}}\sum_{N=0}^{\infty}\langle ax\rangle_N \cdot Z_{\ua}^{(N)}\big|_{V_{\unfk}\times\uX_{\unfk}}.
\]
Here for every $f \in k$ and $N\in \ZZ_{\geq 0}$,
\[
\langle f\rangle_N \assign 
\begin{cases}
1, & \text{ if\quad $\inf\big\{|f-b-\theta^{-N-1}|_{\infty}\ \big|\ b \in A\big\}<|\theta|_{\infty}^{-N-1}$;} \\
0, & \text{ otherwise.}
\end{cases}
\]
\end{itemize}
\end{thm}

\begin{rem}\label{rem: good-red}
Suppose $v\nmid \nfk$. Then $v$ is unramified in $K_{\unfk}$ (see \cite[Theorem 12.8]{Ros02}). Let $w$ be the place of $K_{\unfk}$ lying above $v$ which corresponds to the restriction of the absolute value $|\cdot|_v$ on $\CC_v$, and $K_{\unfk,w}$, $O_{\unfk,w}$ be the closure of $K_{\unfk}$ and $O_{\unfk}$ in $\CC_v$, respectively.
Let $\pfk_w$ be the maximal ideal of $O_{\unfk,w}$ and $\FF_w:=O_{\unfk,w}/\pfk_w$ be the residue field of $O_{\unfk,w}$.
The description of the divisor of $g_x$ restricting to $V_{\unfk}\times \uX_{\unfk}$ in Theorem~\ref{thm: Coleman-fun} (2) says in particular that the function 
\[
\bar{g}_x\assign g_x \bmod \pfk_w \quad \in \FF_w\otimes_{O_{\unfk}[\nfk^{-1}]} \Big(O_{\unfk}[\nfk^{-1}] \otimes_{\FF_q}\uO_{\unfk}\Big) = \FF_w\otimes_{\FF_q}\uO_{\unfk},
\]
which is the reduction of $g_x$ modulo $\pfk_w$, is nonzero.
Moreover, take $\ua \in \uA$ with $\gcd(\ua,\unfk)=1$.
Let $\ell$ be the order of $v \bmod \nfk$ in $(A/\nfk)^{\times}$.
Identifying $\FF_w$ as a subfield of $\CC_v$ via the Teichm\"uller embedding in \eqref{eqn: Teichmuller}, by \cite[Remark~3.3.3]{Ch25} the special value $\bar{g}^{(s+1)}_x(\xi_{\ua})$ for $s \in \ZZ_{\geq 0}$ coincides with the  ``geometric Gauss sum'' (introduced in \cite[Definition~3.1.1 and the Frobenius twists in (5.3)]{Ch25}), i.e.
\begin{equation}\label{eqn: Geo-GS}
\bar{g}_x^{(s+1)}(\xi_{\ua}) = \Gcal_{ax}^{\geo,(s)}
\assign 1+ \sum_{u \in \FF_w^{\times}} \omega_C\big(\bar{C}_{\underline{(v^{\ell}-1)ax}}(u^{-1})\big)\cdot \eT(u)^{q^{s}}.
\end{equation}
Here 
\begin{itemize}
    \item $a=\ua(\theta)=\iota(\ua) \in A$; \item $\omega_C: \FF_w \stackrel{\sim}{\rightarrow}\uC[\uv^{\ell}-1]\subset \bar{k}$ is the ``geomtric Teichm\"uller character'' (see \cite[(5.2)]{Ch25});
    \item $\bar{C}$ is the reduction of the Carlitz module modulo $\pfk_w$, i.e.~
\[
\bar{C}_{\ub}(c \bmod \pfk_w) \assign C_{\ub}(c) \bmod \pfk_w,
\quad \forall c \bmod \pfk_w \in \frac{O_{\unfk,w}}{\pfk_w} = \FF_w;
\] 
\item $\underline{(v^{\ell}-1)ax} = \iota^{-1}((v^{\ell}-1)ax) \in \uA$;
\item $\eT: \FF_w \hookrightarrow O_{\unfk,w}$ is the (usual) Teichm\"uller embedding given in \eqref{eqn: Teichmuller}.
\end{itemize}
Note that by \cite[Proposition 3.2.1 (2)]{Ch25} we know that $\Gcal_{vax}^{\geo,(s+d)} = \Gcal_{ax}^{\geo,(s)}$, which says that 
\begin{equation}\label{eqn: bar-g-inv}
\bar{g}_x^{(s+d+1)}(\xi_{\uv\ua})=
\bar{g}_{x}(\xi_{\ua})^{(s+1)}, \quad \forall s \in \ZZ_{\geq 0}.
\end{equation}
Moreover, as a consequence of the Gross--Koblitz--Thakur formula in \cite[Theorem~5.1.1]{Ch25}, we have that 
\begin{equation}\label{eqn: Gsum-kv}
\bar{g}^{(s+1)}_x(\xi_{\ua}) \in k_v^{\times} \cap \bar{k}^{\times}
\quad \text{ for every $s \in \ZZ_{\geq 0}$.}
\end{equation}
\end{rem}

\subsection{ Soliton \texorpdfstring{$t$}{t}-motives and their \texorpdfstring{$v$}{v}-adic periods}\label{sec: v-period}

Recall that $X_{\unfk}\times \uX_{\unfk}$ is a smooth projective surface over $\FF_q$. Thus every Weil divisor $D$ is a Cartier divisor, and the corresponding sheaf $\Ocal_{X_{\unfk}\times \uX_{\unfk}}(D)$ is locally free of rank one over the structure sheaf $\Ocal_{X_{\unfk}\times \uX_{\unfk}}$ of the surface $X_{\unfk}\times \uX_{\unfk}$.

Recall that the \emph{geometric diamond bracket on $k$} is given by (see \cite[5.5.1.]{ABP04})
\[
\langle f\rangle_{\geo} \assign \sum_{N=0}^{\infty} \langle f \rangle_N \ \in \{0,1\}, \quad \forall\ f \in k.
\]
In particular $\langle f+a\rangle_{\geo} = \langle f \rangle_{\geo}$ for every $a \in A$, and (see \cite[5.5.5.]{ABP04})
\begin{align}\label{eqn: geo-bra-ref}
\sum_{\epsilon \in \FF_q^{\times}} \langle \epsilon f \rangle_{\geo} & = \begin{cases}   
1, & \text{ if $f \notin A$;}\\
0, & \text{ otherwise.}
\end{cases}
\end{align}
Given $x \in \nfk^{-1}A \setminus A$,
let 
\[
W_x \assign
\sum_{\subfrac{\ua \in \uA,\deg \ua <\deg \unfk}{\gcd(\ua,\unfk)=1}}\sum_{N=1}^{\infty}\langle ax\rangle_N \cdot 
\sum_{i=0}^{N-1}Z_{\ua}^{(i)}
\quad \text{ and }\quad 
\tilde{\Xi}_x
\assign 
\sum_{\subfrac{\ua \in \uA,\deg \ua <\deg \unfk}{\gcd(\ua,\unfk)=1}}\langle ax\rangle_{\geo} \cdot Z_{\ua},
\]
Then $W_{x+a} = W_{x}$ and $\tilde{\Xi}_{x+a} = \tilde{\Xi}_{x}$ for every $a \in A$. Moreover, Theorem~\ref{thm: Coleman-fun}~(2) can be rewritten as (cf.~\cite[6.3.9]{ABP04})
\[
\divv(g_x)\big|_{V_{\unfk}\times \uX_{\unfk}}
=-V_{\unfk}\times \uinfty_{\unfk} + \tilde{\Xi}_x\big|_{V_{\unfk}\times\uX_{\unfk}}+W^{(1)}_x\big|_{V_{\unfk}\times\uX_{\unfk}}-
W_x\big|_{V_{\unfk}\times\uX_{\unfk}}.
\]
\begin{defn}\label{defn: Soliton M-x}
(cf.~\cite[Corollary 3.3.8.1]{Sinha97}, \cite[4.2.2]{BP02}, and \cite[6.4.2]{ABP04}) Define
\[
\Mcal_x \assign \Ocal_{X_{\unfk}\times \uX_{\unfk}}(W_x)(U'_{\unfk}\times \uU_{\unfk}),
\] 
which is a locally free of rank one module over $O_{\unfk}[\nfk^{-1}]\otimes_{\FF_q}\uO_{\unfk}$ equipped with a $\tau$-action defined by
\[
\tau \cdot h \assign g_x \cdot h^{(1)}, \quad \forall h \in \Mcal_x.
\]
The \emph{soliton $t$-motive associated with $x$ over $K_{\unfk}$} is given by
\[
M_x \assign K_{\unfk} \otimes_{O_{\unfk}[\nfk^{-1}]} \Mcal_x,
\]
which is a locally free module of rank one over $K_{\unfk}\otimes_{\FF_q}\uO_{\unfk}$ with the extended $\tau$-action.
\end{defn}

\begin{rem}\label{rem: prop-Mx}
${}$
\begin{itemize}
    \item[(1)] $M_x$ is actually a \emph{``CM $t$-motive with CM type $(\uK_{\unfk},\Xi_{x})$''},
    where 
    \[
    \Xi_{x} = \sum_{\subfrac{\ua \in A,\, \deg \ua < \deg \unfk}{\gcd(\ua,\unfk)=1}}\langle ax \rangle_{\geo}\cdot \xi_{\ua}.
    \]
    We refer the readers to \cite[Section 3]{BCPW22} and \cite[Lemma 5.3.1]{Wei26} for further details.
    \item[(2)] Observe that $\Mcal_{x+a} = \Mcal_x$ and $M_{x+a} = M_x$ for every $a \in A$.
    \item[(3)] Suppose $v \nmid \nfk$ as above.
Keep the notation as in Remark~\ref{rem: good-red}.
Then
\[
\Mcal_{x,w}\assign O_{\unfk,w}\otimes_{O_{\unfk}[\nfk^{-1}]} \Mcal_x \quad \text{ is a model $M_x$ at $w$.}
\]
Since $\bar{g}_x \neq 0$ by \eqref{eqn: Gsum-kv}, the  map $\tau:\overline{\Mcal}_{x,w}\rightarrow \overline{\Mcal}_{x,w}$ is injective, whence $\Mcal_{x,w}$ is a good model of $M_x$ at $w$.
\end{itemize}
\end{rem}

In order to study the $v$-isocrystal periods of $M_x$, we first provide a concrete description of the isomorphism $\delta_{\Mcal_{x,w}}$ in Lemma~\ref{lem: key-iso}.
Observe that for every $i \in \ZZ_{\geq 0}$, we have
\[
g_x^{(i)} \in O_{\unfk}[\nfk^{-1}]\otimes_{\FF_q} \uO_{\unfk} \subset O_{\unfk,w}\otimes_{\FF_q} \uO_{\unfk}
\quad \text{ and } \quad 
\bar{g}_x^{(i)} \in \Big(\big(\FF_w \otimes_{\FF_q} \uO_{\unfk} \big)\setminus \{0\}\Big) \subset \Big(\FF_w(\!(\uv)\!)\otimes_{\FF_q[t]}\uO_{\unfk}\Big)^{\times}.
\]
Hence $\displaystyle \frac{g_x^{(i)}}{\bar{g}_x^{(i)}} \in O_{\unfk,w}[\![\uv,\uv^{-1}\} \otimes_{\FF_q[t]} \uO_{\unfk}$ is congruent to $1$ modulo $\pfk_w$.
Define
\[
\psi_x\assign
\prod_{i=0}^{\infty}\frac{g_x^{(i)}}{\bar{g}_x^{(i)}}, \quad \text{ which converges in }\   O_{\unfk,w}[\![\uv, \uv^{-1}\}
\otimes_{\FF_q[t]}\uO_{\unfk}.
\]
\begin{lem}\label{lem: psi-x}
Keep the notation as above. We have
\[
\psi_x \in \Big(O_{\unfk,w}[\![\uv,\uv^{-1}\}[\omega_v(\uv)^{-1}]\otimes_{\FF_q[t]}\uO_{\unfk}\Big)^{\times}.
\]
\end{lem}

\begin{proof}
Recall that $U_{\unfk}' = \Spec(O_{\unfk}[\nfk^{-1}])$. Observe that for $N \in \ZZ_{\geq 0}$, 
\[
\divv(t-\theta^{q^N})\big|_{U'_{\unfk}\times \uX_{\unfk}} = -(q-1)\cdot U'_{\unfk}\times \uinfty_{\unfk} + \sum_{\subfrac{\ua \in \uA,\, \deg \ua < \deg \nfk}{\gcd(\ua,\nfk)=1}} Z_{\ua}^{(N)}\big|_{U'_{\unfk}\times \uX_{\unfk}}.
\]
Put
$\phi(\nfk)\assign \#((A/\nfk)^{\times})$
and 
$\phi^+_i(\nfk)\assign 
\#\{a \in A_{+,i}\mid \text{gcd}(a,\nfk)=1\}$ for $0\leq i <\deg \nfk$.
One has $\sum_{i=0}^{\deg \nfk-1} \phi_i^+(\nfk) = \phi(\nfk)/(q-1)$. Suppose $x= x'/\mfk$ with $\mfk \in A_+$, $\mfk\mid \nfk$, and $\gcd(x',\mfk)=1$.
Then Theorem~\ref{thm: Coleman-fun}~(2) implies that 
\begin{align*}
& \hspace{-1cm} \sum_{b \in (A/\nfk)^{\times}}\divv(g_{bx})\Big|_{U'_{\unfk}\times \uX_{\unfk}} \\
& = -\phi(\nfk)\cdot U'_{\unfk}\times \uinfty_{\unfk}
+
\sum_{N=0}^{\infty}
\sum_{\subfrac{\ua \in \uA,\deg \ua < \deg \unfk}{\gcd(\ua,\unfk)=1}}\sum_{b \in (A/\nfk)^{\times}}\big\langle \frac{abx'}{\mfk}\big\rangle_N \cdot Z_{\ua}^{(N)}\big|_{U_{\unfk}'\times \uX_{\unfk}} \\
& = \frac{\phi(\nfk)}{\phi(\mfk)}
\left(
-\phi(\mfk)\cdot U_{\unfk}'\times \uinfty_{\unfk}
+ 
\sum_{i=0}^{\deg \mfk-1} \phi_{i}^+(\mfk)\cdot 
\sum_{\subfrac{\ua \in \uA,\deg \ua < \deg \unfk}{\gcd(\ua,\unfk)=1}} Z_{\ua}^{(\deg \mfk -i - 1)}\big|_{U_{\unfk}'\times \uX_{\unfk}}
\right)\\
& = \frac{\phi(\nfk)}{\phi(\mfk)} \cdot \sum_{i=0}^{\deg \mfk -1} \phi_{i}^+(\mfk) \cdot \divv(t-\theta^{q^{\deg \mfk -i -1}})\Big|_{U'_{\unfk}\times \uX_{\unfk}}.
\end{align*}
Thus there exists $u \in O_{\unfk}[\nfk^{-1}]^{\times}$ such that 
\[
\prod_{b \in (A/\nfk)^{\times}}g_{bx}
= u \cdot \left(
\prod_{i=0}^{\deg \mfk -1}
(t-\theta^{q^{\deg \mfk - i - 1}})^{\phi^+_i(\mfk)}\right)^{\phi(\nfk)/\phi(\mfk)}.
\]
Take $u_0 \in \FF_w^{\times}$ such that $u-u_0 \in \pfk_w$, and let $u_1 = u/u_0$, which is a $1$-unit in $O_{\unfk,w}^{\times}$. We then obtain that
\[
\prod_{b \in (A/\nfk)^{\times}}\bar{g}_{bx} = u_0\cdot \left(
\prod_{i=0}^{\deg \mfk -1}
(t-\varepsilon_v^{q^{\deg \mfk -i-1}})^{\phi_i^+(\mfk)}
\right)^{\phi(\nfk)/\phi(\mfk)},
\]
whence
\begin{align*}
\prod_{b \in (A/\nfk)^{\times}}\psi_{bx} = \prod_{b \in (A/\nfk)^{\times}}\prod_{i=0}^{\infty}\frac{g_{bx}^{(i)}}{\bar{g}_{bx}^{(i)}}
& = \ \prod_{i=0}^{\infty}u_1^{q^i}
\cdot 
\left(
\prod_{i=0}^{\deg \mfk -1}
\omega_{\theta,v}^{(\deg \mfk - i - 1)}(\uv)^{\phi^+_i(\mfk)}\right)^{\phi(\nfk)/\phi(\mfk)} \\
& \hspace{3cm} \in
\Big(O_{\unfk,w}[\![\uv,\uv^{-1}\}[\omega_v(\uv)^{-1}]\otimes_{\FF_q[t]}\uO_{\unfk}\Big)^{\times}.
\end{align*}
Therefore the desired result holds.
\end{proof}

Note that $W_x$ is an effective divisor, which means that $O_{\unfk}[\nfk^{-1}]\otimes_{\FF_q}\uO_{\unfk} \subset \Mcal_x$.
On the other hand, for each $\ua \in \uA$ with $\gcd(\ua,\unfk)=1$ and $i \in \ZZ_{\geq 0}$, the ideal of $O_{\unfk}\otimes_{\FF_q} \uO_{\unfk}$ corresponding to the divisor $Z_{\ua}^{(i)}\big|_{U_{\unfk}\times \uU_{\unfk}}$ contains $(t-\theta^{q^i})$, from which one has
\[
\left(\prod_{\subfrac{a \in A,\, \deg a < \deg \nfk}{\gcd(a,\nfk)=1}} \prod_{N=1}^{\infty} \prod_{i=0}^{N-1} (t-\theta^{q^i})^{\langle ax\rangle_N}\right) \cdot \Mcal_x \subset O_{\unfk}[\nfk^{-1}]\otimes_{\FF_q}\uO_{\unfk}.
\]
Consider the following commutative diagram
\[
\xymatrix{
\uO_{\unfk} \ar@{^{(}->}[r] \ar@{=}[d] & O_{\unfk,w}\otimes_{\FF_q} \uO_{\unfk} \ar@{^{(}->}[r] \ar[d]_{\bmod \pfk_w} & \Mcal_{x,w}  \ar[d]^{\bmod \pfk_w} \\ 
\uO_{\unfk} \ar@{^{(}->}[r] & \FF_w \otimes_{\FF_q}\uO_{\unfk} \ar@{^{(}->}[r] & \overline{\Mcal}_{x,w}.
}
\]
As $(t-\theta^{q^i}) \in O_{\unfk,w}[\![\uv,\uv^{-1}\}[\omega_v(\uv)^{-1}]^{\times}$ for every $i \in \ZZ_{\geq 0}$ by \eqref{eqn: t-theta-inv} and 
\[
\FF_w(\!(\uv)\!)\otimes_{\FF_q[t]} \uO_{\unfk} = \FF_w(\!(\uv)\!) \otimes_{\FF_w[t]} \overline{\Mcal}_{x,w}, \quad \FF_w(\!(\uv)\!) \subset O_{\unfk,w}[\![\uv,\uv^{-1}\},
\]
tensoring the above two horizontal inclusions by $O_{\unfk,w}[\![\uv,\uv^{-1}\}[\omega_v(\uv)^{-1}]$ induces the following two isomorphisms
\begin{equation}\label{eqn: On-Mw-1}
O_{\unfk,w}[\![\uv,\uv^{-1}\}[\omega_{v}(\uv)^{-1}]\otimes_{\FF_q[t]}\uO_{\unfk} \cong 
O_{\unfk,w}[\![\uv,\uv^{-1}\}[\omega_{v}(\uv)^{-1}]\otimes_{O_{\unfk,w}[t]} \Mcal_{x,w}
\end{equation}
and
\begin{equation}\label{eqn: On-Mw-2}
O_{\unfk,w}[\![\uv,\uv^{-1}\}[\omega_{v}(\uv)^{-1}]\otimes_{\FF_q[t]}\uO_{\unfk} \cong 
O_{\unfk,w}[\![\uv,\uv^{-1}\}[\omega_{v}(\uv)^{-1}]\otimes_{\FF_w[t]} \overline{\Mcal}_{x,w}.
\end{equation}
In other words, the isomorphism $\delta_{\Mcal_{x,w}}$ is uniquely determined by the image of elements in $\uO_{\unfk}$, which can be described as follows:
\begin{prop}\label{prop: sp-iso-x}
For each $h \in \uO_{\unfk} \subset \Mcal_{x,w}$, we have 
\[
\delta_{\Mcal_{x,w}}(h) = \psi_x^{-1}\cdot h
\ \in O_{\unfk,w}[\![\uv,\uv^{-1}\}[\omega_v(\uv)^{-1}]\otimes_{\FF_w[t]}
\overline{\Mcal}_{x,w}.
\]
\end{prop}

\begin{proof}
As 
\begin{align*}
\psi_x^{-1}\cdot (\tau^d h) = \psi_x^{-1}\cdot (\prod_{i=0}^{d-1}g_x^{(i)}) \cdot h
&= (\prod_{i=0}^{d-1}\bar{g}_x^{(i)}) \cdot (\psi_x^{(d)})^{-1} h \\
& = \tau^d \cdot \big(\psi_x^{-1} h)\ \in 
O_{\unfk,w}[\![\uv,\uv^{-1}\}[\omega_v(\uv)^{-1}]\otimes_{\FF_w[t]}
\overline{\Mcal}_{x,w},
\end{align*}
The desired result holds.
\end{proof}
Put
\begin{equation}\label{eqn:g-x-v}
g_{x,v}\assign \prod_{i=0}^{d-1}g_{x}^{(i)} \quad \text{ and } \quad 
\bar{g}_{x,v}\assign \prod_{i=0}^{d-1}\bar{g}_{x}^{(i)}.
\end{equation}

\begin{cor}
Through~\eqref{eqn: On-Mw-1} and \eqref{eqn: On-Mw-2} we identify $H_{\dR}(M_{x},K_{\unfk,w})$ and $H_{\cris}(M_{x},K_{\unfk,w})$ with
\[
\frac{g_{x,v}\cdot K_{\unfk,w}\otimes_{\FF_q}\uO_{\unfk}}{(t-\theta)g_{x,v}\cdot K_{\unfk,w}\otimes_{\FF_q}\uO_{\unfk}}
\quad \text{ and }\quad 
\frac{\bar{g}_{x,v}\cdot K_{\unfk,w}\otimes_{\FF_q}\uO_{\unfk}}{(t-\theta)\bar{g}_{x,v}\cdot K_{\unfk,w}\otimes_{\FF_q}\uO_{\unfk}}, \quad \text{respectively.}
\]
Then the de~Rham--crystalline isomorphism $\Phi_{M_{x},w}$ is given by
\[
\Phi_{M_{x},w}\big(g_{x,v}\cdot c\otimes h \bmod (t-\theta)\big) = (\psi_{x}^{(d)})^{-1} \cdot \bar{g}_{x,v} \cdot c\otimes h \bmod (t-\theta), \quad \forall c \otimes h \in K_{\unfk,w}\otimes_{\FF_q}\uO_{\unfk}.
\]
\end{cor}

Now, identify $K_{\unfk,w}\otimes_{\FF_q} \uO_{\unfk}$ with
\[
\frac{K_{\unfk,w}[t,z]}{(C^*_{\unfk}(t,z))},
\quad \text{ and let }\ 
\bar{z} = z + (C^*_{\unfk}(t,z)) \ \in \frac{K_{\unfk,w}[t,z]}{(C^*_{\unfk}(t,z))}.
\]
Given $\ua \in \uA$ coprime to $\unfk$, define
\begin{equation}\label{eqn: e-ua}
e_{\ua} \assign 
\prod_{\subfrac{\ub \in (\uA/\unfk)^{\times}}{\ub \not\equiv \ua \bmod \unfk}}
\frac{\bar{z}-C_{\ub}(\lambda_{\unfk})}{C_{\ua}(\lambda_{\unfk})-C_{\ub}(\lambda_{\unfk})} \quad \in \frac{K_{\unfk,w}[t,z]}{(C_{\unfk}^*(t,z))} \cong K_{\unfk,w}\otimes_{\FF_q}\uO_{\unfk},
\end{equation}
which satisfies $e_{\ua}^2 = e_{\ua}$ and $\ualpha \cdot e_{\ua} = \xi_{\ua}(\ualpha) \cdot e_{\ua}$ for every $\ualpha \in \uO_{\unfk}$, and let
\begin{align}\label{eqn: omega-a}
\omega_{\ua}^x & \assign 
\left(\prod_{N=1}^{d-1} g_x^{(N)}(\xi_{\ua})\right)^{-1} \cdot g_{x,v} \cdot e_{\ua} \ \bmod (t-\theta) \\ 
& \hspace{2cm}\in  
\frac{g_{x,v}\cdot K_{\unfk}\otimes_{\FF_q}\uO_{\unfk}}{(t-\theta)g_{x,v}\cdot K_{\unfk}\otimes_{\FF_q}\uO_{\unfk}} = H_{\dR}(M_{x},K_{\unfk}). \notag
\end{align}
Then $\{\omega_{\ua}^x\mid \ua \in (\uA/\unfk)^{\times}\}$ forms a basis for $H_{\dR}(M_x,K_{\unfk})$ and
\[
\ualpha_{\dR}(\omega_{\ua}^x) = \xi_{\ua}(\ualpha)\cdot \omega_{\ua}^x, \quad \forall \ualpha \in \uO_{\unfk}.
\]
Here we regard $\ualpha \in \uO_{\unfk}$ as an endomorphism of $M_x$ (by natural multiplication), and $\ualpha_{\dR}$ is the induced map on $H_{\dR}(M_x,K_{\unfk})$.
Moreover, the crystalline action of a Frobenius automorphism $\sigma_v \in \Wcal_{v}$ on this basis can be illustrated as follows:
\begin{prop}\label{prop: v-iso-period-Mx}
Given $\sigma_v \in \Wcal_v$ with $\deg \sigma_v =1$, we have
\[
\sigma_{v,\cris}(\omega_{\ua}^x)
=\left(\prod_{N=1}^{d}g_x^{(N)}(\xi_{\uv \ua})\right) \cdot \frac{\psi_{x}^{(d+1)}(\xi_{\uv \ua})}{\psi_{x}^{(1)}(\xi_{\ua})} \cdot \omega_{\uv \ua}^x, \quad \forall \ua \in (\uA/\unfk)^{\times}.
\]    
\end{prop}

\begin{proof}
Let 
\[
\bar{\omega}_{\ua}^x \assign 
\left(\prod_{N=1}^{d-1} g_x^{(N)}(\xi_{\ua})\right)^{-1} \cdot \bar{g}_{x,v} \cdot e_{\ua} \ \bmod (t-\theta) \ \in  
\frac{\bar{g}_{x,v}\cdot K_{\unfk}\otimes_{\FF_q}\uO_{\unfk}}{(t-\theta)\bar{g}_{x,v}\cdot K_{\unfk}\otimes_{\FF_q}\uO_{\unfk}} = H_{\cris}(M_x,K_{\unfk}).
\]
Then 
\begin{equation}\label{eqn: Phi-M-x}
\Phi_{M_x,w}(\omega_{\ua}^x) = \psi_{x}^{(d)}(\xi_{\ua})^{-1} \cdot \bar{\omega}_{\ua}^x.
\end{equation}
Moreover, from the definition of the $\sigma_v$-linear map $\tilde{\sigma}_{v,\cris}$ on $H_{\cris}(M_x,\bar{k}_v)$ in \eqref{eqn: t-sigma-cris}, we get
\begin{equation}\label{eqn: v-cris-M-x}
\tilde{\sigma}_{v,\cris}(\bar{\omega}_{\ua}^x) = 
\left(\frac{\prod_{N=1}^{d-1} g_x^{(N)}(\xi_{\uv\ua})}{\prod_{N=1}^{d-1} g_x^{(N)}(\xi_{\ua})}\right) \cdot
\bar{g}_{x,v}^{(d)}(\xi_{\uv\ua}) \cdot \bar{\omega}_{\uv \ua}^x.
\end{equation}
Note that from Theorem~\ref{thm: Coleman-fun}~(1) and \eqref{eqn: Gsum-kv} one has that 
\[
g_x^{(N+1)}(\xi_{\ua})\ \text{ and } \ \bar{g}_x^{(N+1)}(\xi_{\ua})
\ \text{ are in } k_v^{\times} \ \text{ for every $N \in \ZZ_{\geq 0}$},
\]
which implies that
$\psi_x^{(d)}(\xi_{\ua}) \in k_v^{\times}$.
As $\bar{g}^{(d)}_{x,v}(\xi_{\uv\ua}) = \bar{g}^{(d)}_x(\xi_{\uv\ua})\cdot  \prod_{N=1}^{d-1}\bar{g}^{(N)}_{x}(\xi_{\ua})$ by \eqref{eqn: bar-g-inv}, we obtain that
\begin{align*}
\sigma_{v,\cris}(\omega_{\ua}^x)
& = \Phi_{M_x,w}^{-1}\circ \tilde{\sigma}_{v,\cris} \circ \Phi_{M_x,w}(\omega_{\ua}^x) \\
& = \Phi_{M_x,w}^{-1}\Big(\tilde{\sigma}_{v,\cris}\big(\psi^{(d)}_x(\xi_{\ua})^{-1} \cdot \bar{\omega}_{\ua}^x\big)\Big) \\
& = \psi^{(d)}_x(\xi_{\ua})^{-1} \cdot 
\Phi_{M_x,w}^{-1}\left(
\left(\frac{\prod_{N=1}^{d-1} g_x^{(N)}(\xi_{\uv\ua})}{\prod_{N=1}^{d-1} g_x^{(N)}(\xi_{\ua})}\right) \cdot
\bar{g}^{(d)}_{x,v}(\xi_{\uv\ua}) \cdot \bar{\omega}_{\uv \ua}^x
\right) \\
& = \frac{\psi_x^{(d)}(\xi_{\uv\ua})}{\psi_{x}^{(d)}(\xi_{\ua})} \cdot 
\left(\frac{\prod_{N=1}^{d-1} g_x^{(N)}(\xi_{\uv\ua})}{\prod_{N=1}^{d-1} g_x^{(N)}(\xi_{\ua})}\right) \cdot
\left(\bar{g}_{x}^{(d)}(\xi_{\uv\ua})\cdot \prod_{N=1}^{d-1}\bar{g}_x^{(N)}(\xi_{\ua})\right) \cdot \omega_{\uv\ua}^x\\
& =
\left(\prod_{N=1}^{d}g_x^{(N)}(\xi_{\uv \ua})\right) \cdot \frac{\psi_{x}^{(d+1)}(\xi_{\uv \ua})}{\psi_{x}^{(1)}(\xi_{\ua})} \cdot \omega_{\uv \ua}^x
\end{align*}
as desired.
\end{proof}

\subsection{General soliton \texorpdfstring{$t$}{t}-motives}
\label{sec: M-boxx}
Let $\Ascr^{\geo}_{\unfk}$ be the free abelian group generated by $x \in \nfk^{-1} A$.
Given $\boxx = \sum_{x} n_x[x] \in \Ascr^{\geo}_{\unfk}$, the weight of $\boxx$ is given by (see \cite[6.1.1]{ABP04})
\[
\wt(\boxx)\assign
\frac{1}{q-1}\sum_{x \in \nfk^{-1}A\setminus A} n_x \quad \in \frac{1}{q-1}\ZZ.
\]
We say that 
$\boxx$ is \emph{effective} if $n_x \geq 0$.
Suppose $\boxx$ is effective with $\wt(\boxx)>0$,
define
\[
W_{\boxx}\assign \sum_{x \in \nfk^{-1}A\setminus A} n_x W_x, \quad g_{\boxx}\assign \prod_{x \in \nfk^{-1}A\setminus A}g_x^{n_x},
\quad \text{ and } \quad 
\Mcal_{\boxx}\assign \Ocal_{X_{\unfk}\times \uX_{\unfk}}(W_{\boxx})(V_{\unfk}\times \uU_{\unfk}),
\]
a locally free of rank one module over $O_{\unfk}[\nfk^{-1}]\otimes_{\FF_q}\uO_{\unfk}$, which is equipped with a $\tau$-action defined by
\[
\tau \cdot h \assign g_{\boxx} \cdot h^{(1)}, \quad \forall h \in \Mcal_{\boxx}.
\]
The \emph{soliton $t$-motive associated with $\boxx$ over $K_{\unfk}$} is $M_{\boxx}\assign K_{\unfk}\otimes_{O_{\unfk}[\nfk^{-1}]} \Mcal_{\boxx}$.

\begin{rem}
Write $\boxx = \sum_{i=1}^d n_i [x_i] + \sum_{i'=1}^{d'} n_i'[a_i]$ with $x_i \notin A$ and $a_i \in A$. 
When $\boxx$ is effective with $\wt(\boxx)>0$,
observe that $\Mcal_{\boxx}$ is isomorphic to
\[
\Mcal_{x_1}^{\otimes_{\unfk} n_1} \otimes_{\unfk} \cdots \otimes_{\unfk} \Mcal_{x_d}^{\otimes_{\unfk} n_d}, 
\quad \text{ where } \quad 
\Mcal_{x}\otimes_{\unfk} \Mcal_{x'} \assign M_{x'}\otimes_{O_{\unfk}[\nfk^{-1}]\otimes_{\FF_q}\uO_{\unfk}} \Mcal_{x'}
\]
with the $\tau$-action given by
\[
\tau \cdot (m\otimes m') \assign (\tau m\otimes \tau m'), \quad \forall m \in M_{x} \text{ and } m' \in M_{x'}.
\]
\end{rem}
Define
\[
\psi_{\boxx} \assign \prod_{i=1}^d\psi_{x_i}^{n_i} \quad \in \
O_{\unfk,w}[\![\uv, \uv^{-1}\}
\otimes_{\FF_q[t]}\uO_{\unfk} \cap 
\Big(O_{\unfk,w}[\![\uv,\uv^{-1}\}[\omega_v(\uv)^{-1}]\otimes_{\FF_q[t]}\uO_{\unfk}\Big)^{\times}
\]
Following the same approach (or, alternatively, regarding $\Mcal_{\boxx}$ as the tensor product in the above),
we can find a $K_{\unfk}$-basis $\{ \omega_{\ua}^{\boxx}\mid \ua \in (\uA/\unfk)^{\times}\}$ for $H_{\dR}(M_{\boxx},K_{\unfk})$
with
\[
\ualpha_{\dR}(\omega_{\ua}^{\boxx}) = \xi_{\ua}(\ualpha)\cdot \omega_{\ua}^{\boxx}, \quad \forall \ualpha \in \uO_{\unfk},
\]
and the $v$-isocrystal action is realized as follows:

\begin{prop}\label{prop: v-iso-period-Mboxx}
Given an effective $\boxx \in \Ascr_{\unfk}^{\geo}$ with $\wt(\boxx)>0$,
\[
\sigma_{v,\cris} (\omega_{\ua}^{\boxx})
=\left(\prod_{N=1}^{d}g_{\boxx}^{(N)}(\xi_{\uv \ua})\right) \cdot \frac{\psi_{\boxx}^{(d+1)}(\xi_{\uv \ua})}{\psi_{\boxx}^{(1)}(\xi_{\ua})} \cdot \omega_{\uv \ua}^{\boxx}, \quad \forall \ua \in (\uA/\unfk)^{\times}.
\]    
\end{prop}

\begin{rem}\label{rem: period-omega-a}
From Theorem~\ref{thm: Coleman-fun} (1) and equation \eqref{eqn: bar-g-inv}
we know that 
\begin{align*}
\left(\prod_{N=1}^{d}g_x^{(N)}(\xi_{\uv \ua})\right) \cdot \frac{\psi_{x}^{(d+1)}(\xi_{\uv \ua})}{\psi_{x}^{(1)}(\xi_{\ua})}
&= \left(\prod_{N=0}^{d-1}g_x^{(N+1)}(\xi_{\uv \ua})\right) \cdot
\prod_{N=d}^{\infty} \left(\frac{g_{x}^{(N+1)}(\xi_{\uv\ua})\big/\bar{g}_{x}^{(N+1)}(\xi_{\uv\ua}) }{g_{x}^{(N+1-d)}(\xi_{\ua})\big/\bar{g}_{x}^{(N+1-d)}(\xi_{\ua})}\right) \\ 
&=
\prod_{N=0}^{d-1} \prod_{\afk \in A_{+,N}}\left(1+\frac{\lbangle vax\rbangle}{\afk}\right) \cdot \prod_{N=d}^{\infty}
\frac{\displaystyle
\prod_{\afk \in \uA_{+,N}}\left(1+\frac{\lbangle vax\rbangle }{\afk}\right)
}{\displaystyle
\prod_{\afk \in \uA_{+,N-d}}\left(1+\frac{\lbangle ax\rbangle}{\afk}\right)
}.
\end{align*}
We shall connect this infinite product with the special value of $v$-adic geometric gamma function in the next section.
\end{rem}

\section{\texorpdfstring{$v$}{v}-adic gamma values and their period interpretations}
\label{Sec: v-gamma-period}

\subsection{ \texorpdfstring{$v$}{v}-adic geometric gamma function}\label{sec: v-gamma}

Given $a \in A_v$, put
\[
a_{\flat}\assign
\begin{cases}
    a, & \text{ if $\ord_v(a)=0$,} \\
    1, & \text{ otherwise.}
\end{cases}
\]
The $v$-adic geometric gamma function (introduced by Thakur \cite[Section 5.9]{Tha91}) is defined by
\[
\Gamma_v^{\geo}(x)\assign x_{\flat}^{-1} \cdot \prod_{\afk \in A_+}\frac{\afk_{\flat}}{(x+\afk)_{\flat}} \ \in k_v^{\times}, \quad \forall x \in A_v.
\]
Similar to the $\infty$-adic case, it has the following three standard functional equations:
\begin{prop}\label{prop: FE}
\begin{itemize}
    \item[(1)] (Translation.) For every $a \in A$,
    \[
    \frac{\Gamma_v^{\geo}(x+a)}{\Gamma_v^{\geo}(x)}
    = \frac{x_{\flat}}{(x+a)_{\flat}} \cdot \prod_{N=0}^{\deg a}
    \prod_{\afk \in A_{+,N}} \frac{(x+\afk)_{\flat}}{(x+a+\afk)_{\flat}}.
    \]
    \item[(2)] (Reflection, see \cite[Section 6.4]{Tha91}.)
    \[
    \prod_{\epsilon \in \FF_q^{\times}}\Gamma_v^{\geo}(\epsilon x) = \zeta \cdot x_{\flat}^{2-q}, \quad \text{ where } \quad \zeta \in \FF_q^{\times}.
    \]
    \item[(3)]
    (Multiplication, see \cite[Section 6.3]{Tha91}.) Given $\mfk \in A_+$ with $v \nmid \mfk$,
    \[
    \prod_{\subfrac{a \in A}{\deg a <\deg \mfk}} \Gamma_v^{\geo}(x+\frac{a}{\mfk}) = \Gamma_v^{\geo}(\mfk x) \cdot Q(x,\mfk) \quad \text{ where } \quad Q(x,\mfk) \in k(x)^{\times}.
    \]
\end{itemize}
\end{prop}

Given $x \in A_v$, let $\partial_v x \in A_v$ and $x_{\circ} \in A$ with $\deg x_{\circ}<d$ 
be the unique elements satisfying
\[
x = x_{\circ}+ v \cdot \partial_v x \quad \in A_v \quad \text{(i.e.~$x\equiv x_{\circ} \bmod v$ \ and \  $\partial_v x = \frac{x-x_{\circ}}{v}$).}
\]
When $x \in k \cap A_v$, we normalize the corresponding geometric special gamma value by
\begin{equation}\label{eqn: normal-v-gamma}
\Gamma_v^{\geo,*}\lbangle x\rbangle \assign (v \cdot 
\partial_v\lbangle x\rbangle )^{-\langle -\lbangle x\rbangle_{\circ}/v\rangle_{\geo}}\cdot 
\prod_{\afk \in A_+}\frac{\afk_{\flat}}{(\afk+\lbangle x\rbangle )_{\flat}}.
\end{equation}
Here we put $0^0 \assign 1 \in k$.
Then the corresponding reflection formula becomes:
\begin{align}\label{eqn: new-reflection}
\prod_{\varepsilon \in \FF_q^{\times}}\Gamma_v^{\geo,*}\lbangle\epsilon x\rbangle & = \zeta \cdot \begin{cases}
    \displaystyle v^{-1} \cdot \frac{\lbangle x \rbangle}{\partial_v \lbangle x \rbangle}, & \text{ if $v \nmid x$;}\\
    1 & \text{ otherwise,}
\end{cases}
\quad \text{ where } \zeta \in \FF_q^{\times}.
\end{align}
Moreover, we have the following Gross--Koblitz--Thakur formula established in \cite[Theorem 5.1.3]{Ch25}:

\begin{thm}\label{thm: G-K-T-formula-geo}
Suppose $x \in \nfk^{-1}A\setminus A$ with $v\nmid \nfk$.
Let $\ell$ be the order of $(v \bmod \nfk) \in (A/\nfk)^{\times}$.
Then
\[
\prod_{i=0}^{\ell-1}\Gamma_v^{\geo,*}\lbangle v^i x\rbangle^{-1} =  \prod_{s=0}^{d\ell-1}\Gcal_x^{\geo,(s)}
\rassign G_{\ell}^{\geo}(x) \quad \in \bar{k}^{\times}.
\]
In particular, we have $G^{\geo}_{\ell}(vx) = G^{\geo}_{\ell}(x)$.
\end{thm}

\subsection{Period interpretation of \texorpdfstring{$v$}{v}-adic special gamma values}
\label{sec: v-gamma-period}

By Proposition~\ref{prop: v-iso-period-Mx} and Remark~\ref{rem: period-omega-a}, we obtain the following period interpretation:

\begin{thm}\label{thm: period-vgamma-x}
Suppose $v \nmid \nfk$. 
Given $x \in \nfk^{-1}A\setminus A$ and $\ua \in (\uA/\unfk)^{\times}$, we have
\[
\hspace{3.5cm}
\sigma_{v,\cris}(\omega_{\ua}^{x})
=  \Gamma_v^{\geo,*}\lbangle vax\rbangle^{-1} \cdot \omega_{\uv \ua}^x
\hspace{1.5cm} 
(\text{Here $a = \ua(\theta) = \iota(\ua)$}).
\]
\end{thm}

\begin{proof}
Observe that $\langle -\lbangle vax\rbangle_{\circ}/v \rangle_{\geo} = 1$ if and only if there exists $\afk_0 \in A_+$ with $\deg \afk_0 < d$ such that $\afk_0 +\lbangle vax \rbangle =  v \lbangle ax \rbangle$.
Hence
\begin{align*}
\Gamma_v^{\geo,*}\lbangle vax \rbangle^{-1}
&= \prod_{N=0}^{d-1}\prod_{\afk \in A_{+,N}}\left(\frac{\afk + \lbangle vax\rbangle}{\afk}\right) \cdot \prod_{N=d}^{\infty}\prod_{\afk \in A_{+,N}}\left(\frac{(\afk+\lbangle vax \rbangle)_{\flat}}{\afk_{\flat}}\right) \\
&= \prod_{N=0}^{d-1}\prod_{\afk \in A_{+,N}}\left(1+ \frac{ \lbangle vax\rbangle}{\afk}\right) \cdot \prod_{N=d}^{\infty}
\frac{\displaystyle\prod_{\afk\in A_{+,N}}\left(1+\frac{\lbangle vax \rbangle}{\afk}\right)}{\displaystyle\prod_{\afk\in A_{+,N-d}}\left(1+\frac{\lbangle ax \rbangle}{\afk}\right)}
.
\end{align*}
Therefore the result holds by the formula in Remark~\ref{rem: period-omega-a}.
\end{proof}

\begin{rem}\label{rem: v-LR-conj}
Let $\beta = \{\omega_{\ua}^x\mid \ua \in (\uA/\unfk)^{\times}\}$ as above.
Then the matrices $[\Phi_{M_x,w}]_{\beta}$ and $[\tilde{\sigma}_{v,\cris}]_{\bar{\beta}}$ are completely determined by \eqref{eqn: Phi-M-x} and \eqref{eqn: v-cris-M-x}, and the normalized special value $\Gamma_v^{\geo,*}(vax)$ in Theorem~\ref{thm: period-vgamma-x} can be explicitly expressed in terms of the de~Rham--crystalline comparison periods in the following simple form by Proposition~\ref{prop: v-iso-period-Mx}:
\begin{align*}
\Gamma_{v}^{\geo,*}\lbangle vax \rbangle
& = \left(\prod_{N=1}^d g_x^{(N)}(\xi_{\uv\ua})^{-1}\right)
\cdot \frac{\psi_{x}^{(1)}(\xi_{\ua})}{\psi_x^{(d+1)}(\xi_{\uv\ua})} \\
& = 
\left(\bar{g}_{x,v}^{(d)}(\xi_{\uv\ua})^{-1} \cdot \prod_{N=1}^{d-1} \frac{g_x^{(N)}(\xi_{\ua})}{g_x^{(N)}(\xi_{\uv\ua})}\right)
\cdot   \frac{\psi_{x}^{(d)}(\xi_{\ua})}{\psi_x^{(d)}(\xi_{\uv\ua})}.
\end{align*}
\end{rem}
${}$\\

In general, for effective $\boxx = \sum_x n_x [x] \in \Ascr_{\unfk}^{\geo}$ and $a \in (A/\nfk)^{\times}$, we set
\begin{equation}\label{eqn: a-act}
a* \boxx \assign \sum_x n_x [ax] \quad \text{ and } \quad 
\Gamma_v^{\geo,*}\lbangle \boxx \rbangle \assign \prod_{x}\Gamma_v^{\geo,*}\lbangle x \rbangle^{n_x}.
\end{equation}
Combining Proposition~\ref{prop: v-iso-period-Mboxx} and Theorem~\ref{thm: period-vgamma-x}, we arrive at:

\begin{thm}\label{thm: period-vgamma-boxx}
Given effective $\boxx \in \Ascr_{\unfk}^{\geo}$with $\wt(\boxx)>0$ and $\ua \in (\uA/\unfk)^{\times}$, we have
\[
\hspace{3.5cm}
\sigma_{v,\cris}(\omega_{\ua}^{\boxx}) = \Gamma_v^{\geo,*}\lbangle va*\boxx\rbangle^{-1} \cdot \omega_{\uv\ua}^{\boxx}
\hspace{1.5cm} 
(\text{Here $a = \ua(\theta) = \iota(\ua)$}).
\]
\end{thm}

\begin{rem}\label{rem: period-Gauss-vmidn}
Suppose $v \nmid \nfk$. When $x \in \nfk^{-1}A\setminus A$ and $\ell \in \NN$ so that $v^{\ell} \equiv 1 \bmod \nfk$, let 
\[\boxx_v = \sum_{i=0}^{\ell-1}[v^ix] \ \in \Ascr^{\geo}_{\nfk}.
\]
From the Gross--Koblitz--Thakur formula stated in Theorem~\ref{thm: G-K-T-formula-geo}, we have
\[
\sigma_{v,\cris}(\omega_{\ua}^{\boxx_v}) = G_{\ell}^{\geo}(ax) \cdot \omega_{\uv\ua}^{\boxx_v}.
\]
This provides a period interpretation of the geometric Gauss sums.
\end{rem}

Taking suitable $\boxx \in \Ascr_{\unfk}^{\geo}$, we shall establish a $v$-adic Chowla--Selberg formula in the next section.

\section{\texorpdfstring{$v$}{v}-adic Chowla--Selberg formula}\label{Sec: v-CS}

Recall that $\uv=\iota^{-1}(v)\in\uA_+$.
Let $\unfk\in\uA_+$ with $\uv\nmid\unfk$, and let $\uK$ be an intermediate field of $\uK_{\unfk}/\uk$. Denote by $\uO_{\uK}$ the integral closure of $\uA$ in $\uK$.
In this section, we apply the theory of complex multiplication developed in \cite{BCPW22} to derive a $v$-adic Chowla–Selberg formula for $t$-motives over $\bar{k}$ with complex multiplication by $\uO_{\uK}$.
Although the CM theory in \cite{BCPW22} is formulated in terms of dual $t$-motives, its results can be translated directly into the language of $t$-motives.

\subsection{Isogeny theorem of CM \texorpdfstring{$t$}{t}-motives}\label{sec: CM}

A finite extension $\uF$ of $\uk$ is \emph{totally real} if the infinite place $\uinfty$ of $\uk$ splits completely in $\uF$.
Note that $\uF$ must be separable over $\uk$, and a compositum of two totally real field extensions of $\uk$ is still totally real.
This implies that every finite extension of $\uk$ has a maximal totally real subfield.

Let $\uK$ be a \emph{CM field over $\uk$}, that is, $\uK$ is a finite separable extension of $\uk$ satisfying that
every infinite place of the totally real subfield $\uK^+$ of $\uK$ is non-split in $\uK$.
For instance, an imaginary field $\uK$ over $\uk$ (i.e.~the infinite place of $\uk$ is non-split in $\uK$) is a CM field (with $\uK^+ = \uk$), and every totally real field $\uF$ over $\uk$ is also considered as a CM field in our setting.

\begin{defn}\label{defn: CM type}
Let $J_{\uK}$ be the set of $\FF_q$-algebra embeddings $\xi:\uK\rightarrow \bar{k}$ with $\xi(t) = \theta$,
and $I_{\uK}$ be the free abelian group generated by $J_{\uK}$.
A \emph{CM type of $\uK$} is a sum $\xi_1+\cdots + \xi_d \in I_{\uK}$ satisfying $d = [\uK^+:\uk]$ and $J_{\uK^+} = \big\{\xi_1\big|_{\uK^+},...,\xi_d\big|_{\uK^+}\big\}$. We call $\Xi \in I_{\uK}$ a \emph{generalized CM type of $\uK$} if $\Xi = \sum_{i=1}^w \Xi_i$, where $\Xi_i$ is a CM type of $\uK$ for $1\leq i \leq w$, and $w$ is called the \emph{weight of $\Xi$} and denoted by $\wt(\Xi)$.
\end{defn}

\begin{defn}\label{defn: CM t-motive}
(cf.~\cite[Proposition 4.3.4]{BCPW22})
Let $\uK$ be a CM field over $\uk$ and $\Xi$ be a generalized CM type of $\uK$. A CM $t$-motive with CM type $(\uK,\Xi)$ over $\bar{k}$ is a $t$-motive $M$ of rank $[\uK:\uk]$ over $\bar{k}$ satisfying:
\begin{itemize}
    \item[(1)] There exists an $\uA$-algebra homomorphism $\uO_{\uK}\hookrightarrow \End(M)$, where $\uO_{\uK}$ is the integral closure of $\uA$ in $\uK$.
    \item[(2)] Write $\Xi = \sum_{\xi \in J_{\uK}} m_{\xi} \xi \in I_{\uK}$. Regarding $M$ as a projective module of rank one over $\bar{k}\otimes_{\FF_q}\uO_{\uK}$, we have
    \[
    \frac{M}{\tau M} \cong \prod_{\xi \in J_{\uK}}\frac{\bar{k}\otimes_{\FF_q}\uO_{\uK}}{\Pfk_{\xi}^{m_{\xi}}},
    \]
    where $\Pfk_{\xi}$ is the kernel ideal of the $\bar{k}$-algebra homomorphism
    $\bar{k}\otimes_{\FF_q} \uO_{\uK}\rightarrow \bar{k}$ sending $c\otimes \ualpha$ to $c \cdot \xi(\ualpha)$ for every $c \in \bar{k}$ and $\ualpha \in \uO_{\uK}$.
    \item[(3)] $M$ is pure.
\end{itemize}
\end{defn}

\begin{ex}\label{ex: CM-t-mot}
We provide the examples of CM $t$-motives to be used as follows.
\begin{itemize}
    \item[(1)] ({\it Carlitz $t$-motive}, see~\cite[proof of Proposition 5.2.3]{BCPW22}.)\\
    Let $\xi_0: \uk\cong k \subset \bar{k}$ be the $\FF_q$-algebra isomorphism sending $t$ to $\theta$.
    Then Carlitz $t$-motive $M_C$ has CM type $(\uk,\xi_0)$. 
    
    \item[(2)] ({\it Soliton $t$-motives}, see~\cite[3.3.8]{Sinha97}, \cite[6.4.2]{ABP04}, and \cite[Lemma 5.3.1]{Wei26}.)\\  Given an effective $\boxx \in \Ascr^{\geo}_{\nfk}$ with $\wt(\boxx)>0$, the soliton $t$-motive $M_{\boxx}$ has CM by $\uO_{\unfk}$. Moreover, write $\boxx = \sum_x n_x [x]$.
    The CM type of $M_{\boxx}$ is $(\uK_{\unfk},\Xi_{\boxx})$, where
    \[
    \Xi_{\boxx}\assign \sum_{\ua \in (\uA/\unfk)^{\times}}\left(\sum_{x} n_x \langle ax \rangle_{\geo}\right) \cdot \xi_{\ua} \quad \in I_{\uK_{\unfk}}.
    \]
    \item[(3)] ({\it $t$-motives associated with CM Drinfeld modules}, see~\cite[after Theorem 6.2.2]{Wei26}.)\\
    Let $\uE_{\rho}$ be a CM Drinfeld module of rank $r$ over $\bar{k}$ with $\End(\uE_{\rho}) \cong \uO_{\uK}$ for an imaginary field $\uK$ over $\uk$.
    Extending $\rho$ to an $\FF_q$-algebra homomorphism from $\uO_{\uK}$ to $\bar{k}[\tau]$, the constant terms $\partial \rho_{\ualpha} \in \bar{k}$ for $\ualpha \in \uO_{\uK}$ (defined at the end of Section~\ref{sec: DM}) induces an $\FF_q$-algebra embedding $\xi_{\rho}:\uK\rightarrow \bar{k}$.
    Then the $t$-motive $M_{\rho}$ associated with $\uE_{\rho}$ is a CM $t$-motive with CM type $(\uK,\xi_{\rho})$.
\end{itemize}
\end{ex}

We recall the following fundamental properties of CM $t$-motives.

\begin{prop}\label{prop: CM}
${}$
\begin{itemize}
\item[(1)] {\rm (See~\cite[Proposition 4.2.7]{BCPW22})} Let $\uK$ be a CM field over $\uk$. Given CM $t$-motives $M_1$ and $M_2$ with CM types $(\uK,\Xi_1)$ and $(\uK,\Xi_2)$, respectively, the tensor product
\[
M_1\otimes_{\uK}M_2 \assign M_1\underset{\bar{k}\otimes_{\FF_q}\uO_{\uK}}{\otimes}M_2 \quad \text{ with }\quad \tau(m_1\otimes m_2)\assign (\tau m_1)\otimes (\tau m_2)
\]
is a CM $t$-motive with CM type $(\uK,\Xi_1+\Xi_2)$ over $\bar{k}$.
\item[(2)] {\rm (See~\cite[proof of Theorem 5.3.2 (2)]{BCPW22})} Let $\uK$ and $\uK'$ be CM fields over $\uk$ with $\uK\subset \uK'$.
Given a CM $t$-motive $M$ with CM type $(\uK,\Xi)$, the \emph{inflation}
\[
\Inf_{\uK}^{\uK'}(M)\assign \uO_{\uK'}\otimes_{\uO_{\uK}} M
\]
is a CM $t$-motive with CM type $(\uK',\Inf_{\uK}^{\uK'}(\Xi))$,
where for $\Xi = \sum_{\xi \in J_{\uK}}m_{\xi} \xi \in I_{\uK}$,
we put
\[
\Inf_{\uK}^{\uK'}(\Xi) \assign \sum_{\xi \in J_{\uK}}m_{\xi}\left(\sum_{\xi' \in J_{\uK'},\ \xi'\big|_{\uK}=\xi} \xi'\right).
\]
\end{itemize}
\end{prop}

Furthermore, we have the following isogeny theorem for CM $t$-motives:

\begin{thm}\label{thm: isog-thm}
{\rm (See~\cite[Proposition 4.2.8]{BCPW22})} Let $M_1$ and $M_2$ be CM $t$-motives with the same CM type $(\uK,\Xi)$ over $\bar{k}$.
Then $M_1$ and $M_2$ are $\uK$-isogenous over $\bar{k}$ (i.e.~the isogeny preserves the $O_{\uK}$-action on both sides).
\end{thm}

Let $M$ be a CM $t$-motive with CM type $(\uK,\Xi)$ over $\bar{k}$.
Recall that $\xi_0:\uk\rightarrow \bar{k}\subset \bar{k}_v$ sends $t$ to $\theta$.
We get
\[
\uK \hookrightarrow \bar{k}_{v}\underset{\xi_0,\, \uk}{\otimes} \uK \cong \frac{\bar{k}_v\otimes_{\FF_q}\uO_{\uK}}{(t-\theta)\cdot \bar{k}_v\otimes_{\FF_q}\uO_{\uK}} \stackrel{\sim}{\longrightarrow} \prod_{\xi \in J_{\uK}} \bar{k}_v, \quad \ualpha \longmapsto \big(\xi(\ualpha)\big)_{\xi\in J_{\uK}}, \quad \forall \alpha \in \uK.
\]
As $M$ is  projective of rank one over $\bar{k}\otimes_{\FF_q}\uO_{\uK}$, its de~Rham space $H_{\dR}(M,\bar{k}_v)$ is a free $\bar{k}_v\underset{\xi_0,\, \uk}{\otimes} \uK$-module of rank one.
Moreover, as mentioned in Remark~\ref{rem: period-Gauss-vmidn}~(2), it is known that every CM $t$-motive over $\bar{k}$ has ``good reduction everywhere''(see \cite[Theorem 15.4]{HaSi21} and \cite[Section 3.6]{Pel09}). Consequently:

\begin{cor}\label{cor: isog-cor}
Keep the notation as in {\rm Theorem~\ref{thm: isog-thm}}. A $\uK$-isogeny $f:M_1\rightarrow M_2$ induces a $\bar{k}_v\underset{\xi_0,\, \uk}{\otimes} \uK$-module isomorphism $f_{\dR}: H_{\dR}(M_1,\bar{k}_v)\stackrel{\sim}{\rightarrow} H_{\dR}(M_2,\bar{k}_v)$ which preserves the crystalline action on both sides.
\end{cor}

\begin{rem}\label{rem: dR-decomp}
${}$
\begin{itemize}
    \item[(1)] Let $M$ be a CM $t$-motive with CM type $(\uK,\Xi)$ over $\bar{k}$. Take $\omega^M \in H_{\dR}(M,\bar{k}_v)$ so that 
    \[
    H_{\dR}(M,\bar{k}_v)= (\bar{k}_v\underset{\xi_0,\, \uk}{\otimes} \uK)\cdot \omega^{M}.
    \]
    For each $\xi \in J_{\uK}$, let $e_{\xi} \in \bar{k}_v\underset{\xi_0,\, \uk}{\otimes}\ \uK$ be the idempotent associated with $\xi$, i.e.~$e_{\xi}^2 = e_{\xi}$ and $\ualpha \cdot e_{\xi} = \xi(\ualpha) \cdot e_{\xi}$ for every $\ualpha \in \uK$.
    Then $\omega^M_{\xi}\assign e_{\xi} \cdot \omega^M$ is an``eigen''-differential satisfying 
    \[
    \ualpha_{\dR}(\omega_{\xi}^M) = \xi(\ualpha) \cdot \omega_{\xi}^M, \quad \forall \ualpha \in \uO_{\uK}, \quad \text{ and } \quad 
    H_{\dR}(M,\bar{k}_v) = \underset{\xi \in J_{\uK}}{\bigoplus} \ \bar{k}_v \cdot \omega_{\xi}^M.
    \]
    \item[(2)] Given CM $t$-motives $M_1$ and $M_2$ with CM type $(\uK,\Xi_1)$ and $(\uK,\Xi_2)$,
    one has
    \[
    H_{\dR}(M_1\otimes_{\uK} M_2,\bar{k}_v) \cong H_{\dR}(M_1,\bar{k}_v) \underset{\bar{k}_v \underset{\xi_0,\, \uk}{\otimes}\uK}{\otimes} H_{\dR}(M_2,\bar{k}_v) = (\bar{k}_v \underset{\xi_0,\, \uk}{\otimes}\uK) \cdot \omega^{M_1}\otimes_{\uK} \omega^{M_2}.
    \]
    Moreover, if we take $\omega^{M_1\otimes_{\uK} M_2}$ be the generator of $H_{\dR}(M_1\otimes_{\uK} M_2)$ corresponding to $\omega^{M_1}\otimes_{\uK} \omega^{M_2}$,
    then
    \[
    \omega_{\xi}^{M_1\otimes_{\uK}M_2} \quad \text{ corresponds to}\quad  \omega_{\xi}^{M_1}\otimes_{\uK} \omega_{\xi}^{M_2}, \quad \forall \xi \in J_{\uK}.
    \]
    \item[(3)] Let $M$ be a CM $t$-motive with CM type $(\uK,\Xi)$ and $\uK'$ be another CM field with $\uK\subset \uK'$.
    We have the natural inclusion $\uK$-morphism (i.e.~preserves the $O_{\uK}$-action)
    \[
    M\hookrightarrow \uO_{\uK'}\otimes_{\uO_{\uK}} M = \Inf_{\uK}^{\uK'}(M), \quad m \longmapsto 1\otimes m, \quad \forall m \in M,
    \]
    which induces an embedding
    \[
    H_{\dR}(M,\bar{k}_v)\hookrightarrow H_{\dR}\big(\Inf_{\uK}^{\uK'}(M),\bar{k}_v\big).
    \]
    In particular, if $\omega^M$ is a generator of $H_{\dR}(M,\bar{k}_v)$ over $\bar{k}_v\underset{\xi_0,\uk}{\otimes}\uK$, then the above embedding sends $\omega^M$ to a generator of $H_{\dR}\big(\Inf_{\uK}^{\uK'}(M),\bar{k}_v\big)$
    over    $\bar{k}_v\underset{\xi_0,\uk}{\otimes}\uK'$, which we denote by $\omega^{\Inf_{\uK}^{\uK'}(M)}$.
    
    On the other hand, the trace map $\tr_{\uK'/\uK}: \uK'\longrightarrow \uK$ induces a $\uK$-morphism
    \begin{equation}\label{eqn: trace-morphism}
    \tr_{\uK'/\uK}: \Inf_{\uK}^{\uK'}(M) \rightarrow M, \quad \ualpha' \otimes m \longmapsto \tr_{\uK'/\uK}(\ualpha')\cdot m, \quad \forall \ualpha' \in \uO_{\uK'},\ m \in M.
    \end{equation}
    This corresponds to a $\bar{k}_v\underset{\xi,\uk}{\otimes}\uK$-linear map
\[
(\tr_{\uK'/\uK})_{\dR}: H_{\dR}\big(\Inf_{\uK}^{\uK'}(M),\bar{k}_v\big) \longrightarrow H_{\dR}(M,\bar{k}_v).
\]
Regarding $\tr_{\uK'/\uK}$ as a 
$\bar{k}_v\underset{\xi,\uk}{\otimes}\uK$-linear map from $\bar{k}_v\underset{\xi',\uk}{\otimes}\uK'$ to $\bar{k}_v\underset{\xi,\uk}{\otimes}\uK$ given by 
\[
c \otimes \ualpha' \longmapsto c \otimes  \tr_{\uK'/\uK}(\ualpha'), \quad \forall c \in \bar{k}_v,\ \ualpha' \in \uK',
\]
one checks that $\tr_{\uK'/\uK}(e_{\xi'}) = e_{\xi}$ for every $\xi' \in J_{\uK'}$ and $\xi = \xi'\big|_{\uK} \in J_{\uK}$.
As the generator $\omega^{\Inf_{\uK}^{\uK'}(M)}$ of $H_{\dR}\big(\Inf_{\uK}^{\uK'}(M),\bar{k}_v\big)$ comes from  $\omega^M$,
 we have
\[
(\tr_{\uK'/\uK})_{\dR}(\omega_{\xi'}^{\Inf_{\uK}^{\uK'}(M)}) = \tr_{\uK'/\uK}(e_{\xi'}) \cdot \omega^M = 
\omega_{\xi}^M \neq 0.
\]
\end{itemize}

\end{rem}

This enables us to establish our $v$-adic Chowla--Selberg formula in the next subsection.

\subsection{\texorpdfstring{$v$}{v}-adic Chowla--Selberg formula}\label{sec: v-CS}

Given $\unfk \in \uA_+$ with $\uv =\iota^{-1}(v) \nmid \unfk$, let 
$\uK_{\unfk}$ be the fraction field of $\uO_{\unfk}$.
Let 
$\uK$ be a subextension of $\uK_{\unfk}$ over $\uk$.
Then $\uK$ is a CM field over $\uk$, and the extension $\uK/\uk$ is geometric and abelian.
Let $M$ be a CM $t$-motive with CM type $(\uK,\Xi)$.
By Proposition~\ref{prop: CM}~(2), we have that $\Inf_{\uK}^{\uK_{\unfk}}(M)$ is a CM $t$-motive with CM type $(\uK_{\unfk},\Inf_{\uK}^{\uK_{\unfk}}(\Xi))$.
According to the recipe in the study of the ($\infty$-adic) Chowla--Selberg phenomenon over function fields in \cite[Section 6.3]{Wei26}, we can find $h,n \in \ZZ_{\geq 0}$ and $\boxx \in \Ascr^{\geo}_{\unfk}$ such that
\[
\Inf_{\uK}^{\uK_{\unfk}}(M^{\otimes_{\uK} h}(n)) \quad \text{ and } \quad M_{\boxx} \quad \text{ are $\uK_{\unfk}$-isogenous over $\bar{k}$.}
\]
Here $M^{\otimes_{\uK} h}(n)$ is the $n$-th Tate twist of
\begin{equation}\label{eqn: M-tensor-power}
M^{\otimes_{\uK} h} \assign \overbrace{M\otimes_{\uK} \cdots \otimes_{\uK} M}^{h \text{ times}}.
\end{equation}
An explicit way of choosing $h,n$ and $\boxx \in \Ascr^{\geo}_{\unfk}$ can be illustrated as follows:
for each Dirichlet character $\chi:(\uA/\unfk)^{\times}\rightarrow \CC^{\times}$, let $\ucfk_{\chi} \in \uA_+$ be the conductor of $\chi$, and let
\begin{align*}
L_{\uA}(s,\chi) & \assign \sum_{\subfrac{\uafk \in \uA_+}{\gcd(\uafk,\ucfk_{\chi})=1}}\frac{\chi(\uafk)}{q^{s\deg \uafk}}, \quad \forall \re(s)>1 \\
& = \sum_{\subfrac{\uafk \in \uA_+,\ \deg \uafk < \deg \ucfk_{\chi}}{\gcd(\uafk,\ucfk_{\chi})=1}}\frac{\chi(\uafk)}{q^{s\deg \uafk}}, \quad \text{ when $\chi$ is non-trivial (see \cite[Proposition 4.3]{Ros02}).}
\end{align*}
In particular, when $\chi$ is non-trivial, one has
\[
L_{\uA}(0,\chi) = \sum_{\subfrac{\uafk \in \uA_+,\, \deg \uafk < \deg \ucfk_{\chi}}{\gcd(\uafk,\ucfk_{\chi})=1}} \chi(\uafk).
\]

Note that the isomorphism $\xi_{\uone}: \uK_{\unfk}\stackrel{\sim}{\rightarrow} K_{\unfk}$ induces an isomorphism between $\Gal(\uK_{\unfk}/\uk)$ and $\Gal(K_{\unfk}/k)\cong (\uA/\unfk)^{\times}$ (by the Artin map \eqref{eqn: Artin-map}) so that for every $\ua \in (\uA/\unfk)^{\times}$, the corresponding $\varrho_{\ua} \in \Gal(\uK_{\unfk}/\uk)$ satisfies
\begin{equation}\label{eqn: ASrho}
\xi_{\uone}\circ \varrho_{\ua} = \varsigma_{\ua} \circ \xi_{\uone} = \xi_{\ua}.
\end{equation}
This allows us to identify $\Gal(\uK_{\unfk}/\uK)$ as a subgroup $H_{\uK}$ of $(\uA/\unfk)^{\times}$, and 
$G_{\uK} \assign \Gal(\uK/\uk)$ is regarded as the quotient of $(\uA/\unfk)^{\times}$ by $H_{\uK}$.
Moreover, every character $\chi$ on $G_{\uK}$ is viewed as a character on $(\uA/\unfk)^{\times}$ trivial on $H_{\uK}$. 
Let $\widehat{G}_{\uK}\assign \Hom(G_{\uK},\CC^{\times})$ be the Pontryagin dual of $G_{\uK}$.
For each $\varrho \in G_{\uK}$, $\ucfk \in \uA_+$ with $1\neq \ucfk \mid \unfk$ and $\ua \in (\uA/\unfk)^{\times}$, we put
\begin{equation}\label{eqn: n-c-rho-a}
n_{\ucfk}(\varrho,\ua)\assign \sum_{\subfrac{\chi \in \widehat{G}_{\uK}\setminus \widehat{G}_{\uK^+}}{\ucfk_{\chi} = \ucfk}}\frac{\chi(\varrho)\cdot \chi(\ua)}{L_{\uA}(0,\chi)} \quad \text{ which lies in $\QQ$,}
\end{equation}
and 
\begin{equation}\label{eqn: rcn}
h_{\uK/\uK^+}\assign \prod_{\chi \in \widehat{G}_{\uK}\setminus \widehat{G}_{\uK^+}} L_{\uA}(0,\chi) \quad \in \NN, \quad 
\text{the \emph{relative class number of $\uK/\uK^+$}.}
\end{equation}
Then $h_{\uK/\uK^+}\cdot n_{\ucfk}(\varrho,\ua) \in \ZZ$.
Now, write $J_{\uK}=\{\xi_{\varrho}\mid \varrho \in G_{\uK}\}$ where 
\begin{equation}\label{eqn: xi-rho}
\xi_{\varrho}\assign \xi_{\uone}\big|_{\uK}\circ \varrho: \uK \hookrightarrow K_{\unfk}\subset \bar{k} \quad \text{ for every $\varrho \in G_{\uK}$}.
\end{equation}
We obtain that
\begin{thm}\label{thm: isog-geoCM}
{\rm (See~\cite[Proposition 6.3.5]{Wei26})} Keep the above notation. For every generalized CM type $\Xi = \sum_{\varrho \in G_{\uK}}m_{\varrho} \cdot \xi_{\varrho}$ of $\uK$, we have
\begin{align*}
[\uK:\uk] \cdot h_{\uK/\uK^+} \cdot \Inf_{\uK}^{\uK_{\unfk}}(\Xi)
& = \big([\uK^+:\uk] \cdot h_{\uK/\uK^+} \cdot \wt(\Xi)\big)\cdot \Inf_{\uk}^{\uK_{\unfk}}(\xi_0) \\ 
& \quad + \sum_{\varrho \in G_{\uK}} \sum_{1\neq \ucfk \mid \unfk} \ \sum_{\subfrac{\ua \in \uA,\, \deg \ua < \deg \ucfk}{\gcd(\ua,\ucfk)=1}} \big(m_{\varrho} \cdot h_{\uK/\uK^+}\cdot n_{\ucfk}(\varrho,\ua)\big) \cdot \Xi_{a/\cfk} \quad \in I_{\uK_{\unfk}}.
\end{align*}
Here $\xi_0:\uk \cong k \subset \bar{k}$ is the CM type of $\uk$ mentioned in Example~\ref{ex: CM-t-mot}~(1), and $a/\cfk = \iota(\ua/\ucfk) \in k$.
\end{thm}

For each $\ucfk \in \uA_+$ with $1\neq \ucfk \mid \unfk$, take a sufficiently large $s_{\ucfk} \in \ZZ_{\geq 0}$ so that 
\begin{equation}\label{eqn: m-c}
n_{\ucfk}'(\varrho,\ua)\assign s_{\ucfk}+ m_{\varrho} \cdot h_{\uK/\uK^+}\cdot n_{\ucfk}(\varrho,\ua) \geq 0 \quad \text{ for every $\varrho \in G_{\uK}$, and $\ua \in (\uA/\unfk)^{\times}$}
\end{equation}
and
\begin{equation}\label{eqn: CM-type-n}
n\assign \sum_{1\neq \ucfk\mid \unfk} \frac{\#(\uA/\ucfk)^{\times}}{q-1} \cdot s_{\ucfk} - [\uK^+:\uk]\cdot h_{\uK/\uK^+}\cdot \wt(\Xi) \geq 0.
\end{equation}
Put 
\begin{equation}\label{eqn: h'-x-xi-ell}
h\assign [\uK:\uk]\cdot h_{\uK/\uK^+}
\quad \text{ and } \quad 
\boxx_{\Xi,n} \assign \sum_{\varrho \in G_{\uK}}  \sum_{1\neq \ucfk \mid \unfk} \sum_{\subfrac{\ua \in \uA,\, \deg \ua < \deg \ucfk}{\gcd(\ua,\ucfk)=1}} n_{\ucfk}'(\varrho,\ua)[\frac{a}{\cfk}] \quad \in \Ascr^{\geo}_{\unfk}.
\end{equation}
Then the isogeny theorem of CM $t$-motives implies that:
\begin{prop}\label{prop: isog}
\[
\Inf_{\uK}^{\uK_{\unfk}}(M^{\otimes_{\uK}h}(n)) \quad \text{ and } \quad M_{\boxx_{\Xi,n}} 
\quad \text{ are $\uK_{\unfk}$-isogenous over $\bar{k}$.}
\]
\end{prop}

\begin{proof}
Observe that for each $x \in \nfk^{-1}A\setminus A$, by \eqref{eqn: geo-bra-ref} we have
\[
\sum_{\epsilon \in \FF_q^{\times}} \Xi_{\epsilon x}
= \sum_{\ua \in (\uA/\unfk)^{\times}} \xi_{\ua} = \Inf_{\uk}^{\uK_{\unfk}}(\xi_0).
\]
Thus the CM type of
$\Inf_{\uK}^{\uK_{\unfk}}(M^{\otimes_{\uK}h}(n))$
is equal to
\begin{align*}
\Inf_{\uK}^{\uK_{\unfk}}\big(h\cdot \Xi+ n \cdot \Inf_{\uk}^{\uK}(\xi_0)\big)
& = [\uK:\uk]\cdot h_{\uK/\uK^+}\cdot \Inf_{\uK}^{\uK_{\unfk}}(\Xi) +  n \cdot \Inf_{\uk}^{\uK_{\unfk}}(\xi_0) \\
& = 
\Big(\sum_{1\neq \ucfk\mid \unfk} \frac{\#(\uA/\ucfk)^{\times}}{q-1} \cdot s_{\ucfk}\Big)\cdot \Inf_{\uk}^{\uK_{\unfk}}(\xi_0) \\
& \quad + 
\sum_{\varrho \in G_{\uK}} \sum_{1\neq \ucfk \mid \unfk} \ \sum_{\subfrac{\ua \in \uA,\, \deg \ua < \deg \ucfk}{\gcd(\ua,\ucfk)=1}} \big(m_{\varrho} \cdot h_{\uK/\uK^+}\cdot n_{\ucfk}(\varrho,\ua)\big) \cdot \Xi_{a/\cfk} \\
& = 
\sum_{\varrho \in G_{\uK}} \sum_{1\neq \ucfk \mid \unfk} \ \sum_{\subfrac{\ua \in \uA,\, \deg \ua < \deg \ucfk}{\gcd(\ua,\ucfk)=1}} \big(s_{\ucfk}+ m_{\varrho} \cdot h_{\uK/\uK^+}\cdot n_{\ucfk}(\varrho,\ua)\big) \cdot \Xi_{a/\cfk} \\
&= \Xi_{\boxx_{\Xi,n}},
\end{align*}
which is the CM type of $M_{\boxx_{\Xi,n}}$.
Therefore the result follows from Theorem~\ref{thm: isog-thm}. 
\end{proof}

Recall in Remark~\ref{rem: dR-decomp}~(3) that the trace map $\tr_{\uK_{\unfk}/\uK}:\uK_{\unfk}\rightarrow \uK$ induces a ${\uK}$-morphism
\[
\tr_{\uK_{\unfk}/\uK}:\Inf_{\uK}^{\uK_{\unfk}}(M^{\otimes_{\uK}h}(n)) = \uO_{\unfk}\otimes_{\uO_{\uK}} M^{\otimes_{\uK}h}(n) \longrightarrow  M^{\otimes_{\uK}h}(n)
\]
given by
\[
\tr_{\uK_{\unfk}/\uK}(\ualpha\otimes m) = \tr_{\uK_{\unfk}/\uK}(\ualpha) \cdot m
= \Big(\sum_{\varrho \in H_{\uK}}\varrho(\ualpha)\Big)\cdot m, \quad \forall \ualpha \in \uO_{\unfk} \ \text{ and }\ m \in M^{\otimes_{\uK}h}(n).
\]
Let $f: M_{\boxx_{\Xi,n}}\rightarrow \Inf^{\uK_{\unfk}}_{\uK}(M^{\otimes_{\uK}h}(n))$ be a $\uK_{\unfk}$-isogeny over $\bar{k}$ given in Proposition~\ref{prop: isog}.
For each $\varrho_0 \in G_{\uK}$, take $\ua_0 \in (\uA/\unfk)^{\times}$ so that $\varrho_{\ua_0}\big|_{\uK} = \varrho_0$, which implies that 
\[
\xi_{\ua_0}\big|_{\uK} = \xi_{\uone}\circ \varrho_{\ua_0}\big|_{\uK} = \xi_{\uone}\big|_{\uK}\circ \varrho_0 =  \xi_{\varrho_0} \in J_{\uK}.
\]
Then the differential $\omega_{\varrho_0}'\assign (\tr_{\uK_{\unfk}/\uK}\circ f)_{\dR}(\omega^{\boxx_{\Xi,n}}_{\ua_0}) \in H_{\dR}\big(M^{\otimes_{\uK}h}(n),\bar{k}\big)$, which only depends on $\varrho_0 \in G_{\uK}$, is nonzero and satisfies
\[
\ualpha_{\dR}(\omega_{\varrho_0}') = \xi_{\varrho_0}(\ualpha)\cdot \omega_{\varrho_0}', \quad \forall \ualpha \in \uO_{\uK}.
\]
From the decomposition of the de~Rham spaces of CM $t$-motives in Remark~\ref{rem: dR-decomp}~(1) and (2), 
there must exist a differential $\omega^{M}_{\varrho_0} \in H_{\dR}(M,\bar{k})$, which is unique up to a mutliple of roots of unity, such that $\ualpha_{\dR}(\omega^{M}_{\varrho_0}) = \xi_{\varrho_0}(\ualpha) \cdot \omega^{M}_{\varrho_0}$ for every $\ualpha \in \uO_{\uK}$ and 
\[
\omega_{\varrho_0}' =(\omega^{M}_{\varrho_0})^{\otimes_{\uK} h} \otimes \omega_C^{\otimes n} \quad \in H_{\dR}\big(M^{\otimes_{\uK}h}(n),\bar{k}\big).
\]
Here $\omega_C \in H_{\dR}(M_C,k)$ is given in Example~\ref{sec: Ex-Car-t}.
By Theorem~\ref{thm: period-vgamma-boxx} and the functoriality of the crystalline action, we arrive at the following $v$-adic Chowla--Selberg formula:

\begin{thm}\label{thm: v-CS}
Keep the above notation.
Then there exists $\zeta_{\varrho_0} \in \overline{\FF}_q^{\times}$ such that for each $\sigma_v \in \Wcal_v$ with $\deg \sigma_v = 1$, we have
\[
\sigma_{v,\cris}(\omega^{M}_{\varrho_0}) = 
\zeta_{\varrho_0} \cdot v^{\frac{\wt(\Xi)}{[\uK:\uK^+]}} \cdot \prod_{\varrho \in G_{\uK}} \left(
\prod_{1\neq \ucfk \mid \unfk} \ \prod_{\subfrac{\ua \in \uA,\, \deg \ua < \deg \ucfk}{\gcd(\ua,\ucfk)=1}}\Gamma_v^{\geo,*}\Big\lbangle \frac{va}{\ucfk}\Big\rbangle^{-n_{\ucfk}(\varrho,\ua)}
\right)^{\frac{m_{\varrho\varrho_0}}{[\uK:\uk]}} \cdot \omega^{M}_{\varrho_{\uv}\varrho_0}.
\]
Here $\varrho_{\uv}\varrho_0\assign \varrho_{\uv}\big|_{\uK}\circ \varrho_0 \in G_{\uK}$.
\end{thm}

\begin{proof}
From the functoriality of the crystalline action, by Theorem~\ref{thm: period-vgamma-boxx} we have that
\begin{align*}
\sigma_{v,\cris}(\omega'_{\varrho_0})
& = (\tr_{\uK_{\unfk}/\uK}\circ f)_{\dR}\big(\sigma_{v,\cris}(\omega^{\boxx_{\Xi,n}}_{\ua_0})\big) \\
& = (\tr_{\uK_{\unfk}/\uK}\circ f)_{\dR}\big(\Gamma_v^{\geo,*}\lbangle va_0* \boxx_{\Xi,n}\rbangle^{-1} \cdot 
\omega^{\boxx_{\Xi,n}}_{\uv\ua_0}\big) \\
&= 
\Gamma_v^{\geo,*}\lbangle va_0* \boxx_{\Xi,n}\rbangle^{-1} \cdot \omega'_{\varrho_{\uv}\varrho_{0}}.
\end{align*}
Since $\omega'_{\varrho} = (\omega_{\varrho_0}^M)^{\otimes_{\uK}h} \otimes \omega_C^{\otimes n}$ and $\sigma_{v,\cris}(\omega_C) = v \cdot \omega_C$ by Example~\ref{sec: Ex-Car-t}, there exists $c \in \bar{k}_v^{\times}$ such that 
\[
\sigma_{v,\cris}(\omega^M_{\varrho_0})
=c \cdot \omega^M_{\varrho_{\uv}\varrho_0} 
\quad \text{ and } \quad 
c^{h} = \Big(v^{n}\cdot \Gamma_{v}^{\geo,*}\lbangle va_0*\boxx_{\Xi,n}\rbangle\Big)^{-1}.
\]
Finally, from the reflection formula in \eqref{eqn: new-reflection}, we get that for every $\ucfk \in \uA_+$ with $1\neq \ucfk \mid \unfk$, there exists $\zeta' \in \overline{\FF}_q^{\times}$ such that
\[
\prod_{\subfrac{\ua \in \uA,\, \deg \ua <\deg \ucfk}{\gcd(\ua,\ucfk)=1}}\Gamma_v^{\geo,*}\Big\lbangle\frac{a}{\cfk}\Big\rbangle = \zeta'\cdot v^{-\frac{\#(A/\cfk)^{\times}}{q-1}}.
\]
Hence there exists $\zeta'' \in \overline{\FF}_q^{\times}$ so that 
\begin{align*}
v^{n}\cdot \Gamma_{v}^{\geo,*}\lbangle va_0*\boxx_{\Xi,n}\rbangle
& = \zeta''\cdot v^{-h_{\uK/\uK^+} \cdot  [\uK^+:\uk]\cdot \wt(\Xi)} \\
& \quad \ \cdot 
\prod_{\varrho \in G_{\uK}} \prod_{1\neq \ucfk \mid \unfk} \ \prod_{\subfrac{\ua \in \uA,\, \deg \ua < \deg \ucfk}{\gcd(\ua,\ucfk)=1}}\Gamma_{v}^{\geo,*}\Big\lbangle \frac{va_0a}{\cfk}\Big\rbangle^{ h_{\uK/\uK^+}\cdot m_{\varrho} \cdot n_{\ucfk}(\varrho,\ua)} \\
&= \zeta''\cdot 
v^{-h_{\uK/\uK^+} \cdot [\uK^+:\uk]\cdot  \wt(\Xi)} \\
& \quad \ \cdot 
\prod_{\varrho \in G_{\uK}} \prod_{1\neq \ucfk \mid \unfk} \ \prod_{\subfrac{\ua \in \uA,\, \deg \ua < \deg \ucfk}{\gcd(\ua,\ucfk)=1}}\Gamma_{v}^{\geo,*}\Big\lbangle \frac{va}{\cfk}\Big\rbangle^{h_{\uK/\uK^+}\cdot m_{\varrho\varrho_0} \cdot  n_{\ucfk}(\varrho,\ua)}.
\end{align*}
Therefore, by taking $h$-th roots on both sides (where $h = h_{\uK/\uK^+}\cdot[\uK:\uk]$), we see that there exists $\zeta_{\varrho_0} \in \overline{\FF}_q^{\times}$ such that
\[
c = \zeta_{\varrho_0} \cdot 
v^{\frac{\wt(\Xi)}{[\uK:\uK^+]}} \cdot \prod_{\varrho \in G_{\uK}} \left(
\prod_{1\neq \ucfk \mid \unfk} \ \prod_{\subfrac{\ua \in \uA,\, \deg \ua < \deg \ucfk}{\gcd(\ua,\ucfk)=1}}\Gamma_v^{\geo,*}\Big\lbangle \frac{va}{\ucfk}\Big\rbangle^{-n_{\ucfk}(\varrho,\ua)}
\right)^{\frac{m_{\varrho\varrho_0}}{[\uK:\uk]}}
\]
as desired.
\end{proof}

\begin{rem}\label{rem: CM-quad}
As every infinite place of the maximal totally real subfield $\uK_{\unfk}^+$ (i.e.~the place lying above $\uinfty$)
is totally ramified in $\uK_{\unfk}$ with ramification index $q-1$,  every imaginary field $\uK$ contained in $\uK_{\unfk}$ satisfies $[\uK:\uk]\mid (q-1)$, and $\uinfty$ is totally ramified in $\uK$. Consequently, when there a subfield $\uK$ of $\uK_{\unfk}$ which is imaginary  quadratic over $\uk$,
we must have that $q$ is odd.
Moreover, the above theorem leads to a precise analogue of Ogus-type $v$-adic Chowla--Selberg formula presented in the next subsection.
\end{rem}

\subsection{Quadratic case}
\label{sec: v-CS-q}

Suppose $q$ is odd. Let $\udfk \in \uA_+$ be a monic square-free polynomial with odd degree, and let $\uK \assign \uk(\sqrt{-\udfk})$, which is contained in the cyclotomic function field $\uK_{\udfk}$. 
Put $\uO_{\uK} = \uA[\sqrt{-\dfk}]$.
Let $\uE_{\rho}$ be a CM Drinfeld module of rank $2$ over $\bar{k}$ with $\End(\uE_{\rho})\cong \uO_{\uK}$, and let $M_{\rho}$ be the $t$-motive associated with $\uE_{\rho}$ over $\bar{k}$. 
It is known that $M_{\rho}$ has potentially good reduction everywhere (see Remark~\ref{rem: period-Gauss-vmidn}~(2)).

Without loss of generality, assume that the CM type of $M_{\rho}$ is $(\uK,\xi_{\id_{\uK}})$.
As $\uK^+=\uk$ and the subgroup $H_{\uK}$ of $(\uA/\udfk)^{\times}$ corresponding to $\uK$ has index $2$, there exists only one non-trivial character $\chi_{\uK}:(\uA/\udfk)^{\times}\rightarrow \{\pm 1\}$ with $\ker (\chi_{\uK}) = H_{\uK}$.
Thus in this case, if we write $G_{\uK} = \{\id_{\uK},\varrho_1\}$, then we get
\[
\wt(\xi_{\id_{\uK}}) = 1,\quad m_{\id_{\uK}}=1,\ m_{\varrho_1}=0, \quad 
h_{\uK/\uK^+} = L_{\uA}(0,\chi_{\uK}) = \#\Pic(\uO_{\uK})\rassign h_{\uK},
\]
\[
\quad \text{ and for $\varrho \in G_{\uK}$,} \quad 
n_{\ucfk}(\varrho,\ua) = 
\begin{cases}
\chi_{\uK}(\varrho) \cdot \displaystyle \frac{\chi_{\uK}(\ua)}{h_{\uK}}, & \text{ if $\ucfk = \udfk$;} \\
0, & \text{ otherwise.}
\end{cases}
\]

Write $J_{\uK} = \{\xi_{+}, \xi_{-}\}$, where $\xi_{\chi_{\uK}(\varrho)}\assign \xi_{\varrho}$ for $\varrho \in G_{\uK}$, and also put
\[
\omega_{\pm }^{\rho} \assign \omega^{M_{\rho}}_{\varrho} \quad \text{ where $\varrho \in G_{\uK}$ with $\chi_{\uK}(\varrho) = \pm 1$.}
\]
Then
Theorem~\ref{thm: v-CS} gives us the following:

\begin{thm}\label{thm: v-CS-quad}
Keep the above notation. There exists $\zeta_{\pm } \in \overline{\FF}_q^{\times}$ so that for each $\sigma_v \in \Wcal_v$ with $\deg \sigma_v =1$, we have 
\[
\sigma_{v,\cris}(\omega_{\pm }^{\rho}) = \zeta_{\pm } \cdot v^{\frac{1}{2}} \cdot \left(\prod_{\subfrac{\ua \in \uA,\, \deg \ua < \deg \udfk}{\gcd(\ua,\udfk)=1}}
\Gamma_v^{\geo,*}\Big\lbangle\frac{va}{\dfk}\Big\rbangle^{\mp \frac{\chi_{\uK}(\ua)}{2h_{\uK}}}\right)
\cdot \omega_{\pm \chi_{\uK}(\uv)}^{\rho}.
\]
\end{thm}

\begin{rem}
It is also natural to ask whether the monomial relations of $v$-adic geometric special gamma values arised from the standard functional equations in Proposition~\ref{prop: FE}, together with the additional relation coming from the Gross--Koblitz--Thakur formula stated in Theorem~\ref{thm: G-K-T-formula-geo}, explain all the algebraic relations among these values in question.

Recall that given $\unfk \in \uA_+$ with $\deg \unfk>0$, $\Ascr_{\unfk}^{\geo}$ is the free abelian group generated by elements in $\nfk^{-1}A$ (where $\nfk = \unfk(\theta)=\iota(\unfk)$).
Let $\Ascr_{\unfk,\QQ}^{\geo}\assign \Ascr_{\unfk}^{\geo}\otimes_{\ZZ}\QQ$ and $\Dscr_{\unfk,\QQ}^{\geo}$ be the subspace of $\Ascr_{\unfk,\QQ}^{\geo}$ generated by
\[
[x+a]-[x], \quad \forall\  x \in \nfk^{-1}A,\ a \in A,
\]
and
\[
[\mfk x] - \sum_{a \in A,~\deg a<\deg \mfk} [x+\frac{a}{\mfk}], \quad \forall \mfk \in A_+ \text{ with } \mfk \mid \nfk.
\]
The quotient $\Uscr_{\unfk,\QQ}^{\geo}\assign \Ascr_{\unfk,\QQ}^{\geo}/\Dcal_{\unfk,\QQ}^{\geo}$ is called the \emph{universal ordinary distribution} on $\nfk^{-1}A/A$.
Suppose $v \nmid \nfk$. Let $\Vscr_{\unfk,\QQ}^{\geo}$ be the subspace of $\Uscr_{\unfk,\QQ}^{\geo}$ generated by the `multiplication' relation
\[
\sum_{\epsilon \in \FF_q^{\times}} [\epsilon x], \quad \forall x \in \nfk^{-1}A,
\]
and let $\Wscr_{\unfk,\QQ}^{\geo}$ be the subspace of $\Uscr_{\unfk,\QQ}^{\geo}$ generated by the following `Gross--Koblitz--Thakur' relation:
\[
\sum_{i=0}^{\ell-1} [v^ix], \quad \forall x \in \nfk^{-1}A, 
\quad \text{where $\ell$ is the order of $\bar{v} \assign (v \bmod \nfk) \in (A/\nfk)^{\times}$.}
\]
Then:
\begin{lem}\label{lem: dim-U}
Keep the notation as above.
Let $H_v$ be the subgroup of $(A/\nfk)^{\times}$ generated by $\bar{v}$ and $\FF_q^{\times}$.
We have
\[
\dim_{\QQ}\left(\frac{\Uscr_{\unfk,\QQ}^{\geo}}{\Vscr_{\unfk,\QQ}^{\geo}+\Wscr_{\unfk,\QQ}^{\geo}}\right) =    
\Big(1- \frac{1}{q-1}- \frac{1}{\ell} + \frac{1}{\#H_v}\Big) \cdot \#\Big(\frac{A}{\nfk}\Big)^{\times}.
\]
\end{lem}

\begin{proof}
Observe that the action of $(A/\nfk)^{\times}$ on $\Ascr_{\unfk,\QQ}^{\geo}$ given in \eqref{eqn: a-act} induces an action on $\Uscr_{\unfk,\QQ}^{\geo}$, and it is known that (see \cite[Theorem 6.1.3]{ABP04}) 
\begin{equation}\label{eqn: dim-U}
\Uscr_{\unfk,\QQ}^{\geo} \cong \QQ[(A/\nfk)^{\times}] \quad \text{ as  $\QQ[(A/\nfk)^{\times}]$-modules.}
\end{equation}
In particular, one has $\dim_{\QQ} \Uscr_{\unfk,\QQ}^{\geo} = \#(A/\nfk)^{\times}$.
Moreover,  $\Vscr_{\unfk,\QQ}^{\geo}$ coincides with the subspace of $\Uscr_{\unfk,\QQ}^{\geo}$ consisting of elements fixed by every $\epsilon \in \FF_q^{\times}$, and  $\Wscr_{\unfk,\QQ}^{\geo}$ equals to the subspace of $\Uscr_{\unfk,\QQ}^{\geo}$ consisting of elements fixed by $\bar{v} \in (A/\nfk)^{\times}$.
Hence
\[
\dim_{\QQ}(\Vscr_{\unfk,\QQ}^{\geo}) = \frac{1}{q-1}\cdot \#\left(\frac{A}{\nfk}\right)^{\times},\quad
\dim_{\QQ}(\Wscr_{\unfk,\QQ}^{\geo}) = \frac{1}{\ell}\cdot \#\left(\frac{A}{\nfk}\right)^{\times},
\]
and
\[\dim_{\QQ}(\Vscr_{\unfk,\QQ}^{\geo}\cap \Wscr_{\unfk,\QQ}^{\geo}) = \frac{1}{\#H_v}\cdot \#\left(\frac{A}{\nfk}\right)^{\times}. 
\]
Therefore
\begin{align*}
\dim_{\QQ}\left(\frac{\Uscr_{\unfk,\QQ}^{\geo}}{\Vscr_{\unfk,\QQ}^{\geo}+\Wscr_{\unfk,\QQ}^{\geo}}\right) &=  
\dim_{\QQ}(\Uscr_{\unfk,\QQ}^{\geo})-
\left(\dim_{\QQ}(\Vscr_{\unfk,\QQ}^{\geo})+\dim_{\QQ}(\Wscr_{\unfk,\QQ}^{\geo})
-\dim_{\QQ}(\Vscr_{\unfk,\QQ}^{\geo}\cap \Wscr_{\unfk,\QQ}^{\geo})\right)\\
&= \Big(1- \frac{1}{q-1}- \frac{1}{\ell} + \frac{1}{\#H_v}\Big) \cdot \#\Big(\frac{A}{\nfk}\Big)^{\times}.
\end{align*}
\end{proof}

Define $\hat{\Gamma}_v^{\geo}: \Ascr_{\unfk,\QQ}^{\geo} \rightarrow \CC_v^{\times}/\bar{k}^{\times}$ by 
\[
\hat{\Gamma}_v^{\geo}(\boxx)\assign \Gamma_v^{\geo,*}\lbangle \boxx \rbangle \cdot \bar{k}^{\times}, \quad \forall \boxx \in \Ascr_{\unfk,\QQ}^{\geo}.
\]
Then the standard functional equations of $\Gamma_v^{\geo}$ in Proposition~\ref{prop: FE} and the Gross--Koblitz--Thakur formula in Theorem~\ref{thm: G-K-T-formula-geo} implies that $\hat{\Gamma}_v^{\geo}$ factors through $\Uscr_{\unfk,\QQ}^{\geo}\big/(\Vscr_{\unfk,\QQ}^{\geo}+\Wscr_{\unfk,\QQ}^{\geo})$, whence
\[
\trdeg_{\bar{k}}\, \bar{k}\big(\Gamma_v^{\geo}(x)\ \big|\ x \in \nfk^{-1} A\big)
\leq \Big(1- \frac{1}{q-1}- \frac{1}{\ell} + \frac{1}{\#H_v}\Big) \cdot \#\Big(\frac{A}{\nfk}\Big)^{\times}.
\]
We propose a Lang--Rohrlich-type conjecture at the end, which will be investigated in a subsequent paper:
\end{rem}

\begin{conj}\label{conj: v-LR-conj}
Given $\nfk \in A_+$ with $\deg \nfk >0$,
suppose $v \nmid \nfk$.
Let $\ell$ be the order of $\bar{v} \assign (v \bmod \nfk) \in (A/\nfk)^{\times}$. 
Let $H_v$ be the subgroup of $(A/\nfk)^{\times}$ generated by $\bar{v}$ and $\FF_q^{\times}$.
Then 
\[
\trdeg_{\bar{k}}\, \bar{k}\big(\Gamma_v^{\geo}(x)\ \big|\ x \in \nfk^{-1} A\big)
\stackrel{?}{=} \Big(1- \frac{1}{q-1}- \frac{1}{\ell} + \frac{1}{\#H_v}\Big) \cdot \#\Big(\frac{A}{\nfk}\Big)^{\times}.
\]    
In other words, all the algebraic relations among $v$-adic geometric special gamma values (at rationals) would be completely explained by the monomial relations coming from the standard functional equations and the Gross--Koblitz--Thakur formula.
\end{conj}

\subsection*{Acknowledgements}

The author is very grateful to Chieh-Yu Chang and Jing Yu for very insightful discussions at the very beginning of this project and for their strong and constant support.

\bibliographystyle{alpha}

\begin{thebibliography}{999999999}

\bibitem[AnFr25]{AnFr25}
Ancona, G.~and Frăţilă, D., {\it Algebraic classes in mixed characteristic and Andr\'e's $p$-adic periods}, 
J.~Inst.~Math. Jussieu 24 no.~4 (2025) 1093--1138.



\bibitem[And82]{And82}
Anderson, G.~W.,
\textit{Logarithmic derivatives of Dirichlet $L$-functions and the periods of abelian varieties}, Compositio Math.\ tome 45, n$^{\rm o}\,$3 (1982) 315--332.

\bibitem[A86]{A86}
Anderson, G.\  W., \textit{$t$-motives}, Duke Math.~J.~\textbf{53} (1986), no.~2, 457--502.

\bibitem[A92]{A92}
Anderson, G.~W.,
{\it A two-dimensional analogue of Stickelberger's theorem}, 57-77, in {\it The Arithmetic of Function Fields}
(D.~Goss, D.~R.~Hayes, and M.~I.~Rosen, eds.), W.~de Gruyter, Berlin, 1992.

\bibitem[ABP04]{ABP04}
Anderson, G.~W., Brownawell, W.~D., and Papanikolas,  M.~A.,
\textit{Determination of the algebraic relations among special $\Gamma$-values in positive characteristic}, Ann.~of Math.~(2) \textbf{160} (2004), no.~1, 237--313.


\bibitem[AsHa26]{AsHa26}
Asakura, M.~and Hagihara, K., {\it Frobenius structure on hypergeometric equations, $p$-adic polygamma values and $p$-adic $L$-values}, 
Tunis.~J.~Math.~8~no.~3 (2026) 453-496.



\bibitem[BM16]{BM16}
Barquero-Sanchez, A.~and Masri, R., \textit{The Chowla--Selberg formula for abelian CM fields and Faltings height}, Compositio Math.\ \textbf{152} (2016), 445--476.






\bibitem[BCPW22]{BCPW22}
Brownawell,W.~D., Chang, C.-Y., Papanikolas, M.~A., and Wei, F.-T., {\it Function field analogue of
 Shimura's conjecture on period symbols}, preprint (arXiv:2203.09131). 

\bibitem[BP02]{BP02}
Brownawell, W.~D.~and Papanikolas, M.~A., {\it Linear independence of Gamma-values in positive characteristic}, J.~Reine Angew.~Math.~549 (2002), 91--148.

\bibitem[Ch26]{Ch26}
Chang, C.-Y., {\it On special gamma values: Chowla--Selberg formula and Rohrlich--Lang conjecture revisited}, preprint.


\bibitem[CWY26]{CWY26}
Chang, C.-Y., Wei, F.-T., and Yu, J.,
{\it $v$-adic periods of Carlitz motives and Chowla--Selberg formula revisited}, to appear in Compositio math.~(arXiv: 2407.15024).



\bibitem[Ch25]{Ch25}
Chang, T.-W.,
{\it Geometric Gauss sums and Gross--Koblitz formula over function fields}, preprint (arXiv: 2502.01109).




\bibitem[CS67]{CS67}
Chowla, S.~and Selberg, A., {\it On Epstein's zeta-function}, J.~Reine Angew.~Math.~ 227(1967), 86--110.



\bibitem[Co93]{Co93}
Colmez, P., {\it P\'eriodes des Vari\'et\'es Ab\'eliannes \`a multiplication complexe}, Ann.~of Math.~\textbf{138} (1993) 625--683.

\bibitem[Co26]{Co26}
Colmez, P., {\it La formule limite de Kronecker}, 
Adv.~Math.~500 (2026) Paper No.~111096.



\bibitem[DH87]{DH87}
Deligne, P.~and Husem\"oller, D., {\it Survey of Drinfel'd modules},
Comptemporary Mathematics, vol.~67 (1987) 25--91.


\bibitem[Gar02]{Gar02}
Gardeyn, F.,
{\it A Galois criterion for good reduction of $\tau$-sheaves}, J.~Number Theory 79 (2002) 447--471.

\bibitem[Gar03]{Gar03}
Gardeyn, F.,
{\it The structure of analytic $\tau$-sheaves}, J.~Number Theory 100 (2003) 332--362.

\bibitem [Ge89] {Ge89} Gekeler, E.-U., \textit{ On the de Rham isomorphism for Drinfeld modules}, J.~Reine
Angew.~Math.~\textbf{401} (1989), 188--208.

\bibitem[GL11]{GL11}
Genestier, A.~and Lafforgue, V., 
{\it Th\'eorie de Fontaine en \'egales charact\'eristiques}, Ann.~Sci.~\'Ecole Norm.~Sup\'er.~44(2) (2011) 263--360; also available at http://www.math.jussieu.fr/$\sim$vlafforg/

  

\bibitem[Goss96] {Goss96} Goss, D., \textit{Basic structures of function field arithmetic}, Springer-Verlag, Berlin, 1996.


\bibitem[Gr78]{Gr78} Gross, B.~H.,
{\it On the periods of abelian integrals and a formula of Chowla and Selberg}, Inventiones math.~45 (1978) 193--211.

\bibitem[Gr21]{Gr21} Gross, B.~H., 
{On the periods of abelian varieites}, ICCM Not.\ 8 no.\ 2 (2020) 10--18.

\bibitem[GrKo79]{GrKo79}Gross, B.~H.~and Koblitz, N., {\it Gauss sums and the $p$-adic $\Gamma$-function}, Ann.~of Math., vol.~109, No.~3 (1979) 569--581.



\bibitem[Ha09]{Ha09}
Hartl, U., {\it A dictionary between Fontaine-Theory and its analogue in equal characteristic},
J.~Number Theory 129 no.~7 (2009) 1734--1757.




\bibitem[HaJu20]{HJ20} Hartl, U.~and Juschka, A.-K., \textit{ Pink's theory of Hodge structures and the Hodge conjecture over function fields} in: $t$-Motives: Hodge Structures, Transcendence and Other Motivic Aspects, Eur. Math. Soc., Z\"{u}rich, 2020, pp.~31--182.

\bibitem[HaKi20]{HK20} Hartl U.~and Kim, W., \textit{Local Shtukas, Hodge--Pink structures, and Galois representations} in: $t$-Motives: Hodge Structures, Transcendence and Other Motivic Aspects, Eur. Math. Soc., Z\"{u}rich, 2020, pp.~183--259.

\bibitem[HaSi21]{HaSi21} Hartl, U.~and Singh, R.~K., {\it Product formulas for periods of CM abelian varieties and the function field analog}, J.~Number Theory 218 (2021) 62--151.


\bibitem[Hay74]{Hay74}
Hayes, D.~R., {\it Explicit class field theory for rational function fields},
Transactions of the American Mathematical Society vol.~189 (1974) 77--91.


\bibitem[KaYo08]{KaYo08}
Kashio, T.~and Yoshida, H., {\it On $p$-adic absolute CM periods I},
Amer.~J.~Math.~130 no.~6 (2008) 1629--1685.

\bibitem[KaYo09]{KaYo09}
Kashio, T.~and Yoshida, H., {\it On $p$-adic absolute CM periods II},
Publ.~Res.~Inst.~Math.~Sci.~45 no.~1 (2009) 187--225.








\bibitem[Lam78]{Lam78}
Lam, T.~Y., {\it Serre's conjecture}, Lecture Notes in
Mathematics 635, Springer-Verlag Berlin Heidelberg New York 1978.




\bibitem[Milne86]{Milne} Milne, J.S., \textit{Abelian Varieties},
  chapter VII in \textit{Arithmetic Geometry}, G.~Cornell and J.~H.~Silverman, eds., Springer-Verlag, New York, 1986.


  \bibitem[NP24]{NP24} Namoijam, C.~and Papanikolas, M.~A., \textit{Hyperderivatives of periods and quasi-periods for Anderson $t$-modules}, Mem.~Amer.~Math.~Soc.~\textbf{302} (2024), no.~1517.



\bibitem[Og90]{Og90}
Ogus, A., {\it A $p$-adic analogue of the Chowla--Selberg formula}, in ``{\it $p$-adic analysis}'', Lecture notes in mathematics 1454, Springer-Verlag Berlin Heidelberg (1990) 319--341.



\bibitem[PT06]{PT06}
Pink, R.~and Tarulsen, M., {\it The isogeny conjecture for $t$-motives associated to direct sums of Drinfeld modules}, J.~number theory, Volume 117 (2006) 355--375.


\bibitem[Pel09]{Pel09} Pelzer, A., {\it Der Hauptsatz der Komplexen Multiplikation f\"ur Anderson $A$-motive}, Diploma thesis, University of Muenster, 2009.


\bibitem[Qui76]{Qui76}
Quillen, D., {\it Projective modules over polynomial rings}, Inventiones math.~36 (1976) 167--171.

  

\bibitem[Ros02]{Ros02}
Rosen, M., {\it Number theory in function fields}, Springer-Verlag New York, 2002. 


\bibitem[Si97] {Sinha97} Sinha, S.~K., {\it Periods of $t$-motives and transcendence},
Duke Math.~J.~\textbf{88} (1997), 465--535.



\bibitem[Sus76]{Sus76}
Suslin, A.~A., {\it Projective modules over polynomial rings are free}, Soviet Mathematics 17(4) (1976) 1160--1164.





\bibitem[Tael09]{Tael09}
Taelman, L., {\it Artin $t$-motifs},
J.~number theory, Volume 129, Issue 1
(2009) 142--157.

\bibitem[Tha91]{Tha91}
Thakur, D.~S., {\it Gamma functions for function fields and Drinfeld modules},
Ann.~of Math.~(2) 134.1 (1991) 25--64.

\bibitem[Wei26]{Wei26}
Wei, F.-T., {\it Algebraic relations among special gamma values and the Chowla-Selberg phenomenon over function fields}, Inventiones~math.~vol.~244 (2026) 815--897.


\bibitem[Win74]{Win74}
Winter, D.~J., {\it The Structure of Fields}, Springer-Verlag, New York, 1974.




\bibitem[Yang10]{Yang10}
Yang, T., \textit{The Chowla--Selberg formula and the Colmez
conjecture}, Canadian J.~Math.~62 (2010) 456--472.



\bibitem[Yo03]{Yo03}
Yoshida, H., {\it Absolute CM-periods},
Mathematical surverys and monographs vol.~106 (2003).





\bibitem [Yu90]{Yu90} Yu, J., \textit{On periods and quasi-periods of
Drinfeld modules}, Compositio Math. \textbf{74} (1990), 235--245.


\end{thebibliography}

\end{document}